\documentclass[a4paper,11pt,reqno]{amsart} 
\usepackage{amsmath}
\usepackage{amsfonts}
\usepackage[applemac]{inputenc}
\usepackage[T1]{fontenc}
\usepackage[colorlinks=true]{hyperref}
\usepackage{hyperref}
\usepackage[mathscr]{euscript}
\usepackage{verbatim,color,lipsum}
\usepackage{amssymb,latexsym}
\usepackage{amsthm}
\usepackage{color}

\usepackage[lmargin=2.5 cm,rmargin=2.5 cm,tmargin=3.5cm,bmargin=2.5cm,paper=a4paper]{geometry}

\newtheorem{thm}{Theorem}[section]
\newtheorem{pro}[thm]{Proposition}

\newtheorem{lem}[thm]{Lemma}
\newtheorem{cor}[thm]{Corollary}

\theoremstyle{remark}
\newtheorem{rem}[thm]{Remark}
\newcommand{\nb}{\nabla}
\newcommand{\ld}{\lambda}
\newcommand{\Int}[2]{\displaystyle{\int_{#1}^{#2}}}

\newcommand{\Bb}{\mathbf {B}}
\newcommand{\R}{\mathbb{R}}
\newcommand{\N}{\mathbb{N}}
\newcommand{\Sum}[2]{\displaystyle{\sum_{#1}^{#2}}}
\newcommand{\norm}[1]{\left\|#1\right\|}

\makeatletter
\g@addto@macro{\endabstract}{\@setabstract}
\newcommand{\authorfootnotes}{\renewcommand\thefootnote{\@fnsymbol\c@footnote}}%
\makeatother
\def\sig#1{\vbox{\hsize=5.5cm
\kern2cm\hrule\kern1ex
\hbox to \hsize{\strut\hfil #1 \hfil}}}
\newcommand\signatures[4]{%
\vspace{1cm}
\hbox to \hsize{\hfil #1\today\hfil}
\vspace{1cm}
\hbox to \hsize{\quad#2\hfil\hfil #3\quad}
\vspace{1cm}
\hbox to \hsize{\hfil#4\hfil}}
\numberwithin{equation}{section}

\title{Energy of surface states for 3D magnetic Schr\"{o}dinger operators}
\begin{document}
\maketitle
\begin{center}
 \par
  \authorfootnotes
 Marwa Nasrallah\textsuperscript{a,b}
 \scriptsize
 \par \bigskip
 \textsuperscript{a}{\it Department of Mathematics, Aarhus University, Denmark}\par
 \textsuperscript{b}{\it Ecole Doctorale Sciences Et Technologie, Hadath, Lebanon}\par
  \email{\it Email: marwa@imf.au.dk}
\end{center}

\begin{abstract}
We establish a semi-classical formula for the sum of eigenvalues of
a magnetic Schr\"{o}dinger operator in a three-dimensional domain
with compact smooth boundary and Neumann boundary conditions. The
eigenvalues we consider have eigenfunctions localized near the
boundary of the domain, hence they correspond to surface states.
Using relevant coordinates that straighten out the boundary, the
leading order term of the energy is described in terms of the
eigenvalues of model operators in the half-axis and the half-plane.
\end{abstract}
 \section{Introduction and main result}
\subsection{Introduction}

The computation of the number and the sum of eigenvalues of
Schr\"odinger operators in various asymptotic regimes is a central
question in mathematical physics. One motivation comes from the
problem of stability of matter (see Lieb-Solovej-Yngvason
\cite{LSY}). Another motivation is  the calculation of the quantum
current (see Fournais \cite{Fo}). The object of study in \cite{LSY}
is mainly the Pauli operator with magnetic field and electric
potential. The study of the finiteness of the number and the energy
of {\it negative} eigenvalues of the Pauli operator has been the
object of study of numerous papers, starting probably with the
establishing of the Cwickle-Rozenblum-Lieb
and Lieb-Thirring bounds, and followed up by many important papers
such as \cite{LSY, Er-So,Sob,FrEk}.

This paper aims at answering the same question as in \cite{LSY}  but
for the Schr\"odinger operator with magnetic field. The electric
potential is removed but the operator is defined in a domain with
boundary. This leads to a similar situation as in \cite{LSY}, but
the geometry of the boundary will have a significant influence on
the expression of the leading order terms (see
Theorem~\ref{main-thm} below). Details will be discussed at a later
point of this introduction.

The similar analogy between the results of this paper and those of
\cite{LSY} has been observed previously. In \cite{HeMo1}, while
estimating the ground state energy of a Schr\"odinger operator in a
domain with boundary, Helffer-Mohamed observed an analogy between
the semi-classical analysis of Schr\"odinger operators with electric
potentials and that of Schr\"odinger operators in domains with
boundaries. Loosely speaking, this analogy can be summarized by
saying that `boundaries' play a similar role to `electric
potentials'. More precisely, this analogy is established in
\cite{HeMo1} for the question of computing the ground state energy
for an operator in a domain with boundary. Guided by this analogy,
several important applications to the analysis of the
Ginzburg-landau model of superconductivity are given. We refer the
reader to the monograph \cite{FH-b} and references therein.

It is natural to wonder whether the same type of analogy between
`boundaries' and `electric potentials' still exists for the question
of computing the energy, as done in \cite{LSY}. The paper of
Fournais-Kachmar \cite{Fo-Ka} established this type of analogy
between boundaries and electric potentials  for two dimensional
domains and Neumann boundary condition. The goal of this paper is
to generalize the results of \cite{Fo-Ka} to the case of three
dimensional {\it smooth} domains.

\subsection{Earlier results}

Let $d\in\{2,3\}$ and $\mathcal{O}\subset\R^{d}$ be a bounded domain
with {\bf compact} and {\bf smooth} boundary  $\partial\mathcal{O}$.
We will consider both the
 case of interior domains
$\Omega=\mathcal{O}$ and exterior domains
$\Omega=\R^{d}\setminus\overline{\mathcal{O}}$.

We consider a magnetic vector potential ${\bf A}\in
C^{\infty}(\overline{\Omega};\R^{d})$ and introduce the magnetic
field ${\bf B}={\rm curl }~{\bf A}$ and the quantities
\begin{equation}\label{intensity+min}
 b:=\inf_{x\in\overline\Omega}{|{\bf B}(x)|},\qquad b^{\prime}=\inf_{x\in\partial\Omega}{|{\bf B}(x)|}\,.
\end{equation}
We assume that $b> 0$.
For $h>0$, we introduce the Neumann Schr\"{o}dinger operator $\mathcal{P}_{h}$ with magnetic field~:
\begin{equation}\label{main-ope}
 \mathcal{P}_{h}=(-ih\nabla +{\bf A})^{2}\quad {\rm in}\quad L^{2}(\Omega),
\end{equation}
whose domain is,
\begin{multline*}
    \mathcal{D}(\mathcal{P}_{h})=\big\{u\in L^{2}(\Omega)~:~(-ih\nb+{\bf{A}})^{j}u\in L^{2}(\Omega),\quad j=1,2,\quad\nu~.~(-ih\nb+{\bf  A})u=0
    \quad\text{   on   }\partial\Omega\big\}\,.
\end{multline*}
Here, for $x\in\partial\Omega$, $\nu(x)$ denotes the unit interior normal vector to $\partial\Omega$ at $x$.

The operator \eqref{main-ope} is the Friedrichs' self-adjoint
extension in $L^{2}(\Omega)$ associated with the semi-bounded closed
quadratic form~:
\begin{equation}\label{main-qf}
\mathcal{Q}_{h}(u):= \int_{\Omega}|(-ih\nb +{\bf A})u|^{2}dx,\quad \mathcal{D}(\mathcal{Q}_{h}):= \{u\in L^{2}(\Omega)\,:\,(-ih\nb +{\bf A})u\in L^{2}(\Omega)\}.
\end{equation}
If the domain $\Omega$ is bounded (interior case), it results from the compact embedding of $\mathcal{D}(\mathcal{Q}_{h})$ into $L^{2}(\Omega)$ that $\mathcal{P}_{h}$ has compact resolvent. Hence the spectrum is purely discrete consisting of a sequence of positive eigenvalues accumulating at infinity.

In the case of exterior domains, the operator ${\mathcal{P}}_{h}$
can have essential spectrum. It was established in
\cite[Theorem~3.1]{HeMo1} that there exists a constant $C_d\geq0$
such that for all $h\in(0,h_{0}]$, we have
\begin{equation}\label{est-HM}
 \int_{\Omega}|(-ih\nb +{\bf A})u|^{2}dx\geq h\int_{\Omega}(|{\bf B}(x)|-C_dh^{1/4})|u(x)|^{2}, \quad \forall\,u\in C^{\infty}_{0}(\Omega).
\end{equation}
Notice that in the two dimensional case where $d=2$, we have that
$C_d=0$. Using  a magnetic version of Persson's Lemma (see
\cite{Bo1,Per}), we get that
\[
\inf{\rm Spec}_{\rm ess}\mathcal{P}_{h}\geq h(b-C_dh^{1/4}).
\]

If the magnetic field is constant and the domain $\Omega$ has a
smooth boundary, it is established that:
\begin{equation}
\inf{\rm Spec}~\mathcal{P}_{h}=h\Theta_0b+o(h)\,,\quad (h\to 0_+)\,,
\end{equation}
where $\Theta_{0}\in (0,1)$ is the universal constant introduced in
\eqref{theta0'}. In such a situation, we see that if
$\Lambda\in[0,b)$, then the set
$${\rm Spec}~\mathcal P_h\,\cap\,[0,\Lambda h)\not=\emptyset\,.$$

In general, we consider $\Lambda\in[0,b)$ and  work under the
assumption that ${\rm Spec}~\mathcal P_h\,\cap\,[0,\Lambda
h)\not=\emptyset$ and denote the elements of this set as an
increasing sequence of eigenvalues counting multiplicities,
$${\rm Spec}(\mathcal P_h)\,\cap \,(-\infty\,,\Lambda h)=\{e_1(h)\,,e_2(h)\,,\cdots\}\,.$$

In \cite{Fo-Ka}, it is established the asymptotic behavior of the
sum
\[
 \sum_{j}(e_{j}(h)-\Lambda h)_{-}:= {\rm Tr}\big(\mathcal{P}_{h}-\Lambda h\big)_{-}
\]
in the semi-classical limit $h\rightarrow 0$. Here
$(x)_{-}=\max(-x,0)$ denotes the negative part of a number $x\in\R$,
and, for a self-adjoint operator $\mathcal{H}$, the operator
$\mathcal{H}_{-}=-\mathbf 1_{(-\infty,0)}(\mathcal H)\, \mathcal H $
is defined via the spectral theorem.

The result of \cite{Fo-Ka}, valid in two dimensions $(d=2)$,  is
recalled in Theorem~\ref{FK-th} below. In the statement of the
theorem, notice that, if $\xi\in\R$, the number $\mu_{1}(\xi)$ is
the lowest eigenvalue of the harmonic oscillator
\[
-\partial_{t}^{2}+(t-\xi)^{2}, \qquad{\rm in}\quad L^{2}(\R_{+}),
\]
and $\Theta_{0}$ is the universal constant defined as follows
\begin{equation}\label{theta0'}
\Theta_0=\inf_{\xi\in\R}\mu_1(\xi)\,.\end{equation}

\begin{thm} \label{FK-th}(Fournais-Kachmar \cite{Fo-Ka}; $d=2$).  As $h\to0_+$, there holds,
\begin{equation}\label{dim-2-N}
\lim_{h\rightarrow 0}h^{-1/2}\Sum{j}{} (e_{j}(h)-bh)_{-}= \dfrac{1}{2\pi}\iint_{\partial\Omega \times\R}
{ B}(x)^{3/2}\left(-\dfrac{b}{{ B}(x)}+\mu_{1}(\xi)\right)_{-}ds(x) d\xi,
\end{equation}
where $ds(x)$ denotes the arc-length measure on the boundary.
\end{thm}

\subsection{Main results}

We focus on the case where $\Omega\subset \R^3$ ($d=3$). The main
result of this paper is  a generalization of Theorem~\ref{FK-th}
valid when $d=3$.

We shall need the following notation
\begin{itemize}
\item Given $\eta\in\R$, $\omega$ an open domain in $\R^{3}$ and a self-adjoint operator $\mathcal{H}$ in $L^{2}(\omega)$ such that the spectrum below $\eta$ is discrete, we shall denote by
\begin{equation}\label{def-number}
\mathcal{N}(\eta;\mathcal{H},\omega):={\rm Tr}({\bf{1}}_{(-\infty,\eta]} (\mathcal{H}))
\end{equation}
 the number of eigenvalues less than $\eta$, counting multiplicities, and by
\begin{equation}\label{def-energy}
\mathcal{E}(\eta; \mathcal{H},\omega):= {\rm Tr}\big(\mathcal{H}-\eta\big)_{-}
\end{equation}
their corresponding sum below $\eta$.
\item If $x$ is a point on the boundary of $\Omega$, then $\theta(x)$ denotes the angle in $[0,\pi/2]$ between
 the magnetic field $B= \mbox{curl}\,{\bf {A}}$ and the tangent plane to $\partial\Omega$ at the point $x$.  More precisely,
\begin{equation}\label{def-theta}
     \partial\Omega\ni x\mapsto  \theta(x)=\arcsin\bigg(\dfrac{|{\bf B}(x)\cdot\nu(x)|}{|{\bf B}(x)|}\bigg)\in [0,{\pi}/{2}].
\end{equation}
\item We let $\R_{+}=(0,\infty)$, $\R^{2}_{+}=\R\times(0,\infty)$ and $\R^{3}_{+}=\R^{2}\times(0,\infty)$.
 \item For $\xi\in\R$, we denote by $\mu_{1}(\xi)$ the lowest eigenvalue of the harmonic oscillator
 \[
 -\partial_{t}^{2}+(t-\xi)^{2}\qquad {\rm in}\quad L^{2}(\R_{+})
 \]
 with Neumann boundary conditions at $t=0$.
 \item For $\theta\in(0,\pi/2]$, we introduce the two-dimensional operator
 \begin{equation}\label{op-L-th}
 \mathcal{L}(\theta):=-\partial_{t}^{2}-\partial_{s}^{2}+(t\cos(\theta)-s\sin(\theta))^{2} \quad{\rm in}\quad L^{2}(\R^{2}_{+}).
 \end{equation}
 It is well known (see~\cite{HeMo4}) that the essential spectrum of $\mathcal{L}(\theta)$ is the interval $[1,\infty)$, and we shall denote by $\{\zeta_{j}(\theta)\}_{j}$ the countable set of eigenvalues of $\mathcal{L}(\theta)$ in the interval $[\zeta_{1}(\theta),1)$.
 \item We define the positive and negative parts of a real number $x$ by $(x)_{\pm}=\max({\pm x,0})$.
\end{itemize}
The main theorem of this paper is~:
\begin{thm}\label{main-thm}
Suppose $\Omega\subset\R^3$ is either an interior or an exterior
domain with compact smooth boundary
 $\partial\Omega$. Given $\Lambda\in[0,b)$, the
 following asymptotic formula holds,
\begin{equation}\label{main-for}
  \sum_{j}(e_{j}(h)-\Lambda h)_{-}={\rm Tr}\big(\mathcal{P}_{h}-\Lambda h\big)_{-}= \int_{\partial\Omega}
 |{\bf B}(x)|^{2} E(\theta(x),\Lambda |{\bf B}(x)|^{-1})d\sigma(x)+ o(1)
  ,\qquad{\rm as}
  \quad h\rightarrow 0,
\end{equation}
where, for $\lambda\in[0,1)$, the function $E(\theta,\lambda)$ is given by
\begin{equation*}
E(\theta,\lambda)=\left\{\begin{array}{lcl}
\dfrac{1}{3\pi^{2}}\displaystyle\int_{0}^{\infty}(\mu_{1}(\xi)-\lambda)_{-}^{3/2}d\xi&{\rm if}&\theta=0,\\
\label{E2}\dfrac{\sin(\theta)}{2\pi}\displaystyle\sum_{j}(\zeta_{j}(\theta)-\lambda )_{-}&{\rm if}&\theta\in(0,\pi/2],
\end{array}\right.
\end{equation*}
 and $d\sigma(x)$ denotes the surface  measure on the boundary $\partial\Omega$.
\end{thm}
\begin{rem}
In the case $\theta=\frac{\pi}{2}$, it is well known (see\,\cite{HeMo3}) that the first eigenvalue $\zeta_{1}(\frac{\pi}{2})=1$ which implies that $E(\frac{\pi}{2},\lambda)=0$ for any $\lambda\in[0,1)$.
\end{rem}
\begin{rem}
In the case $\theta\in(0,\pi/2)$, we emphasize that the sum
appearing in the formula of $E(\theta,\lambda)$ above, is a finite
sum. Indeed, in view of Lemma~\ref{Nb-Ev-Finite} below, we learn
that the number of eigenvalues of $\mathcal{L}(\theta)$, below a
fixed $\lambda\in[0,1)$, is finite.
\end{rem}

\begin{rem}
In Lemma~\ref{continuity+} below, we show that the function
\[
(\theta,\lambda)\mapsto E(\theta,\lambda),
\]
is a continuous function as a function of two variables. Consequently, we obtain
\[
\lim_{\theta\rightarrow0}\dfrac{\sin(\theta)}{2\pi}\displaystyle\sum_{j}(\zeta_{j}(\theta)-\lambda )_{-}= \dfrac{1}{3\pi^{2}}\int_{0}^{\infty}(\mu_{1}(\xi)-\lambda)_{-}^{3/2}d\xi.
\]
Notice that this formula is connected to the  formula for the number of eigenvalues given in \cite{MoTr}:
 \[
\lim_{\theta\rightarrow0}\big(\sin\theta\,\mathcal N(\lambda,\theta)\big)= \dfrac{1}{\pi}\int_{0}^{\infty}(\mu_{1}(\xi)-\lambda)_{-}^{3/2}d\xi,
\]
where $\mathcal N(\lambda,\theta)={\rm Card}\{\zeta_j(\theta)~:~\zeta_j(\theta)\leq\lambda\}$\,.
\end{rem}

Using the technique  to go from energies to densities
(see~\cite{EvLiSiSo} for details), we can differentiate both sides
of \eqref{main-for} with respect to $\Lambda h$ and get an
asymptotic formula for the number of eigenvalues of
$\mathcal{P}_{h}$ below $\Lambda h$. This is stated in the next
corollary valid under the assumptions made in
Theorem~\ref{main-thm}.
\begin{cor}\label{Cor-nb}
Let $\Lambda\in[0,b)$ and $\sigma$ be the surface measure on
$\partial\Omega$. If,
\begin{equation}\label{asp-nb}
\sigma\big(\big\{x\in\partial\Omega~:~\theta(x)\in(0,\pi/2),\quad {\Lambda}{|{\bf B}(x)|}^{-1}\in  {\rm Spec }\,\mathcal{L}(\theta(x))\big\}\big)=0\,,
\end{equation}
then, the following asymptotic formula holds true as $h\to0_+$,
\begin{equation}
\lim_{h\rightarrow 0}h \mathcal{N}(\Lambda h;\mathcal{P}_{h},\Omega)
=\int_{\partial\Omega}|{\bf B}(x)|n(\theta(x),\Lambda|{\bf B}(x)|^{-1})d\sigma(x)\,,
\end{equation}
where, for $\lambda\in[0,1)$, $n(\theta,\lambda)$ is given by
\begin{equation*}
n(\theta,\lambda)=\left\{\begin{array}{ll}
\dfrac{1}{2\pi^{2}}\displaystyle\int^{\infty}_{0}(\mu_{1}(\xi)-\lambda)_{-}^{1/2}d\xi &{\rm if}\quad\theta=0,\\
\dfrac{\sin(\theta)}{2\pi}\mathcal{N}(\lambda,\theta)&{\rm if}\quad\theta\in(0,\pi/2].
\end{array}\right.
\end{equation*}
\end{cor}

The result of Corollary~\ref{Cor-nb} is a generalization of the
asymptotic formula given in \cite{Fr} for the number of {\it edge
states} in two dimensional domains. However, it is worthy to notice
that the result in two dimensions is valid without the geometric
condition in Corollary~\ref{Cor-nb}.

The geometric condition in \eqref{asp-nb} is satisfied when $\Omega$
is the unit ball, the magnetic field $\Bb$ is constant of unit
length and $\Lambda$ is  sufficiently close to the universal
constant $\Theta_0$. As we shall see, this is closely related to the
behavior of the functions $(0,\pi/2)\ni
\theta\mapsto\zeta_{j}(\theta)$.

In this concern, we recall the following two results.
\begin{lem}\label{thm-HM-zeta}\cite{HeMo4}
The functions $\theta\mapsto \zeta_{j}(\theta)$ are increasing and
continuous on $(0,\pi/2)$. Moreover,
$$\zeta_1(0)=\Theta_0\quad{\rm and}\quad \forall~\theta\in[0,\pi/2)\,,\quad\zeta_1(\theta)<1\,.$$
\end{lem}\label{lem:Po}
The second Lemma is taken from \cite{Po}.
\begin{lem}
Let $N\geq 1$ be an integer and suppose that there exists
$\theta_{\ast}\in(0,\pi/2)$ such that the following assumptions are
satisfied
\begin{enumerate}
\item $\zeta_{N}(\theta_{\ast})<1$;
\item The eigenvalues $\{\zeta_{j}(\theta_{\ast})\}_{1\leq n\leq N}$ are simple.
\end{enumerate}
Define
\[
\theta_{\max,N}:= \sup\{\theta\in(0,\pi/2],~\zeta_{N}(\theta)<1\}.
\]
Then for all $1\leq j\leq N$, the functions $\theta\mapsto
\zeta_{j}(\theta)$ are strictly increasing on $(0,\theta_{\rm
max,N})$.
\end{lem}

It is pointed in \cite{Po} that, to each $N$, there is $\theta_\ast$
such that the two conditions of Lemmas~\ref{lem:Po} are satisfied.
Thus, for every $N$, the conclusion of Lemma~\ref{lem:Po} is true.
In particular, when $N=2$ we get,
$$\delta=\min\left(\frac{\zeta_2(0)-\zeta_1(0)}{2}\,,\,\frac{1-\Theta_0}2\right)>0\,.$$
By continuity of the functions $\zeta_1(\theta)$ and
$\zeta_2(\theta)$,  there exists $\epsilon_{0}\in(0,\theta_{\max,2}
)$ such that for all $\theta\in [0,\epsilon_{0}]$,
\[
\zeta_{2}(\theta)\geq \zeta_1(\theta)+\delta\geq \Theta_{0}+\delta.
\]
 Take $\Lambda\in(\Theta_0,\Theta_0+\delta)$. That way we get that
 \[
 \forall~\theta\in[0,\pi/2]\,,\quad \zeta_{2}(\theta)\geq\Theta_0+\delta>\Lambda\,,
 \]
 and $\zeta_{1}(\theta)=\Lambda$ has at most one solution in $[0,\theta_{\rm max,1}]$. Notice here that $\theta_{\max,1}=\pi/2$ is a consequence of Lemma~\ref{thm-HM-zeta}.

Returning back   the condition \eqref{asp-nb} and the above
discussion, we see that when the magnetic field is constant of unit
length and the domain $\Omega$ is the unit ball, the set
$$\Sigma=\{x\in\partial\Omega~:~\theta(x)\in(0,\pi/2),\quad {\Lambda}{|{\bf B}(x)|}^{-1}\in  {\rm Spec }\,\mathcal{L}(\theta(x))\big\}$$
consists of at most one circle (defined by the solution $\theta$ of
$\zeta_1(\theta)=\Lambda$). That way the set $\Sigma$ has measure
zero relative to the surface measure and the condition
\eqref{asp-nb} is satisfied.

\subsection{Perspectives} We list some natural questions for future
research:
\begin{enumerate}
\item Inspection of the number $\mathcal{N}(\Lambda h;\mathcal{P}_{h},\Omega)$ when the condition in \eqref{asp-nb} is violated.
\item Theorem~\ref{main-thm} is established when the domains $\Omega$ has a smooth boundary. An interesting question is to study the case when the domain $\Omega$ has corners or wedges (see \cite{Po}). In two dimensions, this is done in \cite{KKH}.
\item The inspection of the effect of the boundary conditions might be interesting. Theorem~\ref{main-thm} is established for the operator with Neumann boundary condition. A natural question is to consider the operator with  Robin boundary condition
\[
\nu~.~(-ih\nb+{\bf  A})u+\gamma u=0\quad \mbox{  on  }\quad \partial\Omega,
\]
where $\gamma\in L^{\infty}(\partial\Omega;\R)$ (see
\cite{Ka4}).
\item The asymptotic formula in Theorem~\ref{main-thm} holds for the energy of the eigenvalues below the energy level $\Lambda h$ with $\Lambda<b$.
However, in two dimensions, such a restriction on $\Lambda$ does
not appear ($\Lambda$ is allowed to be $b$).  Removing the
restriction on $\Lambda$ in three dimensions is an interesting
question.
\end{enumerate}

\subsection{Organization of the paper}

Compared to the situation in two dimensional domains, the analysis
of the problem  in three dimensional domains needs new ingredients.
The reason is that the boundary in $3$D is a surface and has richer
geometry than that in $2$D.

The paper is organized as follows. Section $2$ contains some basic tools concerning the variational principle of the sum of negative eigenvalues. Section $3$ is devoted to the spectral analysis of the model operator on a half-cylinder with Neumann boundary condition on one edge and Dirichlet boundary conditions on the other edges. Section $4$ is devoted to the construction of the function  $E(\theta,\lambda)$ as the limit of the energy of the operator in the half-cylinder. Continuity properties of this function are studied in Section $5$. Explicit formulas of $E(\theta,\lambda)$ are established in Section $6$. 
Section $7$ contains the expression of the operator relative to local coordinates near the boundary of the domain $\Omega$.
Section $8$ concludes with the proof of Theorem~\ref{main-thm} and Corollary~\ref{Cor-nb}.

\section{Variational principle}
In this section, we recall useful methods to establish upper and
lower bounds on the energy of eigenvalues (see for example
\cite{We-St}).
\begin{lem}{\label{lem-VP-2}}
Let $\mathcal{H}$ be a semi-bounded self-adjoint operator on $L^{2}(\R^{3})$ satisfying
\begin{equation}\label{hypI}
\inf{\rm Spec}_{\rm ess}(\mathcal{H})\geq 0\,.
\end{equation}
Let $\{\nu_{j}\}_{j=1}^{\infty}$ be the sequence of negative
eigenvalues of $\mathcal H$ couting multiplicities. We have,
\begin{equation}\label{eq-var-2}
-\sum_{j=1}^{\infty}(\nu_{j})_{-}
 =\inf\Sum{j=1}{N}\big\langle \psi_{j},\mathcal{H}\psi_{j}\big\rangle,
\end{equation}
where the infimum is taken over all $N\in\mathbb{N}$ and orthonormal families $\{\psi_{1},\psi_{2},\cdots,\psi_{N}\}\subset D(H)$.
\end{lem}
The next lemma states another variational principle. It is used in several papers, e.g. \cite{LSY}.
\begin{lem}\label{lem-VP-3}
Let $\mathcal{H}$ be a self-adjoint semi-bounded operator satisfying the hypothesis \eqref{hypI}.
Suppose in addition that $(\mathcal{H})_{-}$ is trace class.
For any orthogonal projection $\gamma$ with range belonging to the domain of $\mathcal{H}$ and such that $\mathcal{H}\gamma$ is trace class, we have,
\begin{equation}\label{eq-var-2}
-\sum_{j=1}^{\infty}(\nu_{j})_{-} \leq {\rm Tr}(\mathcal{H\gamma})\,.
\end{equation}
\end{lem}

For later purposes, we include
\begin{cor}\label{NB-vs-EN}
Let $\Omega$ be a subset of $\R^{3}$. Suppose that $P$ is a positive
self-adjoint operator on $L^{2}(\Omega)$ such that its spectrum
below $1$ is discrete. Let $\lambda\in[0,1)$ and $\varsigma\in\R$
such that $-\lambda \leq \varsigma<1-\lambda$. We have
\begin{equation}\label{N-E}
  \mathcal{E}(\lambda+\varsigma; P, \Omega)\leq\mathcal{E}(\lambda; P, \Omega)+ \varsigma \mathcal{N}(\lambda+\varsigma;P,\Omega),
\end{equation}
where $\mathcal{N}(\cdot; P, \Omega)$  and  $\mathcal{E}(\cdot; P,\Omega)$ are introduced in \eqref{def-number} and \eqref{def-energy} respectively.
\end{cor}
\begin{proof}
Let $\{\lambda_{k}\}_{k=1}^{N}$ be the family of eigenvalues below $\lambda+\varsigma$ for $P$ and $\{g_{k}\}_{k=1}^{N}$ are associated (normalized) eigenfunctions . Let us define the trial density matrix $\gamma: L^{2}(\Omega)\ni f \mapsto\gamma f \in L^{2}(\Omega)$,
\[
\gamma f= \Sum{1\leq k\leq N}{} \langle f, g_{k} \rangle g_{k}.
\]
which satisfies $0\leq \gamma\leq 1$ (in the sense of quadratic forms).
By Lemma~\ref{lem-VP-3}, it follows that
\begin{equation}\label{use-lem}
-\mathcal{E}(\lambda; P, \Omega):=-{\rm Tr}\big(P-\lambda\big)_{-}\leq {\rm Tr} \big((P-\lambda)\gamma\big).
\end{equation}
On the other hand, we have
  \begin{align}
  \nonumber {\rm Tr} \big((P-\lambda)\gamma\big) =\Sum{{1\leq k\leq N} }{}(\lambda_{k}-\lambda)
&=\Sum{{1\leq k\leq N}}{}(\lambda_{k}-\lambda-\varsigma)+
 \varsigma\Sum{{1\leq k\leq N}}{} 1\\
\label{tr+dec} &=-\mathcal{E}(\lambda+\varsigma; P, \Omega)+ \varsigma\mathcal{N}(\lambda+\varsigma;P,\Omega).
  \end{align}
 Inserting this into \eqref{use-lem} yields \eqref{N-E}.
 %
\end{proof}
\section{Model operator in the half-space}
Our main goal in this section is to establish an upper bound on the number and the sum of eigenvalues below the infimum of the essential spectrum of a magnetic Schr\"{o}dinger operator in a half-cylinder in terms of the area of the cylinder base.
  \subsection{Reflection with respect to the boundary}
In order to state Lemma \ref{ref-lem} below, we need to define the reflected magnetic Schr\"odinger operator in $L^{2}(\R^{3})$
associated with the Neumann Schr\"odinger operator in $L^{2}(\R^{3}_{+})$.

Given $\theta\in[0,\pi/2]$, we consider the magnetic potential
\begin{equation}\label{MF-Ftilde}
\widetilde{\bf F}_{\theta}(r,s,t):= (0,0,|t|\cos(\theta)- s \sin(\theta) ), \qquad (r,s,t)\in\R^{3}.
\end{equation}
Let
\begin{equation}
  \widetilde{\mathscr{P}}_{\theta}=(-i\nb+\widetilde{\bf F}_{\theta})^{2} \qquad{\rm in}\qquad L^{2}(\R^{3}),
\end{equation}
be the self-adjoint operator defined  by the closed quadratic form
\begin{equation}
  \widetilde{\mathcal Q}_{\theta}(u):= \int_{\R^{3}} |(-i\nb+\widetilde{\bf F}_{\theta})u|^{2}drdsdt,\quad
  \mathcal{D}(\widetilde{\mathcal Q}_{\theta}):=\big\{u\in L^{2}(\R^{3})~:~ (-i\nb+\widetilde{\bf F}_{\theta})u\in L^{2}(\R^{3})\big\}\,.
  \end{equation}
We let $\beta=(0,\cos(\theta),\sin(\theta))$ denote the constant magnetic field generated by the vector potential~:
\begin{equation}\label{Ftheta}
{\bf F}_{\theta}(r,s,t)= (0,0,t\cos(\theta)-s\sin(\theta)), \qquad (r,s,t)\in\R^{3}_{+}.
\end{equation}
Furthermore, let
\begin{equation}\label{Neum-r3+}
\mathscr{P}^{N}_{\theta}= (-i\nb+{\bf F}_{\theta})^{2} \qquad{\rm in}\qquad L^{2}(\R^{3}_{+}),
\end{equation}
be the self-adjoint (Neumann) Schr\"{o}dinger operator associated with the quadratic form
\begin{equation}
  \mathcal{Q}^{N}_{\theta}(u):= \int_{\R^{3}_{+}} |(-i\nb+{\bf F}_{\theta})u|^{2}drdsdt,  \quad  \mathcal{D}({\mathcal Q}^{N}_{\theta}):= \left\{u\in L^{2}(\R^{3}_{+})~:~ (-i\nb+{\bf F}_{\theta})u\in L^{2}(\R^{3}_{+})\right\}.
  \end{equation}
We establish in the next lemma estimates on the eigenvalue counting function and the energy of eigenvalues for a perturbation of $\mathscr{P}_{\theta}^{N}$.
\begin{lem}\label{ref-lem}
Let $\mathscr{U}$ be a positive bounded potential in $L^{2}(\R^{3})$
verifying $\mathscr{U}(\cdot,\cdot,-t)=\mathscr{U}(\cdot,\cdot,t)$.
Assume that the spectrum of $\mathscr{P}_{\theta}^{N}+ \mathscr{U}$
below $\lambda$ is discrete.   We have
\begin{equation}\mathcal{N}(\lambda;{\mathscr{P}}^{N}_{\theta}+\mathscr{U}, \R^{3}_{+})\leq C_{\rm CLR}\int_{\R^{3}}(\mathscr{U})_{-}^{3/2}drdsdt,\quad \mathcal{E}(\lambda;{\mathscr{P}}^{N}_{\theta}+\mathscr{U},\R^{3}_{+})\leq C_{\rm LT}
\int_{\R^{3}}(\mathscr{U})_{-}^{5/2}drdsdt\,,
\end{equation}
where $C_{\rm CLR}$ and $C_{\rm LT}$ are two positive universal
constants.
\end{lem}
\begin{proof}
Let $n\in\mathbb{N}$. Let $\{u_{j}\}_{j=1}^{n}$ be an orthonormal family of eigenfunctions with corresponding eigenvalues $\{ \mu_{j}\}_{j=1}^{n}$ associated with the operator $\mathscr{P}_{\theta}^{N}+ \mathscr{U}$ in $L^{2}(\R^{3}_{+})$. We define the extension to $\R^{3}$ of the function $u_{j}$ by~:
\begin{equation}\label{def-u-tilde}
 \widetilde u_{j}(r,s,t)=\begin{cases}
 \frac{1}{\sqrt{2}}u_{j}(r,s,t)& t\geq 0\\
\frac{1}{\sqrt{2}}u_{j}(r,s,-t)&t<0.
 \end{cases}
\end{equation}
Since $\{u_{j}\}_{j=1}^{n}$ are normalized and pairwise orthogonal, we get for all $1\leq j,k\leq n$,
\begin{equation}\label{norm-u-tilde}
\langle \widetilde{u}_{j},\widetilde u_{k}\rangle_{L^{2}(\R^{3})}
=\langle u_{j},u_{k} \rangle_{L^{2}(\R^{3}_{+})}=\delta_{j,k},
\end{equation}
where $\delta_{j,k}$ is the Kronecker symbol.

The bilinear form associated with $\widetilde{\mathcal{Q}}_{\theta}+\mathscr{U}$ is defined on the form domain by:
\[
\widetilde a_{\theta,\mathscr{U}}(u,v)= \int_{\R^{3}}\left((-i\nb +\widetilde{\bf F}_{\theta})u \overline{(-i\nb +\widetilde{\bf F}_{\theta})v} +\mathscr{U}u\overline{v}\right)drdsdt.
\]
Here the magnetic field $\widetilde{\bf F}_{\theta}$ is the same as in \eqref{MF-Ftilde}.

 It is easy to see that the functions $\{{\widetilde u}_{j}\}_{j}$ belong to the form domain $\mathcal{D}(\widetilde{\mathcal{Q}}_{\theta})$, since, by construction, we have
\(
\widetilde{\mathcal{Q}}_{\theta}(\widetilde u_{j})= \mathcal{Q}^{N}_{\theta}( u_{j}).
\)
Since the potential $\mathscr{U}$ is symmetric in the $t$-variable, we obtain for all $1\leq j,k\leq n$,
\begin{equation}
\widetilde a_{\theta,\mathscr{U}}(\widetilde u_{j},\widetilde u_{k})= \langle u_{j}, (\mathscr{P}_{\theta}^{N}+\mathscr{U} )u_{k} \rangle_{L^{2}(\R^{3}_{+})}.
\end{equation}
Since the $\{u_{j}\}_{j=1}^{n}$ are eigenfunctions of $\mathscr{P}_{\theta}^{N}+\mathscr{U}$, we get using \eqref{norm-u-tilde},
\begin{align}\label{orth-f}
\widetilde a_{\theta,\mathscr{U}}(\widetilde u_{j},\widetilde u_{k})= \delta_{j,k} \mu_{k},\quad\forall~ 1\leq j,k\leq n.
\end{align}
Let $\widetilde\mu_{n}$ be the $n$-th eigenvalue of $\widetilde{\mathscr{P}}_{\theta}+\mathscr{U}$ defined by the min-max principle. Owing to \eqref{norm-u-tilde} and \eqref{orth-f} we find,
\begin{align*}
 \widetilde\mu_{n}=\inf_{ v_{1},\cdots v_{n} \in\mathcal D(\widetilde{\mathcal{Q}}_{\theta})} \max_{\stackrel {v\in{\rm span}[v_{1},\cdots,v_{n}]} {\norm{v}=1}}\widetilde a_{\theta,\mathscr{U}}(v,v)
\leq  \max_{v\in {\rm span}[\widetilde u_{1},\cdots,\widetilde u_{n}]}\widetilde a_{\theta,\mathscr{U}}(v,v) =  {\mu}_{n},
\end{align*}
This yields
\begin{equation}\label{Ext-R3-1}
\mathcal{N}(\lambda;{\mathscr{P}}^{N}_{\theta}+\mathscr{U},\R^{3}_{+})
\leq \mathcal{N}(\lambda;\widetilde{\mathscr{P}}_{\theta}+\mathscr{U},\R^{3}),
\end{equation}
and
\begin{equation}\label{Ext-R3-2}
 \mathcal{E}(\lambda;{\mathscr{P}}^{N}_{\theta}+\mathscr{U},\R^{3}_{+})
\leq \mathcal{E}(\lambda;\widetilde{\mathscr{P}}_{\theta}+\mathscr{U},\R^{3}).\end{equation}
The lemma follows by applying CLR inequality (resp. Lieb-Thirring inequality) to the right-hand side of \eqref{Ext-R3-1} (resp. \eqref{Ext-R3-2}).
 \end{proof}
 \subsection{Schr\"odinger operator in a half-cylinder}
Consider a positive real number $L$, and define the domain
\begin{equation}\label{OmegaL}
\Omega^{L}= \Big(-\frac{L}{2},\frac{L}{2}\Big)^{2}\times\R_{+}.
\end{equation}
In this section, we will analyse the magnetic Schr\"{o}dinger operator
\begin{equation}\label{PL1L2}
\mathscr{P}^{L}_{\theta}=(-i\nb+{\bf F}_{\theta})^{2}\quad{\rm in}\quad L^{2}(\Omega^{L})
\end{equation}
with Neumann boundary conditions at $t=0$, and Dirichlet boundary conditions at $r\in\{-\frac{L}{2},\frac{L}{2}\}$ and $s\in\{-\frac{L}{2},\frac{L}{2}\}$. Here, for $\theta\in[0,\pi/2]$, ${\bf F}_{\theta}$ is the magnetic potential introduced in \eqref{Ftheta}.

More precisely, the operator $\mathscr{P}^{L}_{\theta}$ is defined as the self-adjoint Friedrichs extension
 in $L^{2}(\Omega^{L})$ associated with the semi-bounded quadratic form
\begin{equation}\label{QL1L2}
  \mathcal{Q}_{\theta}^{L}(u)=\int_{\Omega^{L}}|(-i\nb +{\bf F}_{\theta})u|^{2}dr ds dt\,,
\end{equation}
defined for all functions $u$ in the form domain of $\mathcal{Q}_{\theta}^{L}$,
\begin{multline}
\mathcal{D}(\mathcal{Q}_{\theta}^{L})=\Big\{u\in L^{2}(\Omega^{L})\,:\, (-i\nb+{\bf F}_{\theta})u\in L^{2}(\Omega^{L}) ,
\quad u\Big(-\frac{L}{2},\cdot,\cdot\Big)=u\Big(\frac{L}{2},\cdot,\cdot\Big)=0,\\
u\Big(\cdot,-\dfrac{L}{2},\cdot\Big)=u\Big(\cdot,\frac{L}{2},\cdot\Big)=0 \Big\}.
\end{multline}
The next lemma establishes super-additivity properties on the sum of eigenvalues for $\mathscr{P}_{\theta}^{L}$.
\begin{lem}\label{monot}
For all $n\in\mathbb{N}$, $\ld\in[0,1)$ and $L>0$, we have,
\begin{equation}\label{En-inc}
\dfrac{\mathcal{E}(\ld;\mathscr{P}_{\theta}^{nL},\Omega^{nL})}{n^{2} L^{2}}\geq
 \dfrac{\mathcal{E}(\ld;\mathscr{P}_{\theta}^{L},\Omega^{L})}{L^{2}}.
\end{equation}
\end{lem}
\begin{proof}
Let $j,k\in \mathbb{N}$ such that $0\leq j,k\leq n-1$, we define the domain
$$\Omega_{j,k}^{L}:=\Big(\frac{(-n+2j)L}{2}, \frac{(-n+2j+2)L}{2}\Big)\times\Big(\frac{(-n+2k)L}{2},\frac{(-n+2k+2)L}{2}\Big)\times\R_{+}.$$
It is clear that $\Omega^{nL}=\cup_{j,k=0}^{n-1}\Omega_{j,k}^{L}.$
We next consider the (Friedrichs) self-adjoint operator
$\mathscr{P}_{\theta,j,k}^{L}$ defined by the closed quadratic form
\begin{equation}\label{QL1L2ij}
\mathcal{Q}_{\theta,j,k}^{L}(u)=\int_{\Omega_{j,k}^{L}}|(-i\nb+{\bf F}_{\theta})u|^{2} drdsdt,
\end{equation}
with domain,
\begin{multline}
\mathcal{D}(\mathcal{Q}_{\theta,j,k}^{L})=\Big\{u\in L^{2}(\Omega^{L})\,:\,(-i\nb+{\bf F}_{\theta})u\in L^{2}(\Omega_{j,k}^{L}) ,\\
u\Big(\frac{(-n+2j)L}{2},\cdot,\cdot\Big)=u\Big(\frac{(-n+2j+2)L}{2},\cdot,\cdot\Big)=0,\\
u\Big(\cdot,\frac{(-n+2k)L}{2},\cdot\Big)=u\Big(\cdot, \frac{(-n+2k+2)L}{2},\cdot\Big)=0 \Big\}.
\end{multline}
Taking boundary conditions into account, we observe that for all $u=\sum_{j,k}{} u_{j,k}\in \oplus_{j,k}\mathcal{D}(\mathcal{Q}^{L}_{\theta,j,k})$,
\[
  \mathcal{Q}_{\theta}^{nL}(u)= \Sum{j,k}{} \mathcal{Q}_{\theta,j,k}^{L}(u_{i,j}).
\]
This implies, that \(
\mathscr{P}^{nL}_{\theta}\leq {\oplus}_{j,k}\mathscr{P}^{L}_{\theta ,j,k}
\) (in the sense of quadratic forms). From the min-max principle, it follows easily that
\begin{equation}\label{Nb-omega-ij}
       \sum_{j,k}\mathcal{E}(\ld;\mathscr{P}_{\theta,j,k}^{L},\Omega_{j,k}^{L})
       \leq \mathcal{E}(\ld;\mathscr{P}_{\theta}^{nL},\Omega^{nL}),\qquad \forall~1\leq j,k\leq n.
\end{equation}
   Since the operator $\mathscr{P}_{\theta,j,k}^{L}$
   is unitarly equivalent to
$\mathscr{P}_{\theta}^{L}$ by magnetic translation invariance, \eqref{Nb-omega-ij} becomes,
\[
n^{2}\mathcal{E}(\ld;\mathscr{P}_{\theta}^{L},\Omega^{L})\leq\mathcal{E}
(\ld;\mathscr{P}_{\theta}^{nL},\Omega^{nL}).
\]
This gives \eqref{En-inc} upon dividing both sides by $n^2L^{2}$.
\end{proof}
We show in the next lemma a rough bound on the number and the sum of eigenvalues of $\mathscr{P}_{\theta}^{L}$ in terms of $L^{2}$.
\begin{lem}\label{u-b-N}
 Let $L>0$. There exists a constant $C$ such that for all $\lambda\in[0,1)$ and $\theta\in[0,\pi/2]$, it holds true that
\begin{equation}\label{bnd-nb}
\dfrac{\mathcal N(\lambda;\mathscr{P}_{\theta}^{L},\Omega^{L})}{ L^{2}}\leq \dfrac{C}{\sqrt{1-\lambda}},
\end{equation}
and
\begin{equation}\label{bnd-engy}
  \dfrac{\mathcal{E}(\lambda;\mathscr{P}_{\theta}^{L},\Omega^{L})}{ L^{2}}\leq \dfrac{C}{\sqrt{1-\lambda}},
\end{equation}
where $ \mathcal{N}(\ld;\mathscr{P}_{\theta}^{L},\Omega^{L})$ and $ \mathcal{E}(\ld;\mathscr{P}_{\theta}^{L},\Omega^{L})$ are defined in \eqref{def-number} and \eqref{def-energy} respectively.
\end{lem}
\begin{proof}
Let $(\psi_{1}(t),\psi_{2}(t))$ be a partition of unity on $\R_{+}$ with $\psi_{1}^{2}(t)+\psi_{2}^{2}(t)=1$ and:
\begin{equation}
\left\{\begin{array}{ll}
\psi_{1}(t)=1\quad {\rm if}\quad $0<t<1$,&\\
\psi_{1}(t)=0\quad {\rm if}\quad $t>2$.
\end{array}\right.
\end{equation}
Let $T>1$ be a large number to be chosen later. We consider the following two sets
\[
\Omega^{T}_{1}=\Big(-\frac{L}{2},\frac{L}{2}\Big)^{2}\times(0,2T), \qquad \Omega^{T}_{2}=\Big(-\frac{L}{2},\frac{L}{2}\Big)^{2}\times(T,\infty).
\]
We define the partition of unity $(\psi_{1,T}(t),\psi_{2,T}(t))$ by
\[
\psi_{1,T}(t)=\psi_{1}\left(\dfrac{t}{T}\right),\qquad \psi_{2,T}(t)=\psi_{2}\left(\dfrac{t}{T}\right)
\]
We have $\psi_{k,T}^{\prime}(t)=
\frac{1}{T}\psi_{k}^{\prime}({\frac{t}{T}})$. Thus we deduce that
there exists a constant $C_{0}>0$ such that
\begin{equation}\label{par-of-uni}
\sum_{k=1}^{2}|\psi_{k}^{\prime}(t)|^{2}\leq \dfrac{C_{0}}{T^{2}}.
\end{equation}
By the IMS formula and the fact that $\psi_{1,T}^{2}(t)+\psi_{2,T}^{2}(t)=1$, we find for all $u\in\mathcal{D}(\mathcal{Q}^{L}_{\theta})$
\begin{equation}\label{Ims}
\mathcal{Q}^{L}_{\theta}(u)=\sum_{k=1}^{2}\Big(\mathcal{Q}^{L}_{\theta}(\psi_{k,T}u)-\int_{\Omega^{L}}
\mathcal{U}_{T}(t)|\psi_{k,T}u|^{2}drdsdt\Big),\qquad \mathcal{U}_{T}(t)=\sum_{k=1}^{2}|\psi^{\prime}_{k,T}(t)|^{2}.
\end{equation}
Let us denote by $\mathscr{P}^{L}_{\theta,1}$ and $\mathscr{P}^{L}_{\theta,2}$ the self-adjoint operators
 associated with the following quadratic forms~:
\begin{equation}
 \mathcal{Q}^{L}_{\theta,1}(u)=\Int{\Omega_{1}^{T}}{}\Big{[}|\partial_{t}u|^{2}+|\partial_{s}u|^{2}+
  |(-i\partial_{r}+t\cos(\theta)-s\sin(\theta))u|^{2}\Big{]}drdsdt,
\end{equation}
 \begin{multline}
  \mathcal{D}({\mathcal Q}^{L}_{\theta,1}):= \big\{u\in L^{2}(\Omega^{T}_{1})~:~
  (-i\nb+{\bf F}_{\theta})u\in L^{2}(\Omega^{T}_{1}),
  \quad u(2T,\cdot,\cdot)=0 ,\\\Big(-\frac{L}{2},\cdot,\cdot\Big)=u\Big(\frac{L}{2},\cdot,\cdot\Big)=0,\quad
u\Big(\cdot,-\frac{L}{2},\cdot\Big)=u\Big(\cdot, \frac{L}{2},\cdot\Big)=0\big\}\,,
\end{multline}
and
\begin{equation}\label{quad-2}
  \mathcal{Q}^{L}_{\theta,2}(u)=\Int{\Omega_{2}^{T}}{}\Big[|\partial_{t}u|^{2}+|\partial_{s}u|^{2}+
  |(-i\partial_{r}+t\cos(\theta)-s\sin(\theta))u|^{2}\Big]drdsdt,
\end{equation}
  \begin{multline}
  \mathcal{D}({\mathcal Q}^{L}_{\theta,2}):= \big\{u\in L^{2}(\Omega^{T}_{2})~:~
  (-i\nb+{\bf F}_{\theta})u\in L^{2}(\Omega^{T}_{2}),
  \quad u(T,\cdot,\cdot)=0,\\
   \Big(-\frac{L}{2},\cdot,\cdot\Big)=u\Big(\frac{L}{2},\cdot,\cdot\Big)=0,\quad
u\Big(\cdot,-\frac{L}{2},\cdot\Big)=u\Big(\cdot, \frac{L}{2},\cdot\Big)=0 \big\}\,,
\end{multline}
respectively.

It is clear from \eqref{Ims} that,
\begin{equation}\label{app-cdv}
\mathcal{Q}^{L}_{\theta}(u)\geq \mathcal{Q}^{L}_{\theta,1}(\psi_{1,T}u)+\mathcal{Q}^{L}_{\theta,2}(\psi_{2,T}u)- \dfrac{C_{0}}{T^{2}}\int_{\Omega^{L}}
|u|^{2}drdsdt.
\end{equation}
By the variational principle (cf. \cite[Lemma 5.1]{C-Ve}), we see that
\begin{equation}
 \mathcal N(\lambda;\mathscr{P}_{\theta}^{L},\Omega^{L})
   \leq  \mathcal N(\lambda+{C_{0}}{T^{-2}};\mathscr{P}_{\theta,1}^{L}\oplus \mathscr{P}_{\theta,2}^{L},\Omega^{L}).
\end{equation}
It follows that
\begin{align}\label{estimate}
 \mathcal N(\lambda;\mathscr{P}_{\theta}^{L},\Omega^{L})
   \leq  \mathcal N(\lambda+{C_{0}}{T^{-2}};\mathscr{P}_{\theta,1}^{L},\Omega_{1}^{T})+
  \mathcal N(\lambda+{C_{0}}{T^{-2}};\mathscr{P}_{\theta,2}^{L},\Omega_{2}^{T}).
\end{align}
The Dirichlet boundary conditions imposed at $r\in\big\{-\frac{L}{2},\frac{L}{2}\big\}$, $s\in\big\{\frac{L}{2},\frac{L}{2}\big\}$ and $t=T$ ensures that the estimate \eqref{estimate} remains true if we replace $\Omega_{2}^{T}$ by $\R^{3}$ in the definition of $\mathcal{Q}^{L}_{\theta,2}$.

Since the first eigenvalue of the Schr\"odinger operator with constant unit magnetic
 field in $L^{2}(\R^{3})$ is equal to $1$, we thus find
\begin{equation}\label{lar-ld}
\begin{aligned}
\mathcal{Q}_{\theta,2}^{L}(u)
  &\geq\int_{\R^{3}} |u|^{2}drdsdt.
\end{aligned}
\end{equation}
Choose $T=2 \sqrt{\dfrac{C_{0}}{1-\lambda}}$, it holds that $1>\lambda+C_{0}T^{-2}$ and
\[
 \mathcal{Q}_{\theta,2}^{L}(u)>( \lambda+ C_{0}T^{-2}) \int_{\R^{3}} |u|^{2}drdsdt.
\]
This clearly gives $\mathcal N(\lambda+ C_{0}T^{-2};\mathscr{P}_{\theta},\R^{3})=0$. Thus, it remains  to estimate $\mathcal N(\lambda+C_{0}T^{-2} ;\mathscr{P}_{\theta,1}^{L},\Omega^{T}_{1})$.
To do this, we introduce a potential $V$ satisfying
\begin{equation}\label{Ass-V}
\left\{\begin{array}{ll}
V\geq 0,\\
 {\rm supp}~V \subset \R^{3}_{+}\setminus \Omega_{1}^{T}.
\end{array}\right.
\end{equation}
Under these assumptions on $V$, we may write for all $u\in D(\mathcal{Q}^{L}_{\theta,1})$,
\begin{equation}
\int_{\Omega^{T}_{1}}\big|(-i\nb+{\bf F}_{\theta})u\big|^{2}drdsdt= \int_{\Omega^{T}_{1}}\big|(-i\nb+{\bf F}_{\theta})u\big|^{2}drdsdt
+\int_{\R^{3}_{+}}V(x)|u|^{2}drdsdt.
\end{equation}
Here, we have extended $u$ by $0$ to the whole of $\R^{3}_{+}$ in the last integral. Therefore, it follows from the min-max principle that~:
\begin{equation}\label{NB2}
\mathcal N(\lambda+C_{0}T^{-2};\mathscr{P}^{L}_{\theta,1},\Omega_{1}^{T})= \mathcal N(\lambda+C_{0}T^{-2};\mathscr{P}_{\theta,1}^{L}+V,\Omega_{1}^{T}).
\end{equation}
Since any function $u$ that belongs to the form domain of $\mathcal{Q}^{L}_{\theta,1}$ can be extended by $0$ to the half space
 $\R^{3}_{+}$, we get using the bound in \eqref{par-of-uni} and the min-max principle that
\begin{align}\label{Nb3}
\mathcal N(\lambda+C_{0}T^{-2};\mathscr{P}^{L}_{\theta,1}+V,\Omega_{1}^{T})\leq\mathcal  N(\lambda;\mathscr{P}^{N}_{\theta}+V_{1},\R^{3}_{+}),
\end{align}
where
\[
V_{1}=V- \dfrac{C_{0}}{T^{2}}.
\]
We select the  potential $V$  in $\R^{3}_{+}$ as follows,
$$V(r,s,t):=\left(1+\dfrac{C_{0}}{T^{2}}\right) {\bf 1}_{\R^{3}_{+}\setminus\Omega^{T}_{1}} .$$
It is easy to check that $V$ satisfies the assumptions in
\eqref{Ass-V}. To $V$, we associate the reflected potential in
$\R^{3}$ defined by
$$\widetilde V(r,s,t):=\left(1+\dfrac{C_{0}}{T^{2}}\right) {\bf 1}_{\R^{3}\setminus\widetilde\Omega^{T}_{1}},$$
with
\[
\widetilde\Omega^{T}_{1}= \Big\{ (r,s,t)\in\Big(-\frac{L}{2},\frac{L}{2}\Big)^{2}\times\R\,:\, |t|<2T\Big\}.
\]
In view of Lemma~\ref{ref-lem}, we have,
\begin{equation}\label{Nb4}
\mathcal N(\lambda;\mathscr{P}_{\theta}^{N}+V_{1},\R^{3}_{+})\leq C_{\rm CLR}\int_{\R^{3}}\big(\lambda-\widetilde V_{1}\big)_{+}^{3/2}drdsdt
\end{equation}
where
\[
\widetilde V_{1}=\widetilde V- \dfrac{C_{0}}{T^{2}}.
\]
Next, we compute the integral
\[
\int_{\R^{3}}\big(\lambda-\widetilde V_{1}\big)_{+}^{3/2}drdsdt = 2\lambda^{3/2}\int_{\widetilde\Omega^{T}_{1}}drdsdt=4\ld^{3/2} L^{2}T.
\]
Inserting this in \eqref{Nb4}, we obtain
\begin{align}\label{Nb5}
\mathcal N(\lambda;{\mathscr{P}}^{N}_{\theta}+ V_{1},\R^{3}_{+})\leq  4C_{\rm CLR}\ld^{3/2} L^{2}T.
\end{align}
Combining the estimates \eqref{estimate}, \eqref{NB2}, \eqref{Nb3} and \eqref{Nb5} gives \eqref{bnd-nb} upon inserting the choice $T=2\sqrt{\frac{C_{0}}{1-\lambda}}$.

In a similar fashion, we can prove \eqref{bnd-engy} by following the steps of the proof of \eqref{bnd-nb}, and using Lemma~\ref{ref-lem} (the energy case).
\end{proof}
\section{The large area limit}
Consider $\theta\in[0,\pi/2]$ and a large positive number $L>0$.  Recall the magnetic potential introduced in \eqref{Ftheta} and the magnetic Schr\"{o}dinger operator $\mathscr{P}_{\theta}^{L}$ given in \eqref{PL1L2}.
In accordance with the definition of $\mathcal{E}$ in \eqref{def-energy}, we write, for $ \lambda\in[0,1)$,
 $$\mathcal{E}(\lambda;\mathscr{P}_{\theta}^{L},
\Omega^{L})=\sum_{j}\big(\zeta_{j}^{L}(\theta)-\lambda\big)_{-},$$
where $\{\zeta_{j}^{L}(\theta)\}_{j}$ denotes the sequence of eigenvalues of $\mathscr{P}_{\theta}^{L}$.

We are interested in the behaviour of $\mathcal{E}(\lambda;\mathscr{P}_{\theta}^{L},
\Omega^{L})$ as $L$ approach $\infty$. We will obtain a limiting function $E(\theta,\lambda)$ (see Theorem~\ref{thm-trm-limit} below) such that the leading order asymptotics
\[
\mathcal{E}(\lambda;\mathscr{P}_{\theta}^{L},
\Omega^{L})\sim E(\theta,\lambda)L^{2}
\]
holds true as $L\rightarrow\infty$. This approach was adapted in \cite{FK3D,FKP} to prove the existence of several limiting functions related to the Ginzburg-Landau functional.
We aim to prove
\begin{thm}\label{thm-trm-limit}
Let $\theta\in[0,\pi/2]$ and
$\lambda\in[0,1)$. There exists a constant $E(\theta,\lambda)$ such that
\begin{equation}\label{trm-lim}
  \liminf_{L\rightarrow \infty} \dfrac{\mathcal{E}(\lambda;\mathscr{P}^{L}_{\theta},\Omega^{L})}{L^{2}}=
   \limsup_{L\rightarrow \infty} \dfrac{\mathcal{E}(\lambda;\mathscr{P}^{L}_{\theta},\Omega^{L})}{L^{2}}= E(\theta,\lambda).
\end{equation}
Moreover, for all $\lambda_{0}\in [0,1)$, there exist positive uniform constants $L_{0}$ and $C_0$ such that,
\begin{equation}\label{bnd-eng-tl}
    E(\theta,\lambda)-  \dfrac{2C_{0}}{L^{2/3}}\leq  \dfrac{\mathcal{E}(\lambda;\mathscr{P}^{L}_{\theta},\Omega^{L})}{L^{2}}\leq  E(\theta,\lambda),
\end{equation}
for all $\theta \in[0,\pi/2]$, $\lambda\in[0,\lambda_{0}]$, $L\geq 2 L_{0}$.
\end{thm}
The proof of Theorem~\ref{trm-lim} relies on the following lemma, which is proved in \cite[Lemma~2.2]{FK3D}.
\begin{lem}\label{gen-lem}
Consider a decreasing function $d:(0,\infty)\to(-\infty,0]$ such that the function\break
$f:(0,\infty)\ni L\mapsto \frac{d(L)}{L}\in\R$ is bounded.

Suppose that there exist constants $C>0$, $L_{0}>0$ such that  the estimate
\begin{equation}\label{eq-gen-lem-ass}
f(nL) \geq f((1+a)L)-C\left(a+\frac1{a^2L^{2}}\right),\end{equation}
holds true for all $a\in(0,1)$, $n\in{\mathbb N}$, $L\geq L_{0}$.
Then $f(L)$ has a limit $A$ as $L\to\infty$. Furthermore, for all $L\geq 2 L_{0}$, the following estimate holds true,
\begin{equation}\label{eq-gen-lem}
f(L)\leq A+\frac{2C}{L^{2/3}}\,.
\end{equation}
\end{lem}
In order to use the result of Lemma~\ref{gen-lem}, we establish the estimate in the Lemma~\ref{lem-thdL} below.
\begin{lem}\label{lem-thdL}
Let $\lambda_{0}\in [0,1)$ and $\theta\in[0,\pi/2]$. There exist universal constants $C_0>0$ and $L_{0}\geq 1$ such that, for all $L\geq L_{0}$, $\lambda\in[0,\lambda_{0}]$,  $n\in\mathbb N$ and $a\in(0,1)$, we have,
\[
\frac{\mathcal{E}(\lambda;\mathscr{P}_{\theta}^{nL},\Omega^{nL})}{n^{2}L^{2}}\leq
\frac{\mathcal{E}(\lambda;\mathscr{P}_{\theta}^{(1+a)L},\Omega^{(1+a)L})}
{(1+a)^{2}L^{2}}+C_{0}\left(\dfrac{1}{a^2L^{2}}+ a\right)\,.
\]
Furthermore, the function $$L\mapsto \mathcal{E}(\lambda;\mathscr{P}_{\theta}^{L},\Omega^{L})$$ is monotone increasing.
\end{lem}
\begin{proof}
Let $n\geq 2$ be a natural number. If $a\in(0,1)$ and $j=(j_1,j_2)\in\mathbb{Z}^2$, let
\[
K_{a,j}=I_{j_1}\times I_{j_2}\,,
\]
where
\[
\forall~p\in\mathbb Z\,,\quad I_p=\bigg{(} \dfrac{2p+1-n}{2}-\dfrac{(1+a)}{2}\,,\,\dfrac{2p+1-n}{2}+\dfrac{(1+a)}{2}\bigg{)}\,.
\]
Consider a partition of unity $(\chi_j)_{j}$ of $\R^2$ such that:
\begin{equation}\label{pa-of-un}
\sum_{j}|\chi_j|^2=1\,,\quad 0\leq\chi_j\leq1\quad{\rm in}~\R^2\,,\quad {\rm supp}\,\chi_j\subset K_{a,j}\,,\quad
|\nabla\chi_j|\leq \frac{C}{a}\,,
\end{equation}
where $C$ is a universal constant.
We define $\chi_{j, L}(r,s)=\chi_j\big(\frac{r}{L},\frac{s}{L}\big)$. We thus obtain a new partition of unity $\{\chi_{j,L}\}_{j\in \mathcal{J}}$ such that
${\rm supp}\,\chi_{j,L}\subset \mathcal K_{a,j,L}$,
with
\[
\mathcal  K_{a,j,L}=\{(Lr,Ls)~:~(r,s)\in K_{a,j}\}\,.
\]
Let $\mathcal J=\{j=(j_1,j_2)\in \mathbb{Z}^2~:~0\leq  j_1,j_2\leq n-1\}$ and $K^{nL}=\left(-\frac{nL}{2},\frac{nL}{2}\right)^{2}$.
Then the family
$\{\mathcal K_{a,j,L}\}_{j\in\mathcal J}$ is a covering of  $K^{nL}$, and is formed of exactly $n^2$ squares with side length $L$.

We restrict the partition of unity $\{\chi_{j,L}\}_{j\in\mathcal J}$ to the
set $K^{nL}=\left(-\frac{nL}{2},\frac{nL}{2}\right)^{2}$.
Let $\mathcal{Q}^{nL}_{\theta}$ be the quadratic form defined in \eqref{QL1L2} and $\{f_{k,n}\}_{k=1}^{N}$ be any orthonormal set in $\mathcal{D}(\mathcal{Q}^{nL}_{\theta})$. By the IMS formula and the fact that $\sum_{j}\chi_{j,L}^2=1$, we have
\begin{multline}
\Sum{k=1}{N}\left(\mathcal{Q}^{nL}_{\theta}(f_{k,n})-\lambda\norm{f_{k,n}}_{L^{2}(\Omega^{nL})}^{2}\right)\\=\Sum{k=1}{N} \sum_{j\in\mathcal J}
\bigg{\{}\left(\mathcal{Q}_{\theta}^{nL}(\chi_{j,L}f_{k,n})-\lambda\norm{\chi_{j,L}f_{k,n}}_{L^2(\Omega^{nL})}^{2}\right)
-\|\,|\nabla\chi_{j,L}|\,f_{k,n}\|_{L^2(\Omega^{nL})}^2\bigg{\}}
\end{multline}
where $\Omega^{nL}=K^{nL}\times\R_{+}$ is as defined in \eqref{OmegaL}.
Using the bound on $|\nabla\chi_{j}|$ in \eqref{pa-of-un} we obtain
\begin{multline}\label{eq-IMS}
\Sum{k=1}{N}\left(\mathcal{Q}^{nL}_{\theta}(f_{k,n})-\lambda\norm{f_{k,n}}
_{L^{2}(\Omega^{nL})}^{2}\right)
\geq  \Sum{k=1}{N} \sum_{j\in\mathcal J}
\Big(\mathcal{Q}_{\theta}^{nL}(\chi_{j,L}f_{k,n})-
\big(\lambda+\dfrac{C}{a^{2}L^{2}}\big)\norm{\chi_{j,L}f_{k,n}}_{L^{2}(\Omega^{nL})}^{2}\Big).
\end{multline}
For $j\in\mathcal{J}$, we define the trial density matrix $\gamma: L^{2}(\mathcal{K}_{a,j,L})\ni f \mapsto \gamma_{j} f\in L^{2}(\mathcal{K}_{a,j,L})$,
\[
\gamma_{j}f= \chi_{j,L}\Sum{k=1}{N}\big\langle \chi_{j,L} f,f_{k,n}\big\rangle f_{k,n}.
\]
It is clear that $\gamma_{j}$ is a finite rank operator satisfying $\gamma_{j}=\gamma_{j}^2$ and $0\leq \gamma_{j}\leq 1$ (in the sense of quadratic forms). Moreover, we note that $\gamma_{j}$ is constructed so that we can write
\begin{equation}\label{eq-tr}
 {\rm Tr}\Big[\Big(\mathscr{P}_{\theta}^{nL}-\big(\lambda+\dfrac{C}{L^{2}a^{2}})\Big)\gamma_{j} \Big]
= \Sum{k=1}{N} \Big(\mathcal{Q}_{\theta}^{nL}(\chi_{j,L}f_{k,n})-(\lambda+ \dfrac{C}{L^{2}a^{2}})
\norm{\chi_{j,L}f_{k,n}}_{L^{2}(\Omega^{nL})}
^{2}\Big).
\end{equation}
Notice that each $\chi_{j,L}$ is supported in a square with side length $(1+a)L$. Hence, using magnetic translation invariance and Lemma~\ref{lem-VP-3}, it follows that
\begin{equation*}
-\mathcal{E}\Big(\ld + \dfrac{C}{L^{2}a^{2}};\mathscr{P}_{\theta}^{(1+a)L},\Omega^{(1+a)L}\Big)
\leq  {\rm Tr}\Big[\Big(\mathscr{P}_{\theta}^{nL}-\big(\lambda+\dfrac{C}{L^{2}a^{2}})\big)\gamma_{j} \Big],\quad \forall~j\in\mathcal{J}.
\end{equation*}
Using \eqref{eq-tr}, this reads
\begin{equation*}
-\mathcal{E}\Big(\ld + \dfrac{C}{L^{2}a^{2}};\mathscr{P}_{\theta}^{(1+a)L},\Omega^{(1+a)L}\Big)
\leq \Sum{k=1}{N} \Big(\mathcal{Q}_{\theta}^{nL}(\chi_{j,L}f_{k,n})-(\lambda+ \dfrac{C}{L^2a^{2}})
\norm{\chi_{j,L}f_{k,n}}_{L^{2}(\Omega^{nL})}
^{2}\Big).
\end{equation*}
Substituting into \eqref{eq-IMS}, we obtain
\begin{equation}
\Sum{k=1}{N}\left(\mathcal{Q}^{nL}_{\theta}(f_{k,n})-\lambda\norm{f_{k,n}}_{L^{2}(\Omega^{nL})}^{2}\right)
\geq - n^{2}\mathcal{E}\Big(\ld + \dfrac{C}{L^{2}a^{2}};\mathscr{P}_{\theta}^{(1+a)L},\Omega^{(1+a)L}\Big)
\end{equation}
for all orthonormal family $\{f_{k,n}\}_{k=1}^{N}$ and $N\in\mathbb{N}$. Therefore we conclude, in view of Lemma~\ref{lem-VP-2}, that
\begin{equation}\label{energy-lattice}
  {\mathcal{E}(\ld;\mathscr{P}_{\theta}^{nL},\Omega^{nL})}\leq n^{2}
  {\mathcal{E}\Big(\ld + \dfrac{C}{L^{2}a^{2}};\mathscr{P}_{\theta}^{(1+a)L},\Omega^{(1+a)L}\Big)}.
\end{equation}
Let $L_{0}\geq C/(a\sqrt{1-\lambda_{0}}) $ chosen so that $\ld + {C}{a^{-2}L^{-2}}<1$ for all $L\geq L_0$. Applying Lemma~\ref{NB-vs-EN} with $\varsigma=\dfrac{C}{a^{2}L^{2}}$, we find,
\begin{multline*}
    {\mathcal{E}\Big(\ld + \frac{C}{a^{2}L^{2}};\mathscr{P}_{\theta}^{(1+a)L},\Omega^{(1+a)L}\Big)}
    -
      {\mathcal{E}\big(\ld;\mathscr{P}_{\theta}^{(1+a)L},\Omega^{(1+a)L}\big)}
      \leq  \dfrac{C}{L^{2}a^{2}}
      {\mathcal N\Big(\lambda +\frac{C}{a^{2}L^{2}};\mathscr{P}_{\theta}^{(1+a)L},\Omega^{(1+a)L}\Big)}.
\end{multline*}
By \eqref{bnd-nb}, it follows that
\begin{equation*}
    {\mathcal{E}\Big(\ld + \dfrac{C}{a^{2}L^{2}};\mathscr{P}_{\theta}^{(1+a)L},\Omega^{(1+a)L}\Big)}\leq
      {\mathcal{E}\big(\ld;\mathscr{P}_{\theta}^{(1+a)L},\Omega^{(1+a)L}\big)}+ \dfrac{C}{a^{2}},
\end{equation*}
for all $L\geq L_{0}$ and $\lambda\in[0,\lambda_{0}]$.
Inserting this into \eqref{energy-lattice}, we get,
\begin{equation}
  {\mathcal{E}(\ld;\mathscr{P}_{\theta}^{nL},\Omega^{nL})}\leq n^{2}
  {\mathcal{E}\big(\ld;\mathscr{P}_{\theta}^{(1+a)L},\Omega^{(1+a)L}\big)}+ \dfrac{C n^{2}}{a^{2}}\,.
\end{equation}
Dividing both sides by $n^{2}L^{2}$, we find
\begin{equation}
  \dfrac{\mathcal{E}(\ld;\mathscr{P}_{\theta}^{nL},\Omega^{nL})}{n^{2}L^{2}}\leq
  \dfrac{\mathcal{E}\big(\ld;\mathscr{P}_{\theta}^{(1+a)L},\Omega^{(1+a)L}\big)}{L^{2}}+ \dfrac{C}{L^{2}a^{2}}\,.
\end{equation}
We infer from \eqref{bnd-engy} the following upper bound,
\begin{equation}
  \dfrac{\mathcal{E}(\ld;\mathscr{P}_{\theta}^{nL},\Omega^{nL})}{n^{2}L^{2}}\leq
  \dfrac{\mathcal{E}\big(\ld;\mathscr{P}_{\theta}^{(1+a)L},\Omega^{(1+a)L}\big)}{(1+a)^{2}L^{2}}+ C\left(a+\dfrac{1}{a^{2}L^{2}}\right),
\end{equation}
for all $L\geq L_{0}$ and $\lambda\in[0,\lambda_{0}]$. This proves the first assertion of the lemma.

To obtain monotonicity of $\mathcal{E}(\lambda;\mathscr{P}_{\theta}^{L},\Omega_{L})$, we consider $L^{\prime}\geq L> 0$.
Since the extension by zero of a function in the form domain $\mathscr{P}_{\theta}^{L}$ is contained in
the form domain of $\mathscr{P}_{\theta}^{L^{\prime}}$ and the values of both forms coincide for such a function, we may write in the sense of quadratic forms
\[
\mathscr{P}_{\theta}^{L^{\prime}}\leq \mathscr{P}_{\theta}^{L}.
\]
On account of Lemma~\ref{lem-VP-3}, it follows that,
\begin{align*}
-{\rm Tr}\big(\mathscr{P}_{\theta}^{L^{\prime}}-\lambda\big)_{-}
\leq 
-{\rm Tr}\big(\mathscr{P}_{\theta}^{L}-\lambda\big)_{-}.
\end{align*}
This shows that $\mathcal{E}(\lambda;\mathscr{P}_{\theta}^{L},\Omega^{L})$ is monotone increasing with respect to $L$, thereby proving the statement of the lemma.
\end{proof}
\begin{proof}[Proof of Theorem~\ref{thm-trm-limit}]
Let $f(L)= -\frac{\mathcal{E}(\lambda;\mathscr{P}_{\theta}^{L},\Omega^{L})}{L^{2}}$. Thanks
 to Lemma~\ref{lem-thdL}, we know that the functions $f(L)$ and $d(L)=-\mathcal{E}(\lambda;\mathscr{P}_{\theta}^{L},\Omega^{L})$ satisfy the assumptions
 in Lemma~\ref{gen-lem}. Consequently, $f(L)$ has a limit as $L\rightarrow\infty$. 
Let us define
\[
 E(\theta,\lambda):= -\lim_{L\rightarrow\infty} f(L).
\]
By Lemma~\ref{gen-lem}, there exists $L_{0}>0$ such that
\begin{equation}\label{Eq-Ethlb}
E(\theta;\lambda)\leq\dfrac{\mathcal{E}(\lambda;\mathscr{P}^{L}_{\theta},\Omega^{L})}{L^{2}}+ \dfrac{2C_{0}}{L^{2/3}},
\end{equation}
for all $L\geq 2 L_{0}$ and $\lambda\in[0,\lambda_{0}]$.

It remains to establish the upper bound. According to Lemma~\ref{monot}, we know that the energy satisfies
\begin{equation*}
\dfrac{\mathcal{E}(\ld;\mathscr{P}_{\theta}^{nL},\Omega^{nL})}{ n^{2} L^{2}}\geq
  \dfrac{\mathcal{E}(\ld;\mathscr{P}_{\theta}^{L},\Omega^{L})}{ L^{2}}\,.
\end{equation*}
Letting $n\rightarrow\infty$ gives us
\begin{equation*}
  E(\theta,\lambda)\geq   \dfrac{\mathcal{E}(\lambda;\mathscr{P}^{L}_{\theta},\Omega^{L})}{L^{2}}.
  \end{equation*}
  This, together with \eqref{Eq-Ethlb}, completes the proof of Theorem~\ref{thm-trm-limit}.
\end{proof}
 \section{Properties of the function $E(\theta,\lambda)$}
In Theorem~\ref{thm-trm-limit}, we proved the existence of a
limiting function $E(\theta,\lambda)\in[0,\infty)$ defined for
$\theta\in[0,\pi/2]$ and $\lambda\in[0,1)$. We aim in this section
to study the properties of $E(\theta,\lambda)$ as a function of
$\theta$ and $\lambda$. In particular, we establish continuity of
 $E(\theta,\lambda)>0$ with respect to $\theta$ and $\lambda$.
\begin{lem}\label{continuity+}
Let $\lambda_{0}\in[0,1)$ and $\delta>0$. There exists $L^\ast>0$
such that the following is true. Let $ L\geq L^\ast$, there exists
$\eta>0$ such that, if
 $$(\epsilon,\nu)\in (-\eta,\eta)^{2}\quad{~\rm
 and~}\quad(\theta,\lambda)\in[0,\pi/2]\times[0,\lambda_{0}]\,,$$
 then
$$
\Big|\dfrac{\mathcal{E}(\lambda+\nu;\mathscr{P}_{\theta+\epsilon}^{L},\Omega^{L})}{L^{2}}-
\dfrac{\mathcal{E}(\lambda;\mathscr{P}_{\theta}^{L},\Omega^{L})}{L^{2}}\Big|
\leq \frac\delta2\,.
$$
\end{lem}
\begin{proof}
We introduce a partition of unity of $\R$,
\begin{equation}\label{cst-loc}
\zeta_{1}^{2}+\zeta_{2}^{2}=1,\qquad {\rm supp}~\zeta_{1}\subset [0,1],\qquad {\rm supp}~\zeta_{2}\subset [{1}/{2},\infty),\quad |\zeta^{\prime}_{p}|\leq C^{\prime}, \qquad p=1,2.
\end{equation}
Let
 \begin{equation}\label{valueL1}
 L\geq L^\ast:=\max\bigg\{\Big( \dfrac{4C^{\prime}}{1-\lambda_{0}}\Big)^{1/2},\Big(\dfrac{4C^{\prime}C}{\delta\sqrt{(1-\lambda_{0})/2}}\Big)^{1/2}\bigg\},
 \end{equation}
where $C^{\prime}$ and $C$ are the universal constants appearing in
\eqref{cst-loc} and \eqref{bnd-engy} respectively. We put further,
$$\zeta_{p,L }(r,s,t)=\zeta_{p}(t/L ),\qquad p=1,2,\qquad (r,s,t)\in\R^{3}_{+}.$$
Next, let $\{f_{k}\}_{k=1}^{N}$ be an orthonormal family of
compactly supported functions in $\mathcal D(\mathcal{Q}^{L
}_{\theta+\epsilon})$. We have the following IMS decomposition
formula
 \begin{multline}\label{eq-ims}
\Sum{k=1}{N}\left(\mathcal{Q}^{L
}_{\theta+\epsilon}(f_{k})-(\lambda+\nu)\norm{f_{k}}_{L^{2}(\Omega^{L
})}^{2}\right)\\= \Sum{k=1}{N}\Sum{p=1}{2}
\bigg{\{}\left(\mathcal{Q}_{\theta+\epsilon}^{L }(\zeta_{p,L
}f_{k})-(\lambda+\nu)\norm{\zeta_{p,L }f_{k}}_{L^{2}(\Omega^{L
})}^{2}\right) -\|\,|\nabla\zeta_{p,L }|\ f_{k}\|^2_{L^2(\Omega^{L
})}\bigg{\}}\,.
\end{multline}
To estimate the last term we use the bound on $\nb\zeta_{p}$ in
\eqref{cst-loc}, and get, after inserting $\zeta_{1,L
}^{2}+\zeta_{2,L }^{2}=1$, that
\begin{equation}\label{est-1}
\Sum{k=1}{N}\left(\mathcal{Q}^{L }_{\theta+\epsilon}(f_{k})-(\lambda+\nu)\norm{f_{k}}^{2}\right)
\geq \Sum{k=1}{N}\Sum{p=1}{2}
\left(\mathcal{Q}_{\theta+\epsilon}^{L }(\zeta_{p,L }f_{k})-(\lambda+\nu+ C^{\prime}L ^{-2})\norm{\zeta_{p,L }f_{k}}^{2}\right)\,.
\end{equation}
Since the function $\varphi=\zeta_{2,L }f_{k}\in
C^{\infty}_{0}(\Omega^{L })$ is compactly supported in $\Omega^{L
}$, it can be extended by zero to all of $\R^{3}$. Hence, using the
fact that the first eigenvalue of the Schr\"odinger operator in
$L^{2}(\R^{3})$ is $1$, and selecting \(|\nu|<
\frac{\lambda_{0}-\lambda}{4}\), the choice of $L $ in
\eqref{valueL1} yields
\begin{equation}\label{term1}
\mathcal{Q}_{\theta+\epsilon}^{L }(\zeta_{2,L }f_{k})
 \geq \Int{\Omega^{L }}{}|\zeta_{2,L }f_{k}|^{2}drdsdt> (\lambda+\nu+ C^{\prime}L ^{-2})\Int{\Omega^{L }}{}|\zeta_{2,L }f_{k}|^{2}drdsdt.
\end{equation}
Consequently, we find that the term corresponding to $p=2$ on the right hand side of \eqref{est-1} is strictly positive and can be neglected for a lower bound.
 What remains is to estimate the term corresponding to $p=1$ in \eqref{est-1}. Using Cauchy-Schwarz inequality (with $\varrho$ arbitrary), we obtain,
\begin{multline}\label{Eq-ftheta+}
  \mathcal{Q}^{L }_{\theta+\epsilon}(\zeta_{1,L }f_{k})
  \\
  \geq (1-\varrho) \Int{\Omega^{L }}{}|(-i\nb+{\bf F}_{\theta+\epsilon})\zeta_{1,L }f_{k}|^{2}drdsdt-
\varrho^{-1}\Int{\Omega^{L }}{}|({\bf F_{\theta+\epsilon}}-{\bf
F}_{\theta})\zeta_{1,L}f_{k}|^{2}drdsdt.
\end{multline}
where ${\bf F}_{\theta}$ is the same as in \eqref{Ftheta}.
Using the bounds
\[
|\cos(\theta+\epsilon)-\cos(\theta)|\leq|\epsilon|\,,\quad |\sin(\theta+\epsilon)-\sin(\theta)|\leq |\epsilon|\,,
\]
we get
\[
|{\bf F}_{\theta+\epsilon}(r,s,t)-{\bf F}_\theta(r,s,t)|\leq |\epsilon|\left(|s|+|t|\right),\quad\forall~(r,s,t)\in\R^{3}_{+}\,.
\]
Taking the support of $\zeta_{1,L }$ into consideration, we infer
from \eqref{Eq-ftheta+} the following bound,
\begin{align*}
  \mathcal{Q}^{L }_{\theta+\epsilon}(\zeta_{1,L }f_{k})
&\geq (1-\varrho ) \Int{\Omega^{L }}{}|(-i\nb+{\bf F}_{\theta})\zeta_{1,L }f_{k}|^{2}drdsdt
- \varrho^{-1}\epsilon^{2}L ^{2}\Int{\Omega^{L }}{}|\zeta_{1,L }f_{k}|^{2}drdsdt.\,
\end{align*}
Inserting this into \eqref{eq-ims}, we get, using the bound on
$|\zeta^{\prime}_{1,L }|$, that
\begin{multline}\label{eq-ims-1}
\Sum{k=1}{N}\left(\mathcal{Q}^{L }_{\theta+\epsilon}(f_{k})-(\lambda+\nu)\norm{f_{k}}_{L^{2}(\Omega^{L_{1},L })}^{2}\right)\\
\geq\Sum{k=1}{N}\left((1-\varrho)\mathcal{Q}_{\theta}^{L
}(\zeta_{1,L }f_{k})-(\lambda+\nu+ \epsilon^{2}\varrho^{-1} ( L
)^{2}+ C^{\prime}L ^{-2})\norm{\zeta_{1,L
}f_{k}}^{2}_{L^{2}(\Omega^{L })}\right)\,.
\end{multline}
We choose $\varrho=|\epsilon|$ and define the trial density matrix $L^{2}(\R^{3})\ni f \mapsto\gamma f\in L^{2}(\R^{3})$,
\[
\gamma f= \Sum{k=1}{N}\langle f,\zeta_{1,L }f_{k}\rangle\zeta_{1,L }f_{k}.
\]
It is clear that $0\leq \gamma\leq 1$ in the sense of quadratic
forms and that $\mathscr{P}_{\theta}^{L }\gamma$ is trace class (actually this is a finite-rank operator). By Lemma~\ref{lem-VP-3} we see that
\begin{multline}\label{ins-dm}
-\mathcal{E}\Big(\dfrac{\ld +\nu+|\epsilon| L ^{2}+ C^{\prime}L
^{-2}}{1-|\epsilon|};{\mathscr{P}}_{\theta}^{L },\Omega^{L }\Big)
\leq  {\rm Tr}\Big[\Big(\mathscr{P}_{\theta}^{L }-\big(\dfrac{\ld +\nu+|\epsilon| L ^{2}+ C^{\prime}L ^{-2}}{1-|\epsilon|}\big)\Big)\gamma \Big]\\
:= \sum_{k=1}^N\left({\mathcal{Q}}_{\theta}^{L }(\zeta_{1,L
}f_{k})-\big(\dfrac{\ld +\nu+|\epsilon| L ^{2}+ C^{\prime}L
^{-2}}{1-|\epsilon|}\big)\norm{\zeta_{1,L }f_{k}}^2_{L^{2}(\Omega^{L
})}\right)\,.
\end{multline}
Inserting this into \eqref{eq-ims-1}, we obtain
\begin{multline}
\Sum{k=1}{N}\left(\mathcal{Q}^{L
}_{\theta+\epsilon}(f_{k})-(\lambda+\nu)\norm{f_{k}}^2_{L^{2}(\Omega^{L
})}\right) \geq-(1-|\epsilon|) \mathcal{E}\Big(\dfrac{\ld
+\nu+|\epsilon| L ^{2}+ C^{\prime}L
^{-2}}{1-|\epsilon|};{\mathscr{P}}_{\theta}^{L },\Omega^{L }\Big).
\end{multline}
Consequently, it follows from Lemma~\ref{lem-VP-2} that,
\begin{equation}\label{mn}
  {\mathcal{E}(\ld+\nu;\mathscr{P}_{\theta+\epsilon}^{L },\Omega^{L })}\leq
  (1-|\epsilon|){\mathcal{E}\Big(\dfrac{\ld +\nu+ |\epsilon| L ^{2}+ C^{\prime}L ^{-2}}{1-|\epsilon|};\mathscr{P}_{\theta}^{L },\Omega^{L }\Big)}.
\end{equation}
Fix $|\epsilon|<\frac{1-\lambda_{0}}{4(1+L ^{2})}$. Applying
Lemma~\ref{NB-vs-EN} with $\varsigma=\dfrac{ |\epsilon|( L
^{2}+\lambda)+\nu+ C^{\prime}L ^{-2}}{1-|\epsilon|}$, we get,
\begin{multline}\label{pq}
  {\mathcal{E}\Big(\dfrac{\ld +\nu+|\epsilon| L ^{2}+ C^{\prime}L ^{-2}}{1-|\epsilon|};\mathscr{P}_{\theta}^{L },\Omega^{L }\Big)}
  \leq
{\mathcal{E}(\ld;\mathscr{P}_{\theta}^{L },\Omega^{L })}\\
+ \dfrac{|\epsilon|( L ^{2}+\lambda)+\nu+C^{\prime} L
^{-2}}{1-|\epsilon|} \mathcal{N}\Big(\dfrac{\ld +\nu+|\epsilon| L
^{2}+ C^{\prime}L ^{-2}}{1-|\epsilon|};\mathscr{P}_{\theta}^{L
},\Omega^{L }\Big).
\end{multline}
Plugging \eqref{pq} into \eqref{mn}, we obtain from \eqref{bnd-engy} that
\begin{multline}\label{con-theta}
  \mathcal{E}({\ld+\nu;\mathscr{P}_{\theta+\epsilon}^{L },\Omega^{L }})
  \leq
 (1-|\epsilon|){\mathcal{E}(\ld;\mathscr{P}_{\theta}^{L },\Omega^{L })}
 +\frac{C}{\sqrt{(1-\lambda_{0})/2}}(|\epsilon|(\lambda+L ^{2})+\nu+C^{\prime} L ^{-2}) L ^{2},
\end{multline}
where $C$ is the constant from \eqref{bnd-nb}.  Interchanging the roles of $\theta$ and $\theta+\epsilon$ we arrive at
\begin{multline}\label{con-theta1}
  \Big|\mathcal{E}({\ld+\nu;\mathscr{P}_{\theta+\epsilon}^{L },\Omega^{L }})
  -    {\mathcal{E}(\ld;\mathscr{P}_{\theta}^{L },\Omega^{L })}\Big|        \\
                               \leq
 |\epsilon|{\mathcal{E}(\ld;\mathscr{P}_{\theta}^{L },\Omega^{L })}+\frac{C}{\sqrt{(1-\lambda_{0})/2}}(|\epsilon|(\lambda+L ^{2})+|\nu|+C^{\prime} L ^{-2})L ^{2}.
\end{multline}
Dividing both sides by $L ^{2}$, we get,
\begin{multline*}
  \bigg|\dfrac{\mathcal{E}(\ld+\nu;\mathscr{P}_{\theta+\epsilon}^{L },\Omega^{L })}{L ^{2}}- \dfrac{{\mathcal{E}(\ld;\mathscr{P}_{\theta}^{L },\Omega^{L })}}{L ^{2}}\bigg|
  \leq
  |\epsilon|\dfrac{{\mathcal{E}(\ld;\mathscr{P}_{\theta}^{L },\Omega^{L })}}{L ^{2}}+\frac{C(|\epsilon|(( L )^{2}+\lambda)+|\nu|+C^{\prime} L ^{-2})}{\sqrt{(1-\lambda_{0})/2}}\\
   \leq \eta\dfrac{{\mathcal{E}(\ld;\mathscr{P}_{\theta}^{L },\Omega^{L })}}{L ^2}+\frac{C(\eta (L ^{2} +\lambda+1)+ C^{\prime}L ^{-2})}{\sqrt{(1-\lambda_{0})/2}}.
\end{multline*}
Using the estimate in \eqref{bnd-engy}, we further obtain
\begin{align}\label{est-2}
  \left|\dfrac{{\mathcal{E}(\ld+\nu;\mathscr{P}_{\theta+\epsilon}^{L },\Omega^{L })}}{L ^{2}}- \dfrac{\mathcal{E}(\ld;\mathscr{P}_{\theta}^{L },\Omega^{L })}{L ^{2}}\right|
  \leq \frac{C(\eta (L ^{2} +\lambda+2)+C^{\prime} L ^{-2})}{\sqrt{(1-\lambda_{0})/2}}.
\end{align}
Selecting $\eta<\min\Big(\frac{1-\lambda_0}{4(1+L^2)},\frac{\delta
\sqrt{(1-\lambda_{0})/2}}{4 C(\lambda_0+2+L ^{2})}\Big)$ and using
the condition \eqref{valueL1} on $L $, we conclude that,
\[
  \bigg|\dfrac{{\mathcal{E}(\ld+\nu;\mathscr{P}_{\theta+\epsilon}^{L },\Omega^{L })}}{L ^{2}}- \dfrac{{\mathcal{E}(\ld;\mathscr{P}_{\theta}^{L },\Omega^{L })}}{L ^{2}}\bigg|\leq \frac\delta2,
  \]
thereby proving the assertion of the lemma.
\end{proof}
We have the following corollary of Lemma~\ref{continuity+}.
\begin{cor}
Given $\lambda_{0}\in[0,1)$, the function $$[0,\pi/2]\times[0,\lambda_{0}]\ni (\theta,\lambda)\mapsto E(\theta;\lambda)$$ is continuous.
\end{cor}
\begin{proof}
In view of Theorem~\ref{thm-trm-limit}, there exist constants $C_{0}$ and $L_{0}$ such that for all $L\geq 2L_{0}$ and
$(\nu,\epsilon)$ satisfying $\lambda+\nu\in[0,\lambda_{0}]$ and $\theta+\epsilon\in[0,\pi/2]$, one has
\begin{equation}\label{est}
|E(\theta+\epsilon,\lambda+\nu)-E(\theta,\lambda)|
\leq \dfrac{|\mathcal{E}(\ld+{\nu};\mathscr{P}_{\theta+\epsilon}^{L},\Omega^{L}) -\mathcal{E}(\ld;\mathscr{P}_{\theta}^{L},\Omega^{L})  |}{L^{2}}+ \dfrac{2C_{0}}{L^{2/3}}.
\end{equation}
Let $\delta>0$ and select $L^\ast$ as in Lemma~\ref{continuity+}.
Let $L\geq\max\{2L_{0},L^\ast ,(4C_{0}/\delta)^{3/2}\}$. We assign to $L$
a constant $\eta>0$ as described in Lemma~\ref{continuity+}.
Consequently, if
$$(\nu,\epsilon)\in(-\eta,\eta)^2\,,\quad
\lambda+\nu\in[0,\lambda_{0}]\,,\quad{\rm and}
\quad\theta+\epsilon\in[0,\pi/2]\,,$$ then
\[
|E(\theta+\epsilon,\lambda+\nu)-E(\theta,\lambda)|\leq \delta/2+\delta/2=\delta.
\]
This finishes the proof.
\end{proof}

 We conclude this section by the following lemma.
\begin{lem}\label{fct-Lip}Let $\theta\in[0,\pi/2]$. The function
\[
[0,1)\ni\lambda\mapsto E(\theta,\lambda)
\]
is locally Lipschitz.
\end{lem}
\begin{proof}
Fix $\lambda_{0}\in[0,1)$, and let
$\lambda_{1},\lambda_{2}\in[0,\lambda_{0}]$ be such that
$\lambda_{1}<\lambda_{2}$. Let $L>0$ and $\mathscr{P}_{\theta}^{L}$
be as defined in \eqref{PL1L2}. We infer from Lemma~\ref{NB-vs-EN}
that
\[
\mathcal{E}(\lambda_{2};\mathscr{P}_{\theta}^{L},\Omega^{L})-\mathcal{E}(\lambda_{1}; \mathscr{P}_{\theta}^{L},\Omega^{L})\leq (\lambda_{2}-\lambda_{1})\mathcal{N}(\lambda_{2}; \mathscr{P}_{\theta}^{L},\Omega^{L}).
\]
In view of \eqref{bnd-nb}, there exists a constant $C_{0}$ independent of $\theta$ such that
\[
\mathcal{N}(\lambda_{2}; \mathscr{P}_{\theta}^{L},\Omega^{L})\leq C_{0} L^{2}.
\]
This implies
\[
\mathcal{E}(\lambda_{2};\mathscr{P}_{\theta}^{L},\Omega^{L})-\mathcal{E}(\lambda_{1}; \mathscr{P}_{\theta}^{L},\Omega^{L})\leq C_{0} L^{2} (\lambda_{2}-\lambda_{1}).
\]
Dividing both sides by $L^{2}$, we get, after taking $L\rightarrow \infty$,
\begin{equation}\label{Lip}
E(\theta,\lambda_{2})-E(\theta,\lambda_{1})\leq C_{0}(\lambda_{2}-\lambda_{1}).
\end{equation}
Interchanging the roles of $\lambda_{1}$ and $\lambda_{2}$, we get
further,
\begin{equation}\label{Lip-1}
|E(\theta,\lambda_{2})-E(\theta,\lambda_{1})|\leq C_{0}|\lambda_{2}-\lambda_{1}|,
\end{equation}
which gives the assertion of the lemma.
\end{proof}
\section{Explicit formula of $E(\theta,\lambda)$}
Recall the constant $E(\theta,\lambda)$ defined in \eqref{trm-lim}. The aim of this section is to provide an explicit formula for
$E(\theta,\lambda)$ using the projectors on the eigenfunctions of the Neumann Schr\"{o}dinger operator given in \eqref{Neum-r3+}. We shall consider the cases $\theta=0$ and $\theta\in(0,\pi/2]$ separately. Indeed, the construction of eigenprojectors in the case $\theta=0$ is similar in spirit to the two-dimensional case (cf.~\cite[Section~4]{Fo-Ka}), whereas in the case $\theta\in(0,\pi/2]$, the projectors are constructed using the spectral decomposition of the two-dimensional model operator $\mathcal{L}(\theta)$.
\subsection{$E(\theta,\lambda)$ in the case $\theta=0$}
We start by recalling the family of one-dimensional harmonic oscillators $H(\xi)$, $\xi\in\R$, defined by :
\begin{equation}\label{harm-osc}
H(\xi)= -\partial_{t}^{2}+(t-\xi)^{2}\qquad{\rm in}\qquad L^{2}(\R_{+}).
\end{equation}
on their common Neumann domain:
\[
\big\{v\in H^{2}({\R_{+}}),~ t^{2}v\in L^{2}(\R_{+}),~ v^{\prime}(0)=0\big\}.
\]
We denote by $(u_{j}(\cdot;\xi))_{j=1}^{\infty}$ the orthonormal family of real-valued eigenfunctions of the operator $H(\xi)$, i.e.,
\begin{equation}\label{muj}
  H(\xi)u_{j}(t;\xi)=\mu_{j}(\xi)u_{j}(t;\xi),
 \quad u_{j}^{\prime}(0;\xi)=0,\\
\quad\Int{\R_{+}}{}u_{j}(t;\xi)^{2}dt=1.
\end{equation}
The lowest eigenvalue $\mu_{1}(\xi)$ is studied in \cite{BoHe,DaHe}. We collect in the following proposition some of the  properties of $\mu_{1}(\xi)$ as a function of $\xi$~:
\begin{pro}\label{properties-mu1}
The function $\R\ni\xi\mapsto\mu_{1}(\xi)$ is continuous and satisfies
\begin{enumerate}
\item $\mu_{1}(\xi)>0$, for all $\xi\in\R.$
\item At $-\infty$ we have the limit
\begin{equation}\label{limit-at-infty1}
\lim_{\xi\rightarrow-\infty}\mu_{1}(\xi)=+\infty.
\end{equation}
\item At the origin the value is
\begin{equation}\label{mu1-0}
    \mu_{1}(0)=1.
\end{equation}
\item At $+\infty$ we have
\begin{equation}\label{limit-at-infty2}
\lim_{\xi\rightarrow+\infty}\mu_{1}(\xi)=1.
\end{equation}
\item $\mu_{1}$ has a minimum $\Theta_{0}\in (0,1)$ at a unique  $\xi_{0}\in (0,1)$,
\begin{equation}\label{theta0}
\Theta_{0}:=\inf_{\xi\in\R}\mu_{1}(\xi)=\mu_{1}(\xi_{0})<1.
\end{equation}
Moreover, this minimum is non-degenerate and $\mu_{1}(\xi)$ is strictly decreasing on $(-\infty,\xi_{0}]$ from $+\infty$ to $\Theta_{0}$ and strictly increasing on $[\xi_{0},\infty)$ from $\Theta_{0}$ to $1$.
\end{enumerate}
\end{pro}
 The next Lemma is taken from \cite[Lemma~2.1]{Fr}.
\begin{lem}\label{mu2}
The second eigenvalue $\mu_{2}(\xi)$ satisfies,  $$\inf_{\xi\in\R}\mu_{2}(\xi)>1.$$
\end{lem}
Thanks to Proposition~\ref{properties-mu1}, one can easily prove the
following:
\begin{lem}\label{comp-integ}
Let $\mu_{1}(\xi)$ be defined as in \eqref{muj}. We have
 $$
\int_{\R^{2}} \big(\mu_{1}(\xi)+\tau^{2}-\lambda\big)_{-}d\xi d\tau= \dfrac{4}{3} \int_{0}^{\infty}(\mu_{1}(\xi)-\lambda)_{-}^{3/2}d\xi,
 $$
 and  the integrals are finite for all $\lambda\in[0,1)$.
\end{lem}
For later reference, we include Agmon-type estimates on the
eigenfunction $u_{1}(t;\xi)$ (cf. \cite[Theorem 2.6.2]{Ka}).
\begin{lem}\label{Agmon-lem}
Let $\lambda\in[0,1)$. For all $\epsilon\in(0,1)$, there exists a constant $C_{\epsilon}$ such that, for all $\xi\in\R_{+}$ satisfying $\mu_{1}(\xi)\leq \lambda$, we have
\begin{equation}
\Big\|{e^{\frac{\epsilon(t-\xi)^{2}}{2}}u_{1}(t,\xi)}\Big\|^{2}_{H^{1}\big(\{t\in\R_{+}\,:\,(t-\xi)^{2}\geq C_{\epsilon}\}\big)}\leq C_{\epsilon}.
\end{equation}
\end{lem}
Next, we consider the Schr\"{o}dinger operator \eqref{Neum-r3+} in the particular case $\theta=0$, i.e.,
\begin{equation}\label{op-N-th=0}
\mathscr{P}_{0}^{N}=-\partial_{t}^{2}-\partial_{s}^{2}+(-i\partial_{r}+t)^{2}\qquad \rm{in}\qquad  L^{2}(\R^{3}_{+}).
\end{equation}
Let $(\xi,\tau)\in\R^{2}$ and denote by
$\mathcal{F}_{r\rightarrow\xi}$ (resp.
$\mathcal{F}_{s\rightarrow\tau}$) the partial Fourier transform with
respect to the variable $r$ (resp. $s$).
We define the bounded function $\R^{3}_{+}\ni(r,s,t)\mapsto v_{j}(r,s,t;\xi,\tau)$ by~:
\begin{equation}\label{def-vj}
v_{j}(r,s,t;\xi,\tau)=\dfrac{1}{\sqrt{2\pi}}e^{-i\xi r}e^{-i\tau s}  u_{j}(t;\xi).
\end{equation}
Then, we introduce the projectors $\Pi_{j}(\xi,\tau)$ on the functions $v_{j}$~:
\begin{align}\label{PR-Pi-j}
 L^{2}({\R^{3}_{+}})\ni\varphi & \mapsto (\Pi_{j}(\xi,\tau)\varphi)(r,s,t)=
v_{j}(r_{1},s_{1},t_{1};\xi,\tau)\Int{\R^{3}_{+}}{}
\overline{v_{j}(r_{2},s_{2},t_{2};\xi,\tau)}\varphi(r_{2},s_{2},t_{2})dr_{2}ds_{2}dt_{2}
\end{align}
so that we can write, in terms of quadratic forms,
\begin{equation}\label{i-t-f-t}
\begin{aligned}
\big\langle \varphi,\Pi_{j}(\xi,\tau)\varphi\big\rangle_{L^{2}(\R^{3}_{+})}&=\big|\big\langle  \varphi,v_{j}(\cdot;\xi,\tau)                    \big\rangle_{L^{2}(\R^{3}_{+})}\big|^{2}\\
&=
2\pi \big|\big\langle  \mathcal{F}_{r\rightarrow -\xi}\big[\big(\mathcal{F}_{s\rightarrow -\tau} \varphi(\cdot,\cdot,t)\big)(-\tau)\big](-\xi), u_{j}(t;\xi)                    \big\rangle_{L^{2}(\R_{+})}\big|^{2}
\end{aligned}
\end{equation}
We state in the next lemma useful properties of the family $\{\Pi_{j}(\xi,\tau)\}_{(j,\xi,\tau)\in \mathbb{N}\times\R^{2}}$.
\begin{lem}\label{eqs-PR}
For all $\varphi\in \mathcal{D}(\mathscr{P}_{0}^{N})$, we have
\begin{equation}\label{eq1-PR}
\big\langle\mathscr{P}_{0}^{N}\Pi_{j}(\xi,\tau)\varphi,\varphi\big\rangle_{L^{2}(\R^{3}_{+})}=(\mu_{j}(\xi)+\tau^{2})\big\langle\Pi_{j}(\xi,\tau)\varphi,\varphi\big\rangle_{L^{2}(\R^{3}_{+})}.
\end{equation}
Moreover, for all $\varphi\in L^2(\R^3_{+})$, one has
\begin{equation}\label{eq2-PR}
\sum_{j}\int_{\R^{2}} \big\langle \varphi,\Pi_{j}(\xi,\tau)\varphi\big\rangle_{L^{2}(\R^{3}_{+})}d\xi d\tau=2\pi\norm{\varphi}^{2}_{L^{2}(\R^{3}_{+})}.
\end{equation}
\end{lem}
\begin{proof}
Let $(\tau,\xi)\in\R^{2}$. By the definition of $v_{j}$ in \eqref{def-vj}, we find,
$$
\begin{array}{rl}
 \mathscr{P}_{0}^{N}v_{j}(r,s,t;\xi,\tau)=(\mu_{j}(\xi)+\tau^{2})v_{j}(r,s,t;\xi,\tau).
\end{array}
$$
The definition of the projectors in \eqref{PR-Pi-j} immediately
gives us \eqref{eq1-PR}.

Using the fact that $u_{j}(\cdot;\xi)$ is an orthonormal basis of $L^{2}(\R_{+})$ for all $\xi\in\R$, we find, using the representation in \eqref{i-t-f-t},
\begin{align*}
 \sum_{j}\big\langle \varphi,\Pi_{j}(\xi,\tau)\varphi\big\rangle_{L^{2}(\R^{3}_{+})}&=
 2\pi\int_{\R_{+}}\Big|  \mathcal{F}_{r\rightarrow -\xi}\big[\big(\mathcal{F}_{s\rightarrow -\tau} \varphi(\cdot,\cdot,t)\big)(-\tau)\big](-\xi)
 \Big|^{2}dt\,.
\end{align*}
Integrating with respect to $\xi$ and $\tau$, and applying
Plancherel's identity twice, we obtain
\begin{align*}
\int_{\R^{2}} \sum_{j}\big\langle \varphi,\Pi_{j}(\xi,\tau)\varphi\big\rangle_{L^{2}(\R^{3}_{+})}d\xi d\tau &=
 2\pi\norm{\varphi}^{2}_{L^{2}(\R^{3}_{+})},
\end{align*}
The proof of the lemma is thus completed.
\end{proof}
We will prove
\begin{thm}\label{value-th=0}
Given $\lambda\in(0,1)$, the following formula holds true\,:
\begin{equation}\label{asym-the=0}
  E(0,\lambda)
  =\dfrac{1}{3\pi^{2}}\int_{0}^{\infty}(\mu_{1}(\xi)-\lambda)_{-}^{3/2}d\xi,
\end{equation}
where $\mu_{1}(\xi) $ is defined in \eqref{muj}.
\end{thm}
\begin{proof}We start by obtaining an upper bound on $E(0,\lambda)$. Let $L>0$. Pick an arbitrary positive integer $N$ and let $\{f_{1},\cdots,f_{N}\}$ be any $L^{2}$ orthonormal set in $\mathcal{D}(\mathscr{P}_{0}^{L})$. 
In view of \eqref{eq1-PR} and \eqref{eq2-PR}, we have the following splitting (recall the domain $\Omega^{L}$ from \eqref{OmegaL}),
\begin{align*}
\Sum{j=1}{N}\big\langle f_{j},(\mathscr{P}^{L}_{0}-\lambda)f_{j}\big\rangle_{L^{2}(\Omega^{L})}
&=\dfrac{1}{2\pi} \Sum{j=1}{N}\Sum{p=1}{\infty}\int_{\R^{2}}(\mu_{p}(\xi)+\tau^{2}-\lambda )
\big\langle f_{j},\Pi_{p}(\xi,\tau)f_{j}\big\rangle_{L^{2}(\R^{3}_{+})} d\xi d\tau\,,
\end{align*}
where we have extended $f_{j}$ by $0$ to $\R^{3}_{+}\setminus\overline{\Omega^{L}}$.
Since $\lambda\in [0,1)$, Lemma~\ref{mu2} gives that $\mu_{p}(\xi)+\tau^{2}> \lambda$ for $p\geq 2$. Hence, we obtain
\begin{align}\label{inq-ng}
\Sum{j=1}{N}\big\langle f_{j},(\mathscr{P}^{L}_{0}-\lambda)f_{j}\big\rangle_{L^{2}(\Omega^{L})}\geq-\dfrac{1}{2\pi}\int_{\R^{2}}
 \big(\mu_{1}(\xi)+\tau^{2}-\lambda \big)_{-}\Sum{j=1}{N}\langle f_{j},\Pi_{1}(\xi,\tau)f_{j}\big\rangle_{L^{2}(\R^{3}_{+})}\,.
\end{align}
Since $\{f_{j}\}_{j=1}^{N}$ is an orthonormal family in $L^{2}(\Omega^{L})$, we deduce that
\begin{equation}\label{eq-L2}
\begin{aligned}
 \Sum{j=1}{N}\big\langle f_{j},\Pi_{1}(\xi,\tau)f_{j}\big\rangle_{L^{2}(\R^{3}_{+})}&=  \Sum{j=1}{N}\big|\langle v_{1}, f_{j}\rangle \big|^{2}\leq  \norm{v_{1}}^2=\dfrac{1}{2\pi} L^{2}.
\end{aligned}
\end{equation}
The last equality comes from the fact that the function $u_{1}(\cdot;\xi)$ is normalized in $L^{2}(\R_{+})$ for all $\xi$.
Substituting \eqref{eq-L2} into \eqref{inq-ng} yields
    \begin{align*}
 \Sum{j=1}{N}\big\langle f_{j},(\mathscr{P}^{L}_{0}-\lambda)f_{j}\big\rangle_{L^{2}(\Omega^{L})}\geq-\dfrac{L^{2}}{4\pi^{2}}\int_{\R^{2}}\big(\mu_{1}(\xi)+\tau^{2}-\lambda \big)_{-}d\xi d\tau,
\end{align*}
Then, on account of Definition~\ref{def-energy} and Lemma~\ref{lem-VP-2}, we have
\[
\dfrac{\mathcal{E}(\lambda;{\mathscr{P}}^{L}_{0},\Omega^{L})}{L^{2}}\leq \dfrac{1}{4\pi^{2}}\int_{\R^{2}}\big(\mu_{1}(\xi)+\tau^{2}-\lambda \big)_{-}d\xi d\tau.
\]
Letting $L\rightarrow\infty$ and using Lemma~\ref{comp-integ}, we arrive at
\begin{equation}\label{up-bnd}
E(0,\lambda)\leq \dfrac{1}{3\pi^{2}}\int_{0}^{\infty}(\mu_{1}(\xi)-\lambda)_{-}^{3/2}d\xi.
\end{equation}
We give the proof of the lower bound on $E(0,\lambda)$. Let $M(\xi,\tau)$ be a function with $0\leq M\leq 1$ and consider the trial density matrix
  \[
  \gamma=  \int_{\R^{2}} M(\xi,\tau) \Pi_{1}(\xi,\tau)d\xi d\tau.
  \]
We will prove that $0\leq\gamma\leq 2\pi$ in the sense of quadratic forms. Consider $g\in L^{2}(\Omega^{L})$. Using that $0\leq M\leq 1$, we have
  \begin{align*}
    0\leq\big\langle g,\gamma g\big\rangle_{L^{2}(\R^{3}_{+})}\leq \int_{\R^{2}}\big|\langle g, v_{1}\rangle\big|^{2}d\xi d\tau
      \leq \Sum{j}{}\int_{\R^{2}}\big|\langle g, v_{j}\rangle\big|^{2}d\xi d\tau=2\pi\norm{g}^{2}.
  \end{align*}
The last step follows from Plancherel's identity and the fact that
$u_{j}(\cdot,\xi)$ is an orthonormal basis of $L^{2}(\R_{+})$.
Recall the quadratic form $Q_{\theta}^{L}$ from \eqref{QL1L2}. It is
easy to check that
\begin{equation}\label{tr-Pi1}
 \mathcal{Q}_{0}^{L}(v_{1})
=\dfrac{L^{2}}{2\pi}(\mu_{1}(\xi)+\tau^{2}).
\end{equation}
We choose $M$ to be the characteristic function of the set
\(
\big\{(\xi,\tau)\in\R^{2}\,:\,\lambda-\mu_{1}(\xi)-\tau^{2}\geq 0\big\}
\). We compute, using \eqref{tr-Pi1},
  \begin{align*}
  {\rm Tr}\big[\big(\mathscr{P}_{0}^{L}-\lambda)\gamma\big]&=\int_{\R^{2}}M(\xi,\tau) \big(\mathcal{Q}_{0}^{L}(v_{1})- \lambda \norm{v_{1}}^{2}_{L^{2}(\Omega^{L})}\big)d\xi d\tau
  = -\dfrac{ L^{2}}{2\pi}\int_{\R^{2}}\big(\mu_{1}(\xi)+ \tau^{2}- \lambda \big)_{-}d\xi d\tau.
  \end{align*}
In view of Lemma~\ref{lem-VP-3}, we get
  \begin{align*}
          - {\rm Tr}\big(\mathscr{P}_{0}^{L}-\lambda \big)_{-}&
           \leq \dfrac{1}{2\pi}   {\rm Tr}\big((\mathscr{P}_{0}^{L}-\lambda)\gamma\big)= -\dfrac{ L^{2}}{4\pi^{2}}\int_{\R^{2}} \big(\mu_{1}(\xi)+ \tau^{2}- \lambda \big)_{-}d\xi d\tau.
  \end{align*}
This gives us
\[
\dfrac{\mathcal{E}(\lambda;{\mathscr{P}}^{L}_{0},\Omega^{L})}{L^{2}}\geq \dfrac{1}{4\pi^{2}}\int_{\R^{2}}  \big(\mu_{1}(\xi)+\tau^{2}-\lambda \big)_{-}d\xi d\tau=\dfrac{1}{3\pi^{2}}\int_{0}^{\infty}(\mu_{1}(\xi)-\lambda)_{-}^{3/2}d\xi.
\]
Letting $L\rightarrow\infty$ yields the desired upper bound.

 \end{proof}
\subsection{$E(\theta,\lambda)$ in the case $\theta\in(0,\pi/2]$}
The purpose of this subsection is to provide an explicit formula for $E(\theta,\lambda)$ in the case $\theta\in(0,\pi/2]$ and $\lambda\in[0,1)$.
However, we have not been able to compute it directly as in the case $\theta=0$.
Our approach is to find an alternative limiting function $F(\theta,\lambda)$ (see \eqref{Fthld} below), which can be constructed and computed explicitly using the eigenprojectors on the eigenfunctions of the 
two-dimensional model operator from \eqref{op-L-th}:
\begin{equation}
\mathcal{L}(\theta)= -\partial_{t}^{2}-\partial_{s}^{2}+ (t\cos(\theta) -s\sin(\theta) )^{2}\qquad {\rm in}\qquad L^{2}(\R^{2}_{+}).
\end{equation}
Let us recall some fundamental spectral properties of $\mathcal{L}(\theta)$ when $\theta\in(0,\pi/2)$ (see~\cite{FH-b} for details and references).
We denote by $\zeta_{1}(\theta)$ the infimum of the spectrum of $\mathcal{L}(\theta)$~:
\begin{equation}\label{zeta-theta}
\zeta_{1}(\theta):=\inf{\rm Spec}(\mathcal L(\theta)).
\end{equation}
The function $(0,\pi/2)\ni\theta\mapsto\zeta_{1}(\theta)$ is monotone increasing and the essential spectrum is  the interval $[1,\infty)$. 
We denote the non-decreasing  sequence of eigenvalues of
$\mathcal{L}(\theta)$ in $(-\infty,1)$ counting multiplicities by
$(\zeta_{j}(\theta))_{j\in \mathbb{N}}$. The associated orthonormal
sequence of eigenfunctions is denoted by
$(u_{\theta,j})_{j\in\mathbb{N}}$ and satisfies,
\begin{equation}\label{Egv-Lth}
\mathcal{L}(\theta)u_{\theta,j}=\zeta_{j}(\theta)u_{\theta,j},\quad
\langle u_{\theta,j},u_{\theta,k}  \rangle_{L^{2}(\R^{2}_{+})}=\delta_{j,k}.
\end{equation}
Using the technique of `Agmon estimates', it is proved in
\cite[Theorem~1.1]{BoDaPoRa} that the eigenfunctions of
$\mathcal{L}(\theta)$ decay exponentially at infinity. For later
use, we record this as
\begin{lem}\label{exp-dec}
Let $\theta\in(0,\pi/2)$. Given $\lambda\in(0,1)$ and
$\alpha\in(0,\sqrt{1-\lambda})$, there exists a positive constant
$C_{\theta,\alpha}$ such that, for any eigen-pair
$(\zeta(\theta),u_{\theta})$ of $\mathcal{L}(\theta)$ with
$\zeta(\theta)<\lambda$, we have
\[
\mathcal{Q}_{\theta}(e^{\alpha\sqrt{t^{2}+s^{2}}}u_{\theta})\leq C_{\alpha,\theta}\norm{u_{\theta}}^{2}_{L^{2}(\R^2_{+})},
\]
where $\mathcal{Q}_{\theta}$ is the quadratic form associated with $\mathcal{L}(\theta)$.
\end{lem}
We recall in the next lemma an upper bound on the number of eigenvalues of $\mathcal{L}(\theta)$ strictly below $1$,
which we will denote by $\mathcal{N}(1;\mathcal{L}(\theta))$. This result is taken from 
\cite{MoTr} and will be significant  in the calculations that we
will carry later.
\begin{lem}{\label{Nb-Ev-Finite}}
Let $\theta\in(0,\pi/2]$. There exists a constant $C$ such that
\[
 \mathcal{N}(1;\mathcal{L}(\theta))\leq \dfrac{C}{\sin(\theta)}.
\]
\end{lem}
Next, we define the function $\R^{3}_{+}\ni (r,s,t)\mapsto v_{\theta,j}(r,s,t;\xi)$ by
\begin{equation}\label{vjth}
v_{\theta, j}(r,s,t;\xi)=\dfrac{1}{\sqrt{2\pi}} e^{i\xi r}u_{\theta,j}\Big(s-\frac{\xi}{\sin(\theta)},t\Big),
\end{equation}
where $\{u_{\theta,j}\}_{j}$ are the eigenfunctions from \eqref{Egv-Lth}. We define the projectors $\pi_{\theta, j}$ by
\begin{equation}
\big(\pi_{\theta, j}(\xi)\big)\varphi(s_{1},t_{1})=u_{\theta, j}\Big(s_{1}-\frac{\xi}{\sin(\theta)},t_{1}\Big)\int_{\R^{2}_{+}} \overline{u_{\theta, j}\Big(s_{2}-\frac{\xi}{\sin(\theta)},t_{2}\Big)}\varphi(s_{2},t_{2})ds_{2}dt_{2}.
\end{equation}
We then introduce a family of operators $\Pi_{\theta,j}$ defined by
\begin{multline}\label{Pr-th-j-0}
L^{2}(\R_{+}^{3})\ni f\mapsto\Pi_{\theta,j}f(r_{1},s_{1},t_{1})\\
=\int_{\R}v_{\theta, j}(r_{1},s_{1},t_{1};\xi)\bigg\{\int_{\R^{3}_{+}}\overline{ v_{\theta, j}(r_{2},s_{2},t_{2};\xi)} {f(r_{2},s_{2},t_{2})}dr_{2}ds_{2}dt_{2}\bigg\}d\xi
\end{multline}
In terms of quadratic forms, we have,
\begin{equation}\label{Pr-th-j}
\begin{aligned}
L^{2}(\R_{+}^{3})\ni f\mapsto \langle \Pi_{\theta,j}f, f\rangle_{L^{2}(\R^{3}_{+})}
&=\int_{\R}\Big|\langle v_{\theta, j}(\cdot,\xi),f \rangle_{L^{2}(\R^{3}_{+})}\Big|^{2}d\xi\\
& = \int_{\R}\Big\langle
\mathcal{F}^{-1}_{\xi\rightarrow r}\big(\pi_{j,\theta}(\xi)(\mathcal{F}_{r\rightarrow\xi}f(\cdot,s,t)(\xi))\big)(r), f(r,s,t) \Big\rangle_{L^{2}(\R^{3}_{+})}d\xi.
\end{aligned}
\end{equation}
Since the Fourier transform is a unitary transform and
$\pi_{\theta,j}(\xi)$ is a projection, it is easily to be seen that
the operator $\Pi_{\theta,j}$ is a projection too.
The following Lemma illustrates relevant properties of the family of projectors $\{\Pi_{\theta,j}\}_{j}$.
\begin{lem}{\label{computations}} For all $f\in \mathcal{D}({\mathscr{P}_{\theta}^{N}})$, we have
\begin{equation}\label{p-pi-th-j}
\big\langle\mathscr{P}_{\theta}^{N}\Pi_{\theta,j}f,f\big\rangle_{L^{2}(\R^{3}_{+})}
=\zeta_{j}(\theta) \big \langle\Pi_{\theta,j}f,f\big \rangle_{L^{2}(\R^{3}_{+})},
\end{equation}
and for all $f\in L^2(\R^3_{+})$, one has
\begin{align}\label{proj-sum-j-Big-Pi}
\Big\langle\sum_{j} \Pi_{\theta,j}f,f \Big\rangle_{L^{2}(\R^{3}_{+})}
\leq \norm{f}_{L^{2}(\R^{3}_{+})}^{2}.
\end{align}
Moreover, for any smooth cut-off function $\chi\in C^{\infty}_{0}(\R^{2})$, it holds true that
\begin{equation}\label{tr-pi-th-j-1}
{\rm Tr}(\chi \Pi_{\theta,j}\chi)= \frac{\sin(\theta)}{2\pi}\int_{\R^{2}}\chi^{2}(r,s)drds.
\end{equation}
\end{lem}
\begin{proof}
Applying the operator ${\mathscr P}_{\theta}^{N}$ to the function $v_{\theta,j}$, we find
\begin{align*}
\mathscr{P}_{\theta}^{N}v_{\theta, j}(r,s,t;\xi)
&= \zeta_{j}(\theta)v_{\theta, j}(r,s,t;\xi).
\end{align*}
The assertion \eqref{p-pi-th-j} then follows from the definition of $\Pi_{\theta,j}$ in \eqref{Pr-th-j-0}.

To prove \eqref{proj-sum-j-Big-Pi}, we rewrite \eqref{Pr-th-j} as
\begin{align}\label{Sum-Pi}
\Big\langle \Pi_{\theta, j}f,f \Big\rangle_{L^{2}(\R^{3}_{+})}
&= \int_{\R}\Big\langle \pi_{\theta,j}(\xi) \big(\mathcal{F}_{r\rightarrow\xi}f(\cdot,s,t)\big)(\xi), \big(\mathcal{F}_{r\rightarrow\xi}f(\cdot,s,t)(\xi)\big) \Big\rangle_{L^{2}(\R^{2}_{+})}d\xi.
\end{align}
It can be easily shown that $\sum_{j}\pi_{\theta,j}$ is a projection. Hence, by Plancherel's identity, we see that
\begin{align*}
\Big\langle\sum_{j} \Pi_{\theta,j}f,f \Big\rangle_{L^{2}(\R^{3}_{+}}&\leq  \int_{\R^{2}_{+}}\int_{\R}| \mathcal{F}_{r\rightarrow\xi}f(\cdot,s,t)|^{2}d\xi dsdt= \int_{\R^{3}_{+}}|f(r,s,t)|^{2}drdsdt.
\end{align*}
We come to the proof of \eqref{tr-pi-th-j-1}. For this, we notice that
\begin{align}
\label{Tr-chi-Pi-chi}
{\rm Tr}\big(\chi\Pi_{\theta,j}\chi\big)
\nonumber&=\dfrac{1}{2\pi} \int_{\R^{3}_{+}}\chi^{2}(r,s)\Big(\int_{\R} \big|e^{i\xi r} u_{\theta,j}\big(s-\frac{\xi}{\sin(\theta)},t\big)\big|^{2}d\xi \Big)drdsdt.
\end{align}
Performing the change of variable $\nu=s-\frac{\xi}{\sin\theta}$ and using that the functions $\{u_{\theta,j}\}_{j}$ are normalized, we get
\begin{equation}\label{calcul-kern}
\begin{aligned}
{\rm Tr}\big(\chi\Pi_{\theta, j}\chi\big)&=\dfrac{1}{{2\pi}}\sin(\theta) \int_{\R^{2}}\chi^{2}(r,s) dr ds\int_{\R^{2}_{+}} |u_{\theta,j}(\nu,t)|^{2}d\nu dt= \dfrac{1}{{2\pi}}\sin(\theta) \int_{\R^{2}}\chi^{2}(r,s)drds.
\end{aligned}
\end{equation}
Thereby completing the proof of the lemma.
\end{proof}
Let $a>0$. In order to define $F(\theta,\lambda)$ below, we need to
introduce the cut-off function $\chi_{a}\in C^{\infty}_{0}(\R^{2})$,
which satisfies
\begin{equation}
0\leq \chi_{a}\leq 1,~\mbox{in}~\R^{2},\quad \mbox{supp}~\chi_{a}\in \Big(-\dfrac{1+a}{2}, \dfrac{1+a}{2}\Big)^{2},
\quad \chi_{a}=1\quad{\rm in}~ \Big(-\dfrac{1}{2},\dfrac{1}{2}\Big)^{2},~|\nb \chi_{a}|\leq Ca^{-1}.
\end{equation}
Let $L>0$. We set
\begin{equation}\label{def-chiL}
\chi_{a,L}(r,s)=\chi_{a}\bigg(\dfrac{r}{L},\dfrac{s}{L}\bigg), \quad (r,s)\in\R^{2},
\end{equation}
and,
\begin{equation}\label{def-mu}
\mu_{a}=\int_{\R^{2}}\chi_{a}^{2}(r,s)drds.
\end{equation}
Notice that $\mu_a$ satisfies
\begin{equation}\label{eq:mua}
\mu_a=1+\mathcal O(a)\,,\quad (a\to0_+)\,.
\end{equation}
We define $F(\theta,\lambda;a)$ to be
\begin{equation}\label{Fthld}
F(\theta,\lambda;a):=\limsup_{L\rightarrow\infty}\dfrac{{\rm Tr}\big(\chi_{a,L} ({\mathscr{P}}_{\theta}^{N}-\lambda)\chi_{a,L}\big)_{-}}{L^{2}}\,,
\end{equation}
where ${\mathscr{P}}_{\theta}^{N}$ is the self-adjoint operator given in \eqref{Neum-r3+}. We now formulate the main theorem of this section.
\begin{thm}{\label{thm-exp-for-Eth}} Let $\theta\in(0,\pi/2]$, $\lambda\in[0,1)$ and $E(\theta,\lambda)$ be as introduced in \eqref{trm-lim}.
We have the following explicit formula of $E(\theta,\lambda)$
\begin{equation}\label{E-th-ld-th-n-z}
E(\theta,\lambda)= \dfrac{1}{2\pi}\sin(\theta) \sum_{j}(\zeta_{j}(\theta)-\lambda)_{-},
\end{equation}
where the $\{\zeta_{j}(\theta)\}_{j}$ are the eigenvalues from \eqref{Egv-Lth}.
\end{thm}
The proof of Theorem \ref{thm-exp-for-Eth} is splitted into two
lemmas.
\begin{lem}\label{lem-exca}
Let $a\in(0,1)$, $\lambda\in[0,1)$ and $\theta\in(0,\pi/2]$. The
following limit exists, $$\lim_{L\rightarrow\infty}\dfrac{{\rm
Tr}\big(\chi_{a,L}
({\mathscr{P}}_{\theta}^{N}-\lambda)\chi_{a,L}\big)_{-}}{L^{2}}\,,$$
and its value is  $F(\theta,\lambda;a)$ introduced in \eqref{Fthld}.
Furthermore, there holds,
\begin{equation}
F(\theta,\lambda;a)= \dfrac{1}{2\pi}\mu_{a}\sin(\theta) \sum_{j}(\zeta_{j}(\theta)-\lambda)_{-},
\end{equation}
where $\mu_{a}$ is the constant defined in \eqref{def-mu}.
\end{lem}
\begin{proof}
Let ${\mathscr{P}}_{\theta}^{N}$ be the self-adjoint operator given in \eqref{Neum-r3+},
 and let $\{g_{1},\cdots, g_{N}\}$ be any orthonormal set in $\mathcal{D}(\mathscr{P}_{\theta}^{N})$. It follows from Lemma~\ref{computations} that
\begin{equation}\label{P-OS}
\sum_{k=1}^{N}\big\langle \chi_{a,L}g_{k}, (\mathscr{P}_{\theta}^{N}-\lambda)\chi_{a,L}g_{k} \big\rangle_{L^{2}(\R^{3}_{+})}
\geq- \sum_{j}(\zeta_{j}(\theta)-\lambda)_{-}\sum_{k=1}^{N}\big\langle \chi_{a,L}g_{k}, \Pi_{\theta,j}\chi_{a,L}g_{k}\big \rangle_{L^{2}(\R^{3}_{+})}.
\end{equation}
Since $\{g_{k}\}_{k=1}^{N}$ is an orthonormal family in $L^{2}(\R_{+}^{3})$ and performing a similar calculation to that in\eqref{calcul-kern}, we deduce that
\begin{align*}
\sum_{k=1}^{N}\big\langle \chi_{a,L}g_{k}, \Pi_{\theta, j}\chi_{a,L}g_{k} \big\rangle_{L^{2}(\R^{3}_{+})}
&\leq \int_{\R}\sum_{k=1}^{N}|\langle g_{k},v_{j,\theta}\chi_{a,L}\rangle_{L^{2}(\R^{3}_{+})}\big|^{2}d\xi\leq  \frac{1}{2\pi}\mu_{a}L^{2}\sin(\theta).
\end{align*}
Implementing this in \eqref{P-OS}, we obtain
\begin{equation}\label{Eq-vp-ub}
\sum_{k=1}^{N}\big\langle \chi_{a,L}g_{k}, (\mathscr{P}_{\theta}^{N}-\lambda)\chi_{a,L}g_{k} \big\rangle_{L^{2}(\R^{3}_{+})}\geq-\dfrac{1}{2\pi} L^{2}\mu_{a} \sin(\theta)\sum_{j}(\zeta_{j}(\theta)-\lambda)_{-}.
\end{equation}
By the variational principle in Lemma~\ref{lem-VP-2}, we find
\begin{equation}
{\rm Tr}\big(\chi_{a,L} ({\mathscr{P}}_{\theta}^{N}-\lambda\big)\chi_{a,L})_{-}
\leq \dfrac{1}{2\pi}  L^{2}\mu_{a} \sin(\theta)\sum_{j}(\zeta_{j}(\theta)-\lambda)_{-}.
\end{equation}
Dividing by $L^{2}$ on both sides, we get after passing to the limit $L\rightarrow \infty$,
\begin{equation}\label{ub-ethld}
\limsup_{L\rightarrow\infty}\dfrac{{\rm Tr}\big(\chi_{a,L} ({\mathscr{P}}_{\theta}^{N}-\lambda\big)\chi_{a,L})_{-}}{L^{2}}\leq\dfrac{1}{2\pi} \mu_{a} \sin(\theta)\sum_{j}(\zeta_{j}(\theta)-\lambda)_{-}.
\end{equation}
To prove a lower bound, 
we consider the density matrix
\begin{equation}
\gamma= \sum_{{\{ j\,:\,\zeta_{j}(\theta)\leq \lambda\}}} \Pi_{\theta,j}.
\end{equation}
It is easy to see that $\gamma\geq 0$, and in view of \eqref{proj-sum-j-Big-Pi}, it follows that
\begin{equation}
\langle f, \gamma f\rangle_{L^{2}(\R^{3}_{+})} \leq \norm{f}^{2}_{L^{2}(\R^{3}_{+})}.
\end{equation}
Next, by Cauchy-Schwarz inequality, we have,
\begin{multline}\label{Eq111}
{\rm Tr}\big(\chi_{a,L}(\mathscr{P}_{\theta}^{N}-\lambda)\chi_{a,L}\Pi_{\theta,j}\big)= \int_{\R}\int_{\R_{+}^{3}}\Big(|(-i\nabla + {\bf F}_{\theta})\chi_{a,L} v_{\theta, j}|^{2} -\lambda |\chi_{a,L}v_{\theta,j}|^{2}\Big)drdsdt d\xi\\
\leq \int_{\R}\int_{\R_{+}^{3}}\Big\{\chi_{a,L}^{2}(r,s)\Big(|(-i\nabla + {\bf F}_{\theta})v_{\theta, j}|^{2} - \lambda |v_{\theta,j}|^{2}\Big)
+|\nabla \chi_{a,L}(r,s)|^{2}|v_{\theta,j}|^{2}\Big\}drdsdt d\xi.
\end{multline}
Performing the change of variable $\nu=s-\frac{\xi}{\sin(\theta)}$ in \eqref{Eq111} and using that the function $u_{j,\theta}$ is normalized, we arrive at
\begin{multline}
 {\rm Tr}\big(\chi_{a,L}(\mathscr{P}_{\theta}^{N}-\lambda)\chi_{a,L}\gamma\big)
 \leq\dfrac{\sin(\theta)}{2\pi}\sum_{\{j\,:\,\zeta_{j}(\theta)\leq \lambda\}}  \Big\{(\zeta_{j}(\theta)-\lambda) \mu_{a} L^{2}
 +\int_{\R^{2}}|\nabla\chi_{a}(r,s)|^{2}drds\Big\},
\end{multline}
where $\mu_{a}$ is the constant from \eqref{def-mu}. Dividing both
sides of the aforementioned inequality by $L^{2}$, we see that,
 \begin{multline}\label{eq100}
\dfrac{ {\rm Tr}\big(\chi_{a,L}(\mathscr{P}_{\theta}^{N}-\lambda)\chi_{a,L}\gamma\big)}{L^{2}}
\leq \dfrac{1}{2\pi} \sin(\theta)\sum_{\{j\,:\,\zeta_{j}(\theta)\leq \lambda\}}  \Big\{\mu_{a}(\zeta_{j}(\theta) -\lambda)+ L^{-2}\int_{\R^{2}}|\nabla \chi_{a}|^{2}drds\Big\}.
\end{multline}
Here we point out that the number $\mathcal{N}(\lambda;\mathcal{L}(\theta),\R^{2}_{+})$ is controlled by $C/\sin(\theta)$ according to Lemma~\ref{Nb-Ev-Finite}. From Lemma~\ref{lem-VP-3}, it follows that
\begin{multline*}
-\dfrac{ {\rm Tr}\big(\chi_{a,L}(\mathscr{P}_{\theta}^{N}-\lambda)\chi_{a,L}\big)_{-}}{L^{2}}
\leq -\dfrac{1}{2\pi}\mu_{a} \sin(\theta)\sum_{j}  (\zeta_{j}(\theta)-\lambda)_{-}
+C(2\pi)^{-1} L^{-2}\int_{\R^{2}}|\nabla \chi_{a}|^{2}drds.
\end{multline*}
Taking  $\displaystyle\liminf_{L\rightarrow\infty}$, we deduce that
\begin{equation*}
\liminf_{L\rightarrow\infty}\dfrac{ {\rm Tr}\big(\chi_{a,L}(\mathscr{P}_{\theta}^{N}-\lambda)\chi_{a,L}\big)_{-}}{L^{2}}
\geq \dfrac{1}{2\pi}\mu_{a} \sin(\theta)\sum_{j}  ( \zeta_{j}(\theta)-\lambda)_{-}.
\end{equation*}
This together with \eqref{ub-ethld} and the definition of
$F(\theta,\lambda;a)$ in \eqref{Fthld}, finish the proof of the
Lemma.
 \end{proof}
 Our next goal is to establish a connection between the function $F(\theta,\lambda)$ obtained in Lemma~\ref{lem-exca} and $E(\theta,\lambda)$ from \eqref{trm-lim}.
 \begin{thm}\label{th-rewith-E}
Let $a\in(0,1)$, $\theta\in(0,\pi/2]$ and $\lambda\in[0,1)$. There
holds,
 \[
 \mu_a E(\theta,\lambda)\leq F(\theta,\lambda;a)\leq (1+a)^2  E(\theta,\lambda)\,,
 \]
 where $\mu_{a}$ is the constant from \eqref{def-chiL}.
 \end{thm}
 \begin{proof}
Let $L\gg \ell\gg 1$. We consider the domain
 \begin{equation*}
 \Omega_{j,k,\ell}=(j \ell,(j+1)\ell)\times (k\ell,(k+1)\ell)\times\R_{+}, \qquad (j,k)\in\mathbb{Z}^{2}.
 \end{equation*}
 We will denote by $\mathscr{N}_{B}$ the number of boxes of the form $\Omega_{j,k,\ell}$ intersecting ${\rm supp}\,\chi_{a,L}$\,:
 \begin{equation}
 \mathscr{N}_{B}=\sharp \{(j,k)\in \mathbb{Z}^{2}\,:\, {\rm supp}\,\chi_{a,L}\cap \Omega_{j,k,\ell}\neq \emptyset\}.
 \end{equation}
Recall the magnetic potential ${\bf F}_{\theta}$ defined in
\eqref{Ftheta}. Consider the self-adjoint operator
$\mathscr{P}_{\theta,j,k}^{\ell}$ defined  by the closed quadratic
form
\[
\mathcal{Q}_{\theta,j,k}^{\ell}(u)
=\int_{\Omega_{j,k,\ell}}|(-i\nb+{\bf F}_{\theta})u|^{2} drdsdt,
\]
with domain,
\begin{multline}
\mathcal{D}(\mathcal{Q}_{\theta,j,k}^{\ell})=\Big\{u\in L^{2}(\Omega_{j,k,\ell})\,:\, (-i\nb+{\bf F}_{\theta})u\in L^{2}(\Omega_{j,k,\ell}) ,\\
\quad u(j\ell,\cdot,\cdot)=u((j+1)\ell,\cdot,\cdot)=0,
u(\cdot,k\ell,\cdot)=u(\cdot, (k+1)\ell,\cdot)=0 \Big\}.
\end{multline}
Since any function that belongs to the form domain $\bigoplus_{j,k}\mathcal{D}(\mathcal{Q}_{\theta,j,k}^{\ell})$ lies in the form domain $\mathcal{D}(\mathcal{Q}_{\theta}^{N})$ and the values of both quadratic forms coincide for such a function, we have the operator inequality
 \begin{equation}
 \chi_{a,L}(\mathscr{P}_{\theta}^{N}-\lambda)\chi_{a,L}\leq \bigoplus_{j,k}\chi_{a,L}(\mathscr{P}_{\theta,j,k}^{\ell}-\lambda)\chi_{a,L}.
 \end{equation}
 It follows that
 \begin{equation}\label{op-inq}
 -{\rm Tr}(\chi_{a,L}(\mathscr{P}_{\theta}^{N}-\lambda)\chi_{a,L})_{-}\leq - \sum_{j,k}{\rm Tr}(\chi_{a,L}(\mathscr{P}_{\theta,j,k}^{\ell}-\lambda)\chi_{a,L})_{-}.
 \end{equation}
 Let $S_{jk}\in \N$ be the number of eigenvalues of ${\mathscr{P}^{\ell}_{\theta,j,k}}$ that are below $\lambda$, denoted by $\{\lambda_{m}\}_{m=1}^{S_{jk}}$, and let $\{f_{s}\}_{m=1}^{S_{jk}}\in\mathcal{D}(\mathcal{Q}_{\theta,j,k}^{\ell})$ be associated (normalized) eigenfunctions. We consider the density matrix
 \[
 \gamma_{j,k} f=\sum_{m=1}^{S_{jk}}\langle f,f_{m}\rangle f_{m}, \quad f\in L^{2}(\Omega_{jk}^{\ell}).
 \]
We compute
\begin{multline}\label{tr-chi-Pth-chi}
{\rm Tr}\big(\chi_{a,L}(\mathscr{P}_{\theta,j,k}^{\ell}-\lambda)\chi_{a,L}\gamma_{j,k}\big)= \sum_{m=1}^{S_{jk}} (\mathcal{Q}_{\theta,j,k}^{\ell}(\chi_{a,L}f_{m})-\lambda\|\chi_{a,L}f_{m}\|^{2}) \\
=\sum_{m=1}^{S_{jk}} \Big\{(\lambda_{m}-\lambda)\| \chi_{a,L}f_{m}\|^{2}+\||\nb\chi_{a,L}|f_{m}\|^{2} \Big\},
\end{multline}
where the last step follows since $\{f_{m}\}_{m=1}^{S_{jk}}$ are eigenfunctions.

Let us denote by $x_{j,k,\ell}^{\star}$ the point belonging to the interval $(j\ell,(j+1)\ell)\times(k\ell,(k+1)\ell)$ where the function $\chi_{a,L}^{2}$ attains its minimum\,:
 \[
\chi_{a,L}^{2}(x_{j,k,\ell}^{\star})=\min_{(r,s)\in(j\ell,(j+1)\ell)\times(k\ell,(k+1)\ell) }\chi_{a,L}^{2}(r,s).
 \]
It follows that
\begin{equation}\label{tr-chi-pth-chi-gamma-ij}
{\rm Tr}\big(\chi_{a,L}(\mathscr{P}_{\theta,j,k}^{\ell}-\lambda)\chi_{a,L}\gamma_{j,k}\big)
\leq     \chi_{a,L}^{2}(x_{j,k,\ell}^{\star}) \sum_{m=1}^{S_{jk}}     (\lambda_{m}-\lambda)+\sum_{m=1}^{S_{jk}}    \||\nb\chi_{a,L}|f_{m}\|^{2},
\end{equation}
where have used that the term $\sum_{m=1}^{S_{jk}}(\lambda_{m}-\lambda)$ is negative. Inserting this into \eqref{tr-chi-pth-chi-gamma-ij} and using the bound $|\nb \chi_{a,L}|\leq {C}(aL)^{-1}$, we find
\begin{equation*}
{\rm Tr}\big(\chi_{a,L}(\mathscr{P}_{\theta,j,k}^{\ell}-\lambda)\chi_{a,L}\gamma_{j,k}\big)
\leq      \chi_{a,L}^{2}(x_{j,k,\ell}^{\star})  \sum_{m=1}^{S_{jk}}    (\lambda_{m}-\lambda)+ CS_{jk}(a L)^{-2}.
\end{equation*}
By \eqref{bnd-nb}, we have $S_{jk}\leq C\ell^{2}$. Using \eqref{bnd-eng-tl}, we obtain
\begin{equation}
{\rm Tr}\big(\chi_{a,L}(\mathscr{P}_{\theta,j,k}^{\ell}-\lambda)\chi_{a,L}\gamma_{j,k}\big)
\leq\chi_{a,L}^{2}(x_{j,k,\ell}^{\star})(- E(\theta,\lambda) +C\ell^{-2/3})\ell^{2}
+C\ell^{2}(a L)^{-2}.
\end{equation}
By \eqref{op-inq} and Lemma~\ref{lem-VP-3}, it follows that
 \begin{equation} \label{Eq-TR}
 -{\rm Tr}\big(\chi_{a,L}(\mathscr{P}_{\theta}^{N}-\lambda)\chi_{a,L}\big)_{-} 
 \leq \Big\{\sum_{j,k}\chi_{a,L}^{2}(x_{j,k,\ell}^{\star})\ell^{2}\Big\}(- E(\theta,\lambda) +C\ell^{-2/3})
 +C\mathscr{N}_{B}\ell^{2}(aL)^{-2}.
 \end{equation}
 The sum $\sum_{j,k}\chi_{a,L}^{2}(x_{j,k,\ell}^{\star})\ell^{2}$ is a (lower) Riemannian sum. Thus, we have
\begin{align*}
\Big| \sum_{j,k}\chi_{a,L}^{2}(x_{j,k,\ell}^{\star})\ell^{2}-\int_{\R^{2}}\chi_{a,L}^{2}(r,s)drds(=\mu_{a}L^{2})\Big| \leq C\ell L.
\end{align*}
Substituting this into \eqref{Eq-TR}, we obtain
\begin{equation*}
 -{\rm Tr}\big(\chi_{a,L}(\mathscr{P}_{\theta}^{N}-\lambda)\chi_{a,L}\big)_{-}
 \leq  E(\theta,\lambda) L^{2}\big(-\mu_{a}+C\ell L^{-1}\big) +C\ell^{-2/3}(\mu_{a}L^{2}+C\ell L)+  C\mathscr{N}_{B}\ell^{2} (aL)^{-2}.
 \end{equation*}
Dividing both sides by $L^{2}$, we get
\begin{equation*}
 -\dfrac{{\rm Tr}\big(\chi_{a,L}(\mathscr{P}_{\theta}^{N}-\lambda)\chi_{a,L}\big)_{-}}{L^{2}}
 \leq  E(\theta,\lambda)\big(-\mu_{a}+C\ell L^{-1}\big)
+C\ell^{-2/3}(\mu_{a}+C\ell L^{-1})+  C\mathscr{N}_{B}\ell^{-2}a^{-2} L^{-4}.
 \end{equation*}
We make the following choice of $\ell$,
$$
 \ell = L^{\eta},\quad \eta<1.$$
Since $\mathscr{N}_{B}\sim ((1+a))^{2}L^{2}\ell^{-2}$ as
$L\rightarrow\infty$, we get, after taking
$\displaystyle\liminf_{L\rightarrow\infty}$, the following lower
bound,
 \begin{equation}\label{upb}
 \liminf_{L\rightarrow\infty}  \dfrac{{\rm Tr} \big(\chi_{a,L }(\mathscr{P}_{\theta}^{N}-\lambda)\chi_{a,L}\big)_{-}}{L^{2}}\geq E(\theta,\lambda)\mu_{a}.
 \end{equation}
It remains to prove the upper bound. Using Lemma~\ref{lem-VP-3} and that the trace is cyclic, we see that
 \begin{align*}
 -{\rm Tr}\big(\chi_{a,L}(\mathscr{P}_{\theta}^{N}- \lambda)\chi_{a,L}\big)_{-}
 &=\inf_{0\leq\gamma\leq 1} {\rm Tr}\big((\mathscr{P}_{\theta}^{N}- \lambda)\chi_{a,L}\gamma\chi_{a,L}\big).
  \end{align*}
 Since the function $\chi_{a,L}$ is supported in $\big(-\frac{(1+a)L}{2},\frac{(1+a)L}{2}\big)^{2}$, it follows that
 \begin{equation}
 -{\rm Tr}\big(\chi_{a,L}(\mathscr{P}_{\theta}^{N}- \lambda)\chi_{a,L}\big)_{-}
 \geq \inf_{0\leq\widetilde\gamma\leq 1} {\rm Tr}\big[(\mathscr{P}_{\theta}^{(1+a)L}- \lambda)\widetilde\gamma\big]
\geq - ((1+a)L)^{2}E(\theta,\lambda).
  \end{equation}
By taking $\displaystyle\limsup_{L\to\infty}$,  this yields that,
  \begin{equation}\label{lob}
   \limsup_{L\rightarrow\infty}  \dfrac{{\rm Tr} \big(\chi_{a,L}(\mathscr{P}_{\theta}^{N}-\lambda)\chi_{a,L}\big)_{-}}{L^{2}}\leq E(\theta,\lambda)(1+a)^{2}.
  \end{equation}
 Combining \eqref{upb} and \eqref{lob},  we obtain
   \begin{multline}
 \mu_{a} E(\theta,\lambda)\leq \liminf_{L\rightarrow\infty}  \dfrac{{\rm Tr} \big(\chi_{a,L}(\mathscr{P}_{\theta}^{N}-\lambda)\chi_{a,L}\big)_{-}}{L^{2}}
 \\
 \leq \limsup_{L\rightarrow\infty}  \dfrac{{\rm Tr} \big(\chi_{a,L}(\mathscr{P}_{\theta}^{N}-\lambda)\chi_{a,L}\big)_{-}}{L^{2}}\leq (1+a)^2E(\theta,\lambda)\,.
  \end{multline}
%
%
Recalling the definition of $F(\theta,\lambda)$ in \eqref{Fthld} and Lemma~\ref{lem-exca} finishes the proof of Lemma~\ref{th-rewith-E}.
 \end{proof}
 \begin{proof}[Proof of Theorem~\ref{thm-exp-for-Eth}]
The proof follows easily by combining the results of
Lemmas~\ref{lem-exca},~\ref{th-rewith-E}, and by sending the
parameter $a$ to $0_+$.
 \end{proof}

\subsection{Dilatation}
 Let us define the unitary operator
 \begin{equation}\label{unit-tran}
  U_{h,b}:L^{2}(\R^{3}_{+})\ni u\mapsto U_{h,b}u(z)= h^{-3/4}b^{3/4}u(h^{-1/2}b^{1/2}z) \in L^{2}(\R^{3}_{+}),
\end{equation}
 Let $h,b>0$ and $\theta\in[0,\pi/2]$. We introduce the self-adjoint operator
\begin{equation}\label{Phb-th}
\mathscr{P}_{\theta, h,b}^{N}= (-ih\nabla +b{\bf F}_{\theta})^{2},\quad \mbox{in}\quad L^{2}(\R^{3}_{+}),
\end{equation}
with Neumann boundary conditions at $t=0$.
With $\mathscr{P}_{\theta}^{N}$ being the operator from \eqref{Neum-r3+}, it is easy to check that
\begin{equation}\label{ch-hb-op}
\mathscr{P}_{\theta, h,b}^{N}=hb U_{h,b}\mathscr{P}_{\theta}^{N}U_{h,b}^{-1}.
\end{equation}
For $j\in\mathbb{N}$ and $(\xi,\tau)\in\R^{2}$, we introduce the family of projectors
\begin{equation}\label{Pr-resc-th0}
\Pi_{j}(\xi,\tau;h,b)=U_{h,b}\Pi_{j}(\xi,\tau)U_{h,b}^{-1}
\end{equation}
and, for $\theta\in(0,\pi/2]$,
\begin{equation}\label{Pr-theta-j}
\Pi_{\theta,j}(h,b)=U_{h,b}\Pi_{\theta,j}U_{h,b}^{-1}
\end{equation}
where, $\Pi_{j}(\xi,\tau)$ and $\Pi_{\theta,j}$ are introduced in \eqref{PR-Pi-j} and \eqref{Pr-th-j} respectively. 
We deduce the following two generalizations of Lemma \ref{eqs-PR} and Lemma \ref{computations}.
 \begin{lem}\label{eqs-PR-cor}
For all $\varphi\in\mathcal{D}(\mathscr{P}_{0,h,b}^{N})$, we have
\begin{equation}\label{eq1-PR-cor}
\big\langle \mathscr{P}_{0,h,b}^{N}(\Pi_{j}(\xi,\tau;h,b)\varphi), \varphi\big\rangle_{L^{2}(\R^{3}_{+})} =hb(\mu_{j}(\xi)+\tau^{2})\big\langle \Pi_{j}(\xi,\tau;h,b)\varphi,\varphi\big\rangle_{L^{2}(\R^{3}_{+})},
\end{equation}
and for all $f\in L^2(\R^3_{+})$,
\begin{equation}\label{eq2-PR-cor}
\sum_{j}\int_{\R^{2}} \big\langle \varphi,\Pi_{j}(\xi,\tau;h,b)\varphi\big\rangle_{L^{2}(\R^{3}_{+})}d\xi d\tau=2\pi \norm{\varphi}^{2}_{L^{2}(\R^{3}_{+})}.
\end{equation}
Moreover, for any smooth cut-off function $\chi\in C^{\infty}_{0}(\R^{2})$, it holds true that
\begin{equation}\label{tr-pi-0-j-cor}
{\rm Tr}(\chi \Pi_{j}(\xi,\tau;h,b)\chi)= bh^{-1}(2\pi)^{-1}\int_{\R^{2}}\chi^{2}(r,s)drds.
\end{equation}
\end{lem}
 \begin{lem}{\label{computations-cor}} Let $f\in\mathcal{D}(\mathscr{P}_{\theta,h,b}^{N})$, we have
\begin{equation}\label{p-pi-th-j-cor}
\big\langle\mathscr{P}_{\theta,h,b}^{N}\Pi_{\theta,j}(h,b)f,f\big\rangle_{L^{2}(\R^{3}_{+})}
= hb \zeta_{j}(\theta)\big\langle \Pi_{\theta,j}(h,b)f,f\big\rangle_{L^{2}(\R^{3}_{+})},
\end{equation}
and for all $f\in L^2(\R^{3}_{+})$,
\begin{align}\label{proj-sum-j-Big-Pi-cor}
\Big\langle\sum_{j} \Pi_{\theta,j}(h,b)f,f \Big\rangle_{L^{2}(\R^{3}_{+})}
\leq \norm{f}_{L^{2}(\R^{3}_{+})}^{2}.
\end{align}
Moreover, for any smooth cut-off function $\chi\in C^{\infty}_{0}(\R^{2})$, it holds true that
\begin{equation}\label{tr-pi-th-j-cor}
{\rm Tr}(\chi \Pi_{\theta,j}(h,b)\chi)= bh^{-1}(2\pi)^{-1}{\sin(\theta)}\int_{\R^{2}}\chi^{2}(r,s)drds.
\end{equation}
\end{lem}
%
%
\section{Boundary coordinates} \label{bndcoord}
\subsection{Local coordinates}
We denote the standard coordinates on $\R^{3}$ by $x=(x_{1},x_{2},x_{3})$. The standard Euclidean metric is given by
      \begin{equation}\label{stan-metr}
     g_{0}=dx_{1}^{2}+dx_{2}^{2}+dx_{3}^{2}.
     \end{equation}
We introduce a system of coordinates valid near a point of the boundary. These coordinates are used in \cite{HeMo2} and then in \cite{RA1} in order to estimate the ground state energy of a magnetic Schr\"odinger operator with large magnetic field (or with small semi-classical parameter).
Consider a point $x_{0}\in\partial\Omega$. 
Let $\mathcal{V}_{x_{0}}$ be a neighbourhood of $x_{0}
$ such that there exist local boundary coordinates $(r,s)$ in $W=\mathcal{V}_{x_{0}}\cap\partial\Omega$, i.e., there exist an open subset $U$ of $\R^{2}$ and a
diffeomorphism $\phi_{x_{0}}:W\rightarrow U,\, \phi_{x_{0}}(x)=(r,s)$, such that $\phi_{x_0}(x_0)=0$ and $D\phi_{x_0}(x_0)={\rm Id}_2$ where ${\rm Id}_2$ is the $2\times 2$ identity matrix. Then for $t_{0}>0$ small enough, we define the coordinate transformation $\Phi^{-1}_{x_{0}}$ as
    \begin{equation}\label{bigfi}
U\!\times\!(0,t_{0})\ni (r,s,t)\mapsto x:=\Phi^{-1}_{x_0}(r,s,t)=\phi_{x_{0}}^{-1}(r,s)+t\nu,
    \end{equation}
 where $\nu$ is the interior normal unit vector at the point $\phi_{x_{0}}^{-1}(r,s)\in\partial\Omega$.
This defines a diffeomorphism of  $U\!\times\!(0,t_{0})$ onto $\mathcal{V}_{x_{0}}$ and its inverse $\Phi_{x_0}$ defines local coordinates on $\mathcal{V}_{x_{0}}$, $\mathcal{V}_{x_{0}}\ni x\mapsto \Phi_{x_0}(x)= (r(x),s(x),t(x))$ such that
 \[
t(x)={\rm dist}\,(x,\partial\Omega).
 \]
It is easily to be seen that the $\Phi_{x_0}$ is constructed so that
\begin{equation}\label{asp-Phix0}
\Phi_{x_0}(x_0)=0,\qquad D\Phi_{x_0}(x_0)={\rm Id}_3,
\end{equation}
where ${\rm Id}_3$ denotes the $3\times3$ identity matrix.
For convenience, we shall henceforth write $(y_{1},y_{2},y_{3})$ instead of $(r,s,t)$. Let us consider the matrix
 \begin{equation}
   \label{gij} g:=g_{x_{0}}= (g_{pq})_{p,q=1}^{3},
 \end{equation}
 with
    \begin{equation}
   g_{pq} =\left \langle \frac{\partial x}{\partial y_{p}}, \frac{\partial x}{\partial y_{q}}\right\rangle, \qquad  \langle X,Y  \rangle=\Sum{1\leq p,q\leq 3}{}g_{pq}\widetilde X_{p}\widetilde Y_{q},
    \end{equation}
     where $X=\Sum{p}{}\widetilde X_{p}\frac{\partial}{\partial y_{p}}$ and $Y=\Sum{q}{}\widetilde Y_{q}\frac{\partial}{\partial y_{q}}$.\\
The Euclidean metric \eqref{stan-metr} transforms to
\begin{align*}
g_{0}&=\Sum{1\leq p,q\leq 3}{}g_{pq}dy_{p}\otimes dy_{q}
\\
&=dy_{3}\otimes dy_{3}+\Sum{1\leq p,q\leq 2}{}\big[ G_{pq}(y_{1},y_{2})-2y_{3}K_{pq}(y_{1},y_{2})+y_{3}^{2}L_{pq}(y_{1},y_{2})\big]dy_{p}\otimes dy_{q},
\end{align*}
where
    \begin{eqnarray*}
    G&=\Sum{1\leq p,q\leq 2}{}G_{pq}dy_{p}\otimes dy_{q} &=\Sum{1\leq p,q\leq 2}{}\left\langle\frac{\partial x}{\partial y_{p}},\frac{\partial
    x}{\partial y_{q}}\right\rangle dy_{p}\otimes dy_{q},\\
    K&= \Sum{1\leq p,q\leq 2}{}K_{pq}dy_{p}\otimes dy_{q} &=\Sum{1\leq p,q\leq 2}{}\left\langle\frac{\partial \nu}{\partial y_{p}},\frac{\partial x}{\partial
    y_{q}}\right\rangle dy_{p}\otimes dy_{q},\\
    L&= \Sum{1\leq p,q\leq 2}{}L_{pq}dy_{p}\otimes dy_{q} &=\Sum{1\leq p,q\leq 2}{}\left\langle\frac{\partial \nu}{\partial y_{p}},\frac{\partial \nu}{\partial
    y_{q}}\right\rangle dy_{p}\otimes dy_{q}
    \end{eqnarray*}
are the first, second and third fundamental forms on $\partial\Omega$.

Note that if $x\in\mathcal{V}_{x_{0}}\cap \partial\Omega$, i.e., $t(x)=0$, $g_{0}$ reduces to
\begin{equation}\label{metric-on-boundary}
g_{0}= dy_{3}\otimes dy_{3}+ G.
\end{equation}
We denote by $g^{-1}:=(g^{pq})_{p,q=1}^{3}$ the matrix inverse of $g$. Let $y_{0}$ be such that $\Phi^{-1}_{x_{0}}(y_{0})\in\mathcal{V}_{x_{0}}\cap \partial \Omega$. By virtue of \eqref{asp-Phix0}, we may assume, by taking $\mathcal{V}_{x_{0}}$ small enough, that
\begin{equation}\label{Ass-m-j}
\frac{1}{2}\,{\rm Id}_3\leq g^{-1}(y_{0})\leq 2\,{\rm Id}_3.
\end{equation}
Using this and the Taylor expansion of the matrix $g^{-1}=(g^{pq})_{p=1}^{3}$ about $y_{0}$, we find
\begin{align}
\label{app-jac}
(1-2c|y-y_0|)g^{-1}(y_0)\leq g^{-1}(y)\leq (1+2 c|y-y_{0}| )g^{-1}(y_{0}).
\end{align}
Let $|g|={\rm det}(g)$. The Lebesgue measure transforms into $dx=|g|^{1/2}dy$. The Taylor expansion of $|g|^{1/2}$ in $\mathcal{V}_{x_{0}}$, together with \eqref{Ass-m-j}, gives us~:
\begin{equation}\label{app-mf}
 (1-2 c|y-y_{0}|) |g|^{1/2}(y_{0})\leq |g|^{1/2}(y) \leq (1+2 c|y-y_{0}|) |g|^{1/2}(y_{0}).
\end{equation}
Here the constant $c$ appearing in \eqref{app-jac} and \eqref{app-mf} can be chosen uniformly (i.e., independently of $x_{0}$) by compactness and regularity of $\partial\Omega$.
The magnetic potential ${\bf A}=({\bf A}_{1},{\bf A}_{2},{\bf A}_{3})$ is transformed to a magnetic potential in the new coordinates $\widetilde{\bf A}=(\widetilde{\bf A}_{1},\widetilde{\bf A}_{2},\widetilde{\bf A}_{3})$ given by
\begin{equation}{\label{tildeA}}
\widetilde{\bf A}_{p}(y)= \sum_{k=1}^{3}{\bf A}_{k}(\Phi_{x_{0}}^{-1}(y))\dfrac{\partial x_{k}}{\partial y_{p}},\quad p=1,2,3.
\end{equation}
The approximation of the magnetic potential in the new coordinates is done by replacing $\widetilde{\bf A}$ by its linear part at $y_{0}$, which we denote $\widetilde {\bf A}^{\rm lin}=(\widetilde{\bf A}^{\rm lin}_{1},\widetilde {\bf A}^{\rm lin}_{2},\widetilde{\bf A}^{\rm lin}_{3})$, so that
\begin{equation}\label{app-MF}
|\widetilde{\bf A}_{p}(y)- \widetilde{\bf A}^{\rm lin}_{p}(y)|\leq C|y-y_{0}| ^{2},
\end{equation}
for all $p=1,2,3$, where
\begin{equation}\label{Alin}
{\bf\widetilde A}^{\rm lin}(y)= {\bf\widetilde A}(y_{0})+\sum_{p=1}^{3}(y_{p}-y_{0p})\dfrac{\partial{\bf \widetilde A}}{\partial y_{p}}(y_{0}).
\end{equation}
The following identity (cf.~\cite[formula (7.23)]{HeMo3}) gives the strength of the magnetic field in terms of the new coordinates,
\begin{equation}\label{relation-nc}
\big|{\bf B}(\Phi^{-1}_{x_{0}}(y_{0}))\big|^{2}=|\widetilde {\bf B}(y_{0})|^{2}= |g(y_{0})|^{-1}\bigg[ \sum_{p,q=1}^{3}g_{pq}(y_{0})\alpha_{p}\alpha_{q}\bigg],
\end{equation}
where $\alpha={\rm curl}({\bf\widetilde A}^{\rm lin})=(\alpha_{1},\alpha_{2},\alpha_{3})$ is given by
\begin{equation}\label{def-alphas}
\alpha_{1}=\dfrac{\partial \widetilde{\bf A}_{3}}{\partial y_{2}}(y_{0})- \dfrac{\partial \widetilde{\bf A}_{2}}{\partial y_{3}}(y_{0}),\quad
\alpha_{2}=\dfrac{\partial \widetilde{\bf A}_{1}}{\partial y_{3}}(y_{0})- \dfrac{\partial \widetilde{\bf A}_{3}}{\partial y_{1}}(y_{0}),\quad
\alpha_{3}= \dfrac{\partial \widetilde{\bf A}_{2}}{\partial y_{1}}(y_{0})- \dfrac{\partial \widetilde{\bf A}_{1}}{\partial y_{2}}(y_{0}).
\end{equation}
Let $u\in L^{2}(\mathcal{V}_{x_{0}})$, we define the map~:
 \begin{equation}\label{app-tild}
y\mapsto \widetilde u(y): =u(\Phi^{-1}_{x_{0}}(y))\,.
   \end{equation}
The next Lemma expresses, in terms of the new coordinates, the
quadratic form and the $L^2$-norm of a function $u$ supported in a
neighborhood of $x_0$.
\begin{lem}\label{lem:qf'}
Let $u\in\mathcal{D}(\mathcal{Q}_{h})$ such that ${\rm supp}\,u\subset \mathcal{V}_{x_{0}}$. We have
    \begin{equation}\label{new-QF}
    \begin{aligned}
    \mathcal{Q}_{h}(u)&=\Int{\mathcal{V}_{x_{0}}}{}|(-ih\nb+{\bf{A}})u|^{2}dx=\Int{\R^{3}_{+}}{}
     \Sum{p,q=1}{3}g^{pq}\big[(-ih\nb_{y_{p}}+\widetilde{\bf A}_{p})\,\widetilde{u}\big]\,\big[\overline{(-ih\nb_{y_{q}}+\widetilde{\bf A}_{q})\,\widetilde{u}}\big]\,|g|^{1/2}dy,
    \end{aligned}
    \end{equation}
    and
    \begin{align}\label{new-norm}
    \norm{u}_{L^{2}(\Omega)}^{2}&=\int_{\Omega}|u(x)|^{2}dx=\int_{\R^{3}_{+}}|g|^{1/2}|\widetilde u(y)|^{2}dy.
    \end{align}
\end{lem}
 \subsection{Approximation of the quadratic form}
The starting point is to simplify the expression of the quadratic form given in \eqref{new-QF} in terms of the new coordinates. To achieve this, we proceed as follows. Let $\ell,T>0$ ($T$ and $\ell$ depend on $h$ and tend to $0$ as $h\rightarrow 0$). Consider the sets
\begin{equation}
Q_{0,\ell,T}=(-\ell/2,\ell/2)^{2}\times(0,T),\qquad  Q_{0,\ell}=(-\ell/2,\ell/2)^{2}\times\{0\},
\end{equation}
such that $\Phi_{x_{0}}^{-1}\big(Q_{0,\ell,T}\big)\subset \mathcal{V}_{x_{0}}$.

Next, consider an arbitrary point $y_{0}\in Q_{0,\ell}$. Then $\Phi^{-1}_{x_{0}}(y_{0})$ lies on the boundary and the metric $g_{0}$ has the form \eqref{metric-on-boundary}. Consequently, the matrix $g(y_{0})$ can be orthogonally diagonalized (being symmetric in this case), and such a diagonalization amounts to a rotation of the coordinate system. After performing such a diagonalization, we may assume that the matrix $g(y_{0})$ is a diagonal matrix given as follows
\begin{equation}{\label{sm}}
 g(y_{0})=
\begin{pmatrix}
\lambda_{1}&0&0\\
0&\lambda_{2}&0\\
0&0& 1
\end{pmatrix}.
\end{equation}
By virtue of \eqref{Ass-m-j}, it is easy to see that $\lambda_{1},\lambda_{2}>0$. Moreover, we have $|g(y_{0})|^{1/2}=\sqrt{\lambda_{1}\lambda_{2}}$.
Denote
\[
\widetilde{Q}_{0,\ell,T}:= \Big(-\frac{\lambda_{1}^{1/2}\ell}{2}, \frac{\lambda_{1}^{1/2}\ell}{2}\Big)\times \Big(-\frac{\lambda_{2}^{1/2}\ell}{2}, \frac{\lambda_{2}^{1/2}\ell}{2}\Big)\times(0,T).
\]
Let $z=(z_{1},z_{2},z_{3})= (\lambda_{1}^{1/2}y_{1},\lambda_{2}^{1/2}y_{2},y_{3})$ and consider a function $u\in L^{2}(\mathcal{V}_{x_{0}})$ such that $\widetilde u$, defined in \eqref{app-tild}, satisfies
 \begin{equation}\label{Asp-sup-u}
 {\rm supp}\,\widetilde u\in {Q}_{0,\ell,T}.
 \end{equation}
 We then define
\begin{equation}\label{not-dot}
\breve{u}(z):=\widetilde u(\lambda_{1}^{-1/2}z_{1},\lambda_{2}^{-1/2}z_{2},z_{3})
\end{equation}
We will approximate $Q_{h}(u)$ via the quadratic form in the half-space corresponding to a constant magnetic field.
\begin{lem}\label{Lem-App}
Let ${\bf F}_{\theta}$ be the magnetic potential given in \eqref{Ftheta} and let $y_{0}\in Q_{0,\ell}$. There exists a constant $C>0$ (independent of $y_{0}$) and a function $\phi_{0}:=\phi_{y_{0}}\in C^{\infty}(\widetilde{Q}_{0,\ell,T})$ such that, for all $\varepsilon\in(0,1]$ satisfying $\varepsilon\geq (\ell+T)$ and for all $u$ satisfying \eqref{Asp-sup-u} one has
\begin{multline}\label{Ap-QF}
\Big| \mathcal{Q}_{h}(u)- \Int{\widetilde{Q}_{0,\ell,T}}{}|(-ih\nb_{z}+b_{0}{\bf F}_{\theta_{0}})e^{i\phi_{0}/h}\,\breve{{u}}|^{2}\,dz     \Big|\\
\leq C\varepsilon\Int{\widetilde{Q}_{0,\ell,T}}{}|(-ih\nb_{z}+b_{0}{\bf F}_{\theta_{0}})e^{i\phi_{0}/h}\,\breve{{u}}|^{2}\,dz
     + C(\ell^{2}+T^{2})^{2} \varepsilon^{-1}\Int{\widetilde{Q}_{0,\ell,T}}{}
|\breve{{u}}|^{2}\,dz,
\end{multline}
and,
\begin{equation}\label{Ap-norm}
  (1-C(\ell+T)) \Int{\widetilde{Q}_{0,\ell,T}}{}
|\breve{{u}}|^{2}\,dz
\leq \norm{u}_{L^{2}(\mathcal{V}_{x_{0}})}^{2}
\leq (1+C(\ell+T))   \Int{\widetilde{Q}_{0,\ell,T}}{}
|\breve{{u}}|^{2}\,dz.
\end{equation}
Here $b_{0}=|\widetilde{B}(y_{0})|$, $\theta_{0}=\widetilde{\theta}(y_{0})$, and to a function $v(x)$ we associate the functions $\widetilde{v}(y)$ and $\breve{v}(z)$ by means of \eqref{app-tild} and \eqref{not-dot} respectively.
\end{lem}
The proof of Lemma~\ref{Ap-norm} is given in the appendix.

\section{Proof of Theorem~\ref{main-thm}}
\subsection{Lower bound}
In this section, we shall prove the lower bound in Theorem~\ref{main-thm}.
\subsubsection{Splitting into bulk and surface terms}
Let
\begin{equation}\label{varsigma}
h^{1/2}\ll\varsigma\ll1
\end{equation}
be a positive number to be chosen later (see \eqref{choice} below) as a positive power of $h$. We consider smooth real-valued functions $\psi_{1}$ and $\psi_{2}$ satisfying
\begin{equation}\label{part-bl-bd-1}
\psi_{1}^{2}(x)+\psi_{2}^{2}(x)=1\qquad {\rm in} \qquad\Omega,
\end{equation}
where
\begin{equation}\label{Psi-1}
\psi_{1}(x):=\left\{\begin{array}{lcl}
1  &{\rm if}&{\rm dist}(x,\partial\Omega)<\varsigma/2\\
0 &{\rm if}&{\rm dist}(x,\partial\Omega)>\varsigma,
\end{array}\right.
\end{equation}
and such that there exists a constant $C_{1}>0$ so that
 \begin{equation}\label{bnd-gr-psi}
\sum_{k=1}^{2}| \nb \psi_{k}|^{2}\leq C_{1}\varsigma^{-2}.
 \end{equation}
Let $\{f_{j}\}_{j=1}^{N}$ be any $L^{2}$ orthonormal set in $\mathcal{D}(\mathcal{P}_{h})$ and $\mathcal{Q}_{h}$ be the quadratic form introduced in \eqref{main-qf}. To prove a lower bound for $\sum_{j}(e_{j}(h)-\Lambda h)_{-}$, we use the variational principle in Lemma~\ref{lem-VP-2}. Namely, we seek a uniform lower bound of
\[
\Sum{j=1}{N}\big(\mathcal{Q}_{h}(f_{j})-\Lambda h\big).
\]
The following Lemma shows that the bulk contribution is negligible compared to the expected leading order term.
\begin{lem}{\label{neg-bulk}}Let $\Lambda\in[0,b)$ with $b$ from \eqref{intensity+min}. The following lower bound holds true
\begin{equation}\label{Res-Bnd}
\Sum{j=1}{N}\big(\mathcal{Q}_{h}(f_{j})-\Lambda h\big)\geq \Sum{j=1}{N}\left(\mathcal{Q}_{h}(\psi_{1}f_{j}) -(\Lambda h+C_{1} h^{2}\varsigma^{-2})
\norm {\psi_{1}f_{j}}^{2}\right),
\end{equation}
where $\{f_{j}\}_{j=1}^{N}$ is an $L^{2}$ orthonormal set in $\mathcal{D}(\mathcal{P}_{h})$ and $\psi_{1}$ is the function from \eqref{Psi-1}.
\end{lem}
\begin{proof}
 By the IMS formula, we find
\begin{align}
\nonumber\mathcal{Q}_{h}(f_{j})&=
\Sum{k=1}{2}\left(\mathcal{Q}_{h}(\psi_{k}f_{j})  -h^{2}{}\norm{|\nabla \psi_{k}|{f_{j}}}^{2}\right).
\end{align}
Using the fact that $\psi_{1}^{2}+\psi_{2}^{2}=1$ and the bound on $|\nb \psi_{k}|$ in \eqref{bnd-gr-psi}, it follows that
\begin{align}\label{boundary-bulk}
\Sum{j=1}{N}\big(\mathcal{Q}_{h}(f_{j})-\Lambda h\big)&\geq\Sum{k=1}{2}\Sum{j=1}{N}\left(\mathcal{Q}_{h}(\psi_{k}f_{j}) -(\Lambda h+C_{1} h^{2}\varsigma^{-2})
\norm {\psi_{k}f_{j}}^{2}\right).
\end{align}
Let us now examine the term corresponding to $k=2$ in the right hand side of \eqref{boundary-bulk}. Using the inequality \eqref{est-HM} for $u:=\psi_{2}f_{j} $, we see that
 \[
 \int_{\Omega}|(-ih\nb +{\bf A})\psi_{2}f_{j}|^{2}dx\geq h(b-Ch^{1/4}) \int_{\Omega}|\psi_{2}f_{j}|^{2}dx. 
\]
We write
\[
 h(b-Ch^{1/4})\Int{\Omega}{}|\psi_{2}f_{j}|^{2}dx= h\Lambda\Int{\Omega}{}|\psi_{2}f_{j}|^{2}dx+h (b-\Lambda-Ch^{1/4})\Int{\Omega}{}|\psi_{2}f_{j}|^{2}dx.
\]
This yields, in view of \eqref{varsigma},
\begin{equation}\label{int-term}
\mathcal{Q}_{h}(\psi_{2}f_{j})\geq (\Lambda h +C_{1} h^{2}\varsigma^{-2})\Int{\Omega}{}|\psi_{2}f_{j}|^{2}dx.
\end{equation}
This gives that the bulk term in \eqref{boundary-bulk} is positive, and the lemma follows.
\end{proof}
\subsubsection{Partition of unity of the boundary}
Recall the cut-off function $\psi_{1}$ from \eqref{Psi-1}, which is supported near a neighborhood of the boundary $\partial\Omega$. Let
\begin{equation}\label{supp-psi}
{\mathcal{O}}_{1}:= {\rm supp}~\psi_{1}= \{x\in\Omega~:~ {\rm dist}(x,\partial\Omega)\leq \varsigma\},
\end{equation}
where $\varsigma$ is, as introduced in \eqref{varsigma}.

Given a point $x$ of the boundary, we let $\Phi_{x}^{-1}$ be the coordinate transformation valid near a small neighbourhood of $x$ (these coordinates are introduced in Section~\ref{bndcoord}). Since the boundary is smooth, there exists $\delta_{x}>0$ such that
\[
\Phi_{x}^{-1}: \widetilde\Omega_{\delta_{x}}\to \mathcal{O}_{x},
\]
where,
$$
\widetilde\Omega_{\delta_{x}}:=\left(-{\delta}_{x},{\delta}_{x}\right)^{2}\times(0,\delta_{x}),\quad\mathcal{O}_{x}=\Phi_{x}^{-1}(\widetilde{\Omega}_{\delta_{x}}).$$
Next, we consider the subset $\Omega_{\delta_{x}}$ of $\widetilde\Omega_{\delta_{x}}$ to be
\[
\Omega_{\delta_{x}}:=\left(-\frac{\delta_{x}}{2},\frac{\delta_{x}}{2}\right)^{2}\times(0,\delta_{x}),
\]
and a covering of $\mathcal{O}_{1}$ by open sets $\{\mathcal{O}_{x}\}_{x\in \partial\Omega}$. Using the compactness of the boundary, it follows that there exists an integer $K$ and an index set $J=\{1,\cdots ,K\}$, such that the sets $\{\mathcal{O}_{x_{l}}\}_{l\in J}$ form a finite covering of $\mathcal{O}_{1}$. For ease of notation, we write $\delta_{l}$ (respectively $\mathcal{O}_{{l}},\Phi_{l}$) instead of $\delta_{x_{l}}$ (respectively $\mathcal{O}_{x_l},\Phi_{x_{l}}$). We emphasize here that the $\delta_{l}$'s are fixed and independent of $h$. Thus, by choosing $\varsigma=\varsigma(h)$ sufficiently small (see \eqref{choice} below), we may assume that
\begin{equation}\label{sec-asp}
 \varsigma\ll \delta_{0}:=\min_{l\in J}\delta_{l}.
\end{equation}
Next, we choose $\{\chi_{l}\}_{l\in J}$ to be non-negative, smooth, compactly supported functions such that
 \begin{equation}\label{def-chi-l}
 \sum_{l\in J}{} \chi_{l}^{2}(x)\equiv1\quad{\rm in}\quad\mathcal{O}_{1}, \qquad {\rm supp}~\chi_{l}\subset {\mathcal{O}}_{l},
 \end{equation}
and such that there exists a constant $C_{2}>0$ (independent of $h$) so that
 \begin{equation}\label{bn-gr}
 \Sum{l\in J}{}|\nb \chi_{l}(x)|^{2}\leq C_{2},
 \end{equation}
 for all $x\in \Omega$.
Next, consider the lattice $\{F_{\varsigma}^{m}\}_{m\in \mathbb{Z}^{2}}$ of $\R^{2}$ generated by the square:
\[
F_{\varsigma}=\Big(-\frac{\varsigma}{2},\frac{\varsigma}{2}\Big)^{2}.
\]
If $m\in\mathbb{Z}^{2}$, denote by $(r_{m},s_{m})=m\varsigma\in\R^{2}$ the center of the square  $F_{\varsigma}^{m}$ so that we can write
\[
F_{\varsigma}^{m}=\left(-\dfrac{\varsigma}{2}+r_{m},\dfrac{\varsigma}{2}+r_{m}\right)\times \left(-\dfrac{\varsigma}{2}+s_{m},\dfrac{\varsigma}{2}+s_{m}\right).
\]
We let $\mathcal{I}_{l}=\{m\in \mathbb{Z}^{2}~:~F_{\varsigma}^{m}\cap (-\frac{\delta_{l}}{2},\frac{\delta_{l}}{2})^{2}\neq \emptyset\}$. If $m\in \mathcal{I}_{l}$ and $\eta>0$, we will write
\begin{equation}\label{setsFml}
F^{m,l}_{\eta}=\left(-\dfrac{\eta}{2}+r_{m},\dfrac{\eta}{2}+r_{m}\right)\times \left(-\dfrac{\eta}{2}+s_{m},\dfrac{\eta}{2}+s_{m}\right),\qquad Q^{m,l}_{\eta}:=F^{m,l}_{\eta}\times(0,\varsigma).
\end{equation}
Let $a\ll 1$ to be chosen later as a positive power of $h$ (see \eqref{choice} below). We introduce a new partition of unity of the square $ (-\frac{\delta_{l}}{2},\frac{\delta_{l}}{2})^{2}$ by smooth functions $\{\widetilde\varphi_{m,l}\}_{m\in\mathcal{I}_{l}}$ with the following properties
\begin{equation}\label{Overlap-part}
\sum_{m\in \mathcal{I}_{l}} \widetilde\varphi_{m,l}^{2}\equiv1\qquad{\rm in }\quad \Big(-\frac{\delta_{l}}{2},\frac{\delta_{l}}{2}\Big)^{2},\quad
{\rm supp}~\widetilde\varphi_{m,l}\subset F^{m,l}_{(1+a)\varsigma},\quad
\widetilde\varphi_{m,l}= 1\quad{\rm in}\quad F^{m,l}_{(1-a)\varsigma},
\end{equation}
and such that there exists a constant $C_{3}>0$ so that
\begin{equation}\label{bnd-grad-phi}
\sum_{m\in \mathcal{I}_{l}}|\nb\widetilde\varphi_{m,l}|^{2}\leq C_{3} (a\varsigma)^{-2}.
\end{equation}
We set
\[
\varphi_{m,l}(x)=\widetilde\varphi_{m,l}(\Phi_{l}(x)).
\]
Let $y_{m,l}$ be an arbitrary point of $Q_{(1+a)\varsigma}^{m,l}$.
As we did in Section~\ref{bndcoord}, we may assume, after performing
a diagonalization, that $g_{l}(y_{m,l})$ ($g_{l}$ is the short
notation of $g_{x_{l}}$) is a diagonal matrix given by
\begin{equation}{\label{sml}}
 g_{l}(y_{m,l})=
\begin{pmatrix}
\lambda_{m,l,1}&0&\\
0&\lambda_{m,l,2}&0\\
0&0& 1
\end{pmatrix}.
\end{equation}
For $y=(y_{1},y_{2},y_{3})\in \R^{3}_{+}$, we denote $y^{\perp}=(y_{1},y_{2})\in \R^{2}$. Applying \eqref{app-mf} with $y_{0}:=y_{m,l}=(y_{m,l}^{\perp},0)\in F_{(1+a)\varsigma}^{m,l}\times\{0\}$, we immediately see that
\begin{equation}\label{metr-approximated}
\big||g_{l}|^{1/2}(y)- \lambda_{m,l,1}^{1/2}\lambda_{m,l,2}^{1/2}\big|\leq c\varsigma \lambda_{m,l,1}^{1/2}\lambda_{m,l,2}^{1/2}.
\end{equation}
We also note that we can approximate the function $\widetilde\chi_{l}^{2}$ within the domain $Q_{(1+a)\varsigma}^{m,l}$ by $\widetilde\chi_{l}^{2}(y_{m,l})$. Indeed, by Taylor expansion, we obtain that for some positive constant $c_{5}>0$
\begin{equation}\label{chi-l-app}
\big|\widetilde\chi_{l}^{2}(y)-\widetilde\chi_{l}^{2}(y_{m,l}) \big|\leq c_{5}\varsigma.
\end{equation}
Put $z=(z_{1},z_{2},z_{3})= (\lambda_{m,l,1}^{1/2}y_{1},\lambda_{m,l,2}^{1/2}y_{2},y_{3})$ and denote by
\[
\widetilde{Q}^{m,l}_{(1+a)\varsigma}:= \Big(-\frac{\varsigma_{m,l,1}}{2}, \frac{\varsigma_{m,l,1}}{2}\Big)\times  \Big(-\frac{\varsigma_{m,l,2}}{2}, \frac{\varsigma_{m,l,2}}{2}\Big)\times(0,\varsigma), \quad \varsigma_{m,l,k}=\frac{{\lambda^{1/2}_{m,l,k}}(1+a)\varsigma}{2},\quad k=1,2.
\]
In the following lemma, we apply localization formulas to restrict the analysis into small boxes where we can approximate the quadratic form using Lemma \ref{Lem-App}.
\begin{lem}\label{lem:app'}
Let $\Lambda\in[0,b)$,  $b$ the constant in \eqref{intensity+min},
${\bf F}_{\theta}$ the magnetic potential given in \eqref{Ftheta}
and $y_{m,l}\in F_{(1+a)\varsigma}^{m,l}\times\{0\}$. There exists a
function $\phi_{m,l}:=\phi_{y_{m,l}}\in
C^{\infty}(\widetilde{Q}^{m,l}_{(1+a)\varsigma})$ and a constant
$\widetilde{C}>0$ such that for all $\varepsilon\in(0,1]$ satisfying
$\varepsilon\gg\varsigma$ one has
\begin{multline}\label{eq-int-M}
\sum_{j=1}^{N}\big(\mathcal{Q}_{h}(\psi_{1}f_{j})-\Lambda h\big)\\
\geq (1-\widetilde{C}\varepsilon)\sum_{j=1}^{N} \sum_{l\in J}\sum_{m\in \mathcal{I}_{l}}
  \bigg\{ \Int{\widetilde{Q}^{m,l}_{(1+a)\varsigma}}{}|(-ih\nb_{z}+b_{m,l}{\bf F}_{\theta_{m,l}})\,e^{i\phi_{m,l}/h}\breve{\varphi }_{m,l}\breve{\psi}_{1}\breve{\chi_{l}}\breve{f}_{j}|^{2}                 dz
   \\ -\Lambda_{1}(h,a,\varsigma,\varepsilon)
     \Int{{\widetilde{Q}^{m,l}_{(1+a)\varsigma}}}{}
|\breve{\varphi }_{m,l}\breve{\psi}_{1}\breve{\chi_{l}}\breve{f}_{j}|^{2}\,dz \bigg\},
\end{multline}
where
\begin{equation}\label{mod-Lambda}
\Lambda_{1}(h,a,\varsigma,\varepsilon)= \dfrac{(\Lambda h+ \widetilde{C} h^{2}(a\varsigma)^{-2})(1+\widetilde{C}\varsigma)+ \widetilde{C}\varsigma^{4}\varepsilon^{-1}}{1-\widetilde{C}\varepsilon},
\end{equation}
$b_{m,l}=|\widetilde{B}(y_{m,l})|$, $\theta_{m,l}=\widetilde{\theta}(y_{m,l})$ and to a function $v(x)$, we associate the function $\breve{v}(z)$ by means of \eqref{not-dot}.
\end{lem}

\subsubsection{The leading order term}
 For $h,\mathfrak{b}>0$
and $\theta\in[0,\pi/2]$, we recall the operator
$\mathcal{P}_{\theta,h,\mathfrak{b}}^{N}$ from \eqref{Phb-th}. Let
us rewrite \eqref{eq-int-M} as
\begin{equation}\label{Eq-Split}
  \sum_{j=1}^{N}(\mathcal{Q}_{h}(f_{j})-\Lambda h)\geq I_{1}+I_{2},
\end{equation}
where
\begin{multline}
I_{1}=(1-\widetilde{C}\varepsilon)\times\\
\sum_{j=1}^{N}\sum_{l\in J}\!\!\!\!\sum_{\substack{m\in \mathcal{I}_{l}\\
\theta_{m,l}\in(0,\pi/2]}}\!\!\!\!\!\!\!\!\!
 \big\langle e^{i\phi_{m,l}/h}\breve{\varphi}_{m,l}\breve{{\chi}_{l}}\breve{{\psi}_{1}}
\breve{ f_{j}}     , \big(\mathscr{P}_{\theta_{m,l}, h,b_{m,l}}^{N}-   \Lambda_{1}(h,a,\varsigma,\varepsilon)\big)  e^{i\phi_{m,l}/h}\breve{\varphi}_{m,l}\breve{{\chi}_{l}}\breve{{\psi}_{1}}
\breve{ f_{j}}  \big\rangle,
\end{multline}
and
\begin{multline}
I_{2}=(1-\widetilde{C}\varepsilon)\times\\
\sum_{j=1}^{N}\sum_{l\in J}\sum_{\substack{m\in \mathcal{I}_{l}\\
\theta_{m,l}=0}}
\big\langle e^{i\phi_{m,l}/h}\breve{\varphi}_{m,l}\breve{{\chi}_{l}}\breve{{\psi}_{1}}
\breve{ f_{j}}     , \big(\mathscr{P}_{\theta_{m,l}, h,b_{m,l}}^{N}-   \Lambda_{1}(h,a,\varsigma,\varepsilon)\big)  e^{i\phi_{m,l}/h}\breve{\varphi}_{m,l}\breve{{\chi}_{l}}\breve{{\psi}_{1}}
\breve{ f_{j}}  \big\rangle.
\end{multline}
Below in \eqref{choice}, the parameters $a,\varsigma$ and $\varepsilon$ are chosen so that, when $h$ is sufficiently small, one has
\begin{equation}\label{ass-lbd}
h^{-1}\Lambda_{1}(h,a,\varsigma,\varepsilon)< b,
\end{equation}
where $b$ is defined in \eqref{intensity+min}. We first start by estimating $I_{1}$. Using Lemma~\ref{computations-cor}, we see that
\begin{multline}\label{eq-int-M-1}
 I_{1}
  \geq - h ( 1-\widetilde{C}\varepsilon) \sum_{l\in J}\sum_{\substack{m\in \mathcal{I}_{l}\\ \theta_{m,l}\in[0,\pi/2)}}  b_{m,l}\sum_{k}\big(\zeta_{k}(\theta_{m,l})-h^{-1}b_{m,l}^{-1}\Lambda_{1}(h,a,\varsigma,\varepsilon)\big)_{-}\times\\
 \sum_{j=1}^{N}\Big\langle e^{i\phi_{m,l}/h}\breve{{\psi}}_{1}\breve{{\chi}}_{l}\breve{{ \varphi}}_{m,l}\breve{ f}_{j}  , \Pi_{\theta_{m,l},k} (h,b_{m,l})e^{i\phi_{m,l}/h}\breve{\psi_{1}}\breve{ \chi_{l}}\breve{\varphi}_{m,l}\breve{ f_{j}}\Big\rangle_{L^{2}(\widetilde{Q}^{m,l}_{(1+a)\varsigma})}.
\end{multline}
Here, for $\theta\in(0,\pi/2]$ and $\mathfrak{b}>0$, $\{\zeta_{k}(\theta)\}_{k}$ are the eigenvalues from \eqref{Egv-Lth} and $\Pi_{\theta,k}(h,\mathfrak{b})$ is the projector defined in \eqref{Pr-theta-j}.
Using \eqref{metr-approximated} and that $dz=\lambda_{m,l,1}^{1/2}\lambda_{m,l,2}^{1/2}dy$, we obtain that for some constant $C_{4}>0$
\begin{multline}\label{mid-eq}
 \sum_{j=1}^{N}\Big\langle e^{i\phi_{m,l}/h}\breve{{\psi}}_{1}\breve{{\chi}}_{l}\breve{{ \varphi}}_{m,l}\breve{ f}_{j}  , \Pi_{\theta_{m,l},k} (h,b_{m,l})e^{i\phi_{m,l}/h}\breve{\psi_{1}}\breve{ \chi_{l}}\breve{\varphi}_{m,l}\breve{ f_{j}}\Big\rangle_{L^{2}(\widetilde{Q}^{m,l}_{(1+a)\varsigma})}\\
 \leq (1+ C_{4}\varsigma)\sum_{j=1}^{N}\langle f_{j}, H(m,l,k,h,b_{m,l},\theta_{m,l}) f_{j}
 \rangle_{L^{2}(\Omega)}\,.
\end{multline}
Here $H(m,l,k,h,b_{m,l},\theta_{m,l})$ is a positive operator, which is given by,
\begin{equation*}
H(m,l,k,h,b_{m,l},\theta_{m,l}):= \psi_{1}\chi_{l}\varphi_{m,l}U_{\Phi_{l}}
V_{z\to y}e^{i\phi_{m,l}/h}\Pi_{\theta_{m,l},k}(h,b_{m,l})e^{i\phi_{m,l}/h}V^{-1}_{z\rightarrow y}U^{-1}_{\Phi_{l}}\psi_{1}\chi_{l}\varphi_{m,l}\,,
\end{equation*}
where, for a function $v$, $V_{z\to y}$ is defined by
\begin{equation}\label{def-Vzy}
(V_{z\to y}{v})(y)=v(\lambda_{m,l,1}^{1/2}y_{1},\lambda_{m,l,2}^{1/2}y_{2},y_{3}),
\end{equation}
and, for a function $u$, the transformation $U_{\Phi_{l}}$ is given by
\begin{equation}\label{def-Uphi}
(U_{\Phi_{l}}u)(x)=u(\Phi_{l}(x))\,.
\end{equation}
Since $\{f_{j}\}_{j=1}^{N}$ is an orthonormal family in $L^{2}(\Omega)$, we deduce that
\begin{equation}\label{endeq}
\sum_{j=1}^{N}\langle f_{j}, H(m,l,k,h,b_{m,l},\theta_{m,l}) f_{j} \rangle_{L^{2}(\Omega)}
\leq {\rm Tr}(H(m,l,k,h, b_{m,l},\theta_{m,l})).
\end{equation}
Combining \eqref{eq-int-M-1}, \eqref{mid-eq} and \eqref{endeq}, and using that $\varepsilon
\gg \varsigma$ (see \eqref{choice} below), we obtain that for some constant $C_{5}>0$
\begin{multline}\label{eq-int-M-2}
 I_{1}
  \geq -( 1-{C}_{5}\varepsilon) h     \sum_{l\in J}\sum_{\substack{m\in \mathcal{I}_{l}\\ \theta_{m,l}\in(0,\pi/2]}}b_{m,l}\sum_{k}\big(\zeta_{k}(\theta_{m,l})-h^{-1}b_{m,l}^{-1}\Lambda_{1}(h,a,\varsigma,\varepsilon)\big)_{-}\times\\
{\rm Tr}(H(m,l,k,h, b_{m,l},\theta_{m,l})).
\end{multline}
It is straightforward to show that
\begin{multline}
{\rm Tr}(H(m,l,k,h, b_{m,l},\theta_{m,l}))\\
=b_{m,l}^{3/2}h^{-3/2}\int_{\R}\int_{\R^{3}_{+}}|g_{l}(y)|^{1/2}\widetilde{\psi}_{1}^{2}(y)\widetilde{\chi}_{l}^{2}(y){\widetilde\varphi}_{m,l}^{2}(y_{1},y_{2},0)|V_{z\to y}\big(v_{\theta_{m,l},k}(h^{-1/2}b_{m,l}^{1/2}z;\xi)\big)|^{2}dyd\xi,
\end{multline}
where, for $\theta\in[0,\pi/2]$, $v_{\theta,k}(\cdot,\xi)$ is the function defined in \eqref{vjth}.
Using \eqref{metr-approximated} and \eqref{chi-l-app}, and that $\psi_{1}(x)\leq 1$ for all $x\in\Omega$, it follows that
\begin{multline}
{\rm Tr}(H(m,l,k,h, b_{m,l},\theta_{m,l}))
\leq  (2\pi)^{-1}(1+c\varsigma)(\widetilde{\chi}^{2}(y_{m,l})+c_{5}\varsigma)\lambda_{m,l,1}^{1/2}\lambda_{m,l,2}^{1/2} h^{-3/2}b_{m,l}^{3/2}\times\\
\int_{\R}\int_{\R^{3}_{+}}\widetilde\varphi_{m,l}^{2}(y_{1},y_{2},0)|u_{\theta_{m,l},k}(h^{-1/2}b_{m,l}^{1/2}\lambda_{m,l,2}^{1/2}y_{2}-\frac{\xi}{\sin(\theta_{m,l})})|^{2}dyd\xi,
\end{multline}
where, for $\theta\in[0,\pi/2]$, the functions $\{u_{\theta,k}\}_{k}$ are introduced in \eqref{Egv-Lth}. Performing similar calculations to that in \eqref{calcul-kern}, we deduce that
\begin{multline}\label{cal-Tr-th-n0}
{\rm Tr}(H(m,l,k,h, b_{m,l},\theta_{m,l}))\\ \leq (2\pi)^{-1} h^{-1}b_{m,l} \lambda_{m,l,1}^{1/2}\lambda_{m,l,2}^{1/2} (\chi_{l}^{2}(y_{m,l})+c_{5}\varsigma)(1+c\varsigma) \sin(\theta_{m,l}) \int_{\R^{2}} |\widetilde\varphi_{m,l}|^{2}dy_{1}dy_{2}\\
\leq (2\pi)^{-1} h^{-1}b_{m,l} \lambda_{m,l,1}^{1/2}\lambda_{m,l,2}^{1/2} (\chi_{l}^{2}(y_{m,l})+c_{5}\varsigma)(1+c\varsigma) \sin(\theta_{m,l})(1+a)^{2}\varsigma^{2},
\end{multline}
where in the last step, we used that the function $\widetilde\varphi_{m,l}$ is less than one and supported in the square $F^{m,l}_{(1+a)\varsigma}$.
Recalling \eqref{E-th-ld-th-n-z} and substituting \eqref{cal-Tr-th-n0} into \eqref{eq-int-M-2}, we obtain that for some positive constant $C_{6}>0$
\begin{multline}\label{eq-int-M-5}
 I_{1}
  \geq
-(1-C_{6}\varepsilon) (1+a)^{2}
\sum_{l\in J}\!\!\!\!\sum_{\substack{m\in \mathcal{I}_{l}\\ \theta_{m,l}\in(0,\pi/2]}}\!\!\!\!\!\!E\big(\theta_{m,l}, h^{-1}b_{m,l}^{-1}\Lambda_{1}(h,a,\varsigma,\varepsilon)\big)
b_{m,l}^{2}\lambda_{m,l,1}^{1/2}\lambda_{m,l,2}^{1/2}(\chi_{l}^{2}(y_{m,l})+c_{5}\varsigma) \varsigma^{2}.
\end{multline}
We now proceed in a similar manner to get a lower bound on $I_{2}$.
By virtue of Lemma~\ref{eqs-PR-cor}, it follows that
\begin{multline}\label{eq-int-M-1-1}
 I_{2}
  \geq - h (2\pi)^{-1}(1-\widetilde{C}\varepsilon) \sum_{l\in J}\sum_{\substack{m\in \mathcal{I}_{l}\\ \theta_{m,l}=0}}  b_{m,l}\int_{\R^{2}}\sum_{k}\big(\mu_{k}(\xi)  +\tau^{2}-h^{-1}b_{m,l}^{-1}\Lambda_{1}(h,a,\varsigma,\varepsilon)   \big)_{-}\times\\
\sum_{j=1}^{N}\Big\langle e^{i\phi_{m,l}/h}\breve{{\psi}}_{1}\breve{ \chi_{l}}\breve{\varphi}_{m,l}\breve{ f}_{j},\Pi_{k} (h,b_{m,l};\xi,\tau)e^{i\phi_{m,l}/h}\breve{\psi_{1}}\breve{{\chi}}_{l} \breve{{ \varphi}}_{m,l}\breve{ f_{j}}\Big\rangle_{L^{2}(\widetilde{Q}^{m,l}_{(1+a)\varsigma})}d\xi d\tau,
\end{multline}
where, for $(\xi,\tau)\in\R^{2}$ and $\mathfrak{b}>0$, $\{\mu_{k}(\xi)\}_{k}$ are the eigenvalues from \eqref{muj} and $\Pi_{1}(\xi,\tau;h,\mathfrak{b})$ is the projector defined in \eqref{Pr-resc-th0}. Using Lemma~\ref{mu2}, it follows that
 \begin{multline}\label{eq-int-M-1-1-1}
 I_{2}
  \geq - h (2\pi)^{-1}(1-\widetilde{C}\varepsilon) \sum_{l\in J}\sum_{\substack{m\in \mathcal{I}_{l}\\ \theta_{m,l}=0}}  b_{m,l}\int_{\R^{2}}\big(\mu_{1}(\xi)  +\tau^{2}-h^{-1}b_{m,l}^{-1}\Lambda_{1}(h,a,\varsigma,\varepsilon)   \big)_{-}\times\\
\sum_{j=1}^{N}\Big\langle e^{i\phi_{m,l}/h}\breve{{\psi}}_{1}\breve{ \chi_{l}}\breve{\varphi}_{m,l}\breve{ f}_{j},\Pi_{k} (h,b_{m,l};\xi,\tau)e^{i\phi_{m,l}/h}\breve{\psi_{1}}\breve{{\chi}}_{l} \breve{{ \varphi}}_{m,l}\breve{ f_{j}}\Big\rangle_{L^{2}(\widetilde{Q}^{m,l}_{(1+a)\varsigma})}d\xi d\tau,
\end{multline}
By an equality similar to that
in \eqref{eq-int-M-2}, we find 
\begin{multline}\label{eq-int-M-1-2}
 I_{2}
  \geq - h(1-C_{5}\varepsilon) (2\pi)^{-1}\sum_{l\in J}\sum_{\substack{m\in \mathcal{I}_{l}\\ \theta_{m,l}=0}}  b_{m,l}\times\\
  \int_{\R^{2}}\sum_{k}\big(\mu_{k}(\xi)  +\tau^{2}-h^{-1}b_{m,l}^{-1}\Lambda_{1}(h,a,\varsigma,\varepsilon)   \big)_{-}
 {\rm Tr}(H^{\prime}(m,l,h,b_{m,l};\xi,\tau))d\xi d\tau.
\end{multline}
Here $H^{\prime}(m,l,h,b_{m,l};\xi,\tau)$ is a positive operator, which is given by,
\begin{equation*}
H^{\prime}(m,l,h,b_{m,l};\xi,\tau):= \psi_{1}\chi_{l}\varphi_{m,l}U_{\Phi_{l}}V_{z\to y}e^{i\phi_{m,l}/h}\Pi_{1}(\xi,\tau;h,b_{m,l})e^{i\phi_{m,l}/h}V^{-1}_{z\rightarrow y}U^{-1}_{\Phi_{l}}\psi_{1}\chi_{l}\varphi_{m,l}
\,.
\end{equation*}
where $U_{\Phi_{l}}$ and $V_{z\to y}$ are the same as defined in \eqref{def-Uphi} and \eqref{def-Vzy} respectively.
It is easy to see that
\begin{multline}\label{eq-int-M-1-3}
{\rm Tr}(H^{\prime}(m,l,h,b_{m,l};\xi,\tau))\\
=b_{m,l}^{3/2}h^{-3/2}\int_{\R^{3}_{+}}|g_{l}(y)|^{1/2}\widetilde{\psi}_{1}^{2}(y)\widetilde{\chi}_{l}^{2}(y){\widetilde\varphi}_{m,l}^{2}(y_{1},y_{2},0)|V_{z\to y}\big(v_{1}(h^{-1/2}b_{m,l}^{1/2}z;\xi,\tau)\big)|^{2}dy,
\end{multline}
where the function $v_{1}$ is defined in \eqref{def-vj}.
Using \eqref{metr-approximated} and \eqref{chi-l-app}, and that $\psi_{1}(x)\leq 1$ for all $x\in\Omega$, it follows that
\begin{multline}
{\rm Tr}(H^{\prime}(m,l,h,b_{m,l};\xi,\tau))
\leq  (2\pi)^{-1}(1+c\varsigma)(\widetilde{\chi}^{2}(y_{m,l})+c_{5}\varsigma)\lambda_{m,l,1}^{1/2}\lambda_{m,l,2}^{1/2} h^{-3/2}b_{m,l}^{3/2}\times\\
\int_{\R^{3}_{+}}\widetilde\varphi_{m,l}^{2}(y_{1},y_{2},0)|u_{1}(h^{-1/2}b_{m,l}^{1/2}y_{3},\xi)|^{2}dy,
\end{multline}
Using that the function $u_{1}(\cdot;\xi)$ (from \eqref{muj}) is normalized in $L^{2}(\R_{+})$, we get
\begin{multline}\label{cal-Tr-th-n0-2}
{\rm Tr}(H^{\prime}(m,l,h,b_{m,l};\xi,\tau))
\leq (2\pi)^{-1} h^{-1}b_{m,l} \lambda_{m,l,1}^{1/2}\lambda_{m,l,2}^{1/2} (\chi_{l}^{2}(y_{m,l})+c_{5}\varsigma)(1+c\varsigma) (1+a)^{2}\varsigma^{2}.
\end{multline}
Inserting \eqref{cal-Tr-th-n0-2} in \eqref{eq-int-M-1-2} yields
\begin{multline}\label{eq-int-M-1-4}
 I_{2}
  \geq - h (1-C_{6}\varepsilon)(4\pi^{2})^{-1} \sum_{l\in J}\sum_{\substack{m\in \mathcal{I}_{l}\\ \theta_{m,l}=0}}  b_{m,l}^{2}
\lambda_{m,l,1}^{1/2}\lambda_{m,l,2}^{1/2} (\widetilde\chi_{l}^{2}(y_{m,l})+C\varsigma) (1+a)^{2}\varsigma^{2}\times\\
 \int_{\R^{2}}\Big(\mu_{1}(\xi)  +\tau^{2}-h^{-1}b_{m,l}^{-1}\Lambda_{1}(h,a,\varsigma,\varepsilon)   \Big)_{-}  d\xi d\tau.
\end{multline}
Using \eqref{asym-the=0}, it follows that
\begin{multline}\label{eq-int-M-1-5}
 I_{2}
  \geq -(1-C_{6}\varepsilon) (1+a)^{2}
    \sum_{l\in J}\sum_{\substack{m\in \mathcal{I}_{l}\\ \theta_{m,l}=0}}E(0,\Lambda_{1}(h,a,\varsigma,\varepsilon))
b_{m,l}^{2}\lambda_{m,l,1}^{1/2}\lambda_{m,l,2}^{1/2}(\chi_{l}^{2}(y_{m,l})+c_{5}\varsigma)\varsigma^{2}.
\end{multline}
Therefore, combining \eqref{eq-int-M-5} and \eqref{eq-int-M-1-5}, and using \eqref{Eq-Split}, we obtain
\begin{multline}\label{eq-int-M-3}
  \Sum{j=1}{N}\big( \mathcal{Q}_{h}(f_{j})-\Lambda h\big)
  \geq -(1-C_{6}\varepsilon)  (1+a)^{2}  \times\\
  \sum_{l\in J}\sum_{m\in \mathcal{I}_{l}}b_{m,l}^{2} \lambda_{m,l,1}^{1/2}\lambda_{m,l,2}^{1/2}(\chi_{l}^{2}(y_{m,l})+c_{5}\varsigma) E\big(\theta_{m,l}, h^{-1}b_{m,l}^{-1} \Lambda_{1}(h,a,\varsigma,\varepsilon)\big)\varsigma^{2}.
\end{multline}
Using the fact that for all $\lambda_{0}\in(0,1)$, the function $(0,\lambda_0]\times[0,\pi/2]\mapsto E(\theta,\lambda)$ is bounded by Lemma~\ref{continuity+}, we see that the term
\[
  C\varsigma(1+a)^{2}\sum_{l\in J}\sum_{m\in \mathcal{I}_{l}}  b_{m,l}^{2}E(\theta_{m,l}, h^{-1}b_{m,l}^{-1} \Lambda_{1}(h,a,\varsigma,\varepsilon))\varsigma^{2}
\]
is bounded by $C\varsigma \sum_{l\in J}\sum_{m\in \mathcal{I}_{l}}\lambda_{m,l,1}^{1/2}\lambda_{m,l,2}^{1/2}\varsigma^{2}\sim  C\varsigma |\partial\Omega|$. This leads to
\begin{multline}\label{eq-int-M-4}
  \Sum{j=1}{N}\big( \mathcal{Q}_{h}(f_{j})-\Lambda h\big)
  \geq -( 1-C_{6}\varepsilon)  (1+a)^{2} \times\\ \sum_{l\in J}\sum_{m\in \mathcal{I}_{l}}\lambda_{m,l,1}^{1/2}\lambda_{m,l,2}^{1/2}\widetilde\chi_{l}^{2}(y_{m,l}) b_{m,l}^{2}E(\theta_{m,l}, h^{-1}b_{m,l}^{-1} \Lambda_{1}(h,a,\varsigma,\varepsilon))\varsigma^{2}+O(\varsigma).
\end{multline}
By \eqref{sm}, we have $\lambda_{m,l,1}^{1/2}\lambda_{m,l,2}^{1/2}=|g_{l}(y_{m,l})|^{1/2}$. For $y=(y^{\perp},0)\in F_{\varsigma}^{m,l}\times\{0\}$, we define the function
\begin{equation}
G(y):=|g_{l}(y)|^{1/2}\widetilde\chi_{l}^{2}(y) |\widetilde{\bf B}(y)|^{2}
E\big(\widetilde \theta(y), h^{-1}| \widetilde{\bf B}(y)|^{-1} \Lambda_{1}(h,a,\varsigma,\varepsilon)\big).
\end{equation}
We pick $y_{m,l}\in F_{\varsigma}^{m,l}\times\{0\}$ so that
\[
\min_{y\in F^{m,l}_{\varsigma}\times\{0\}} G(y)= G(y_{m,l}).
\]
Then the right-hand side of \eqref{eq-int-M-4} is a lower Riemann sum. Hence, we find
\begin{multline}\label{RS}\sum_{m\in \mathcal{I}_{l}}|g_{l}(y_{m,l})|^{1/2} \widetilde\chi_{l}^{2}(y_{m,l}) b_{m,l}^{2}E\big(\theta_{m,l}, h^{-1}b_{m,l} \Lambda_{1}(h,a,\varsigma,\varepsilon,\eta)\big)\varsigma^{2} = \sum_{m\in \mathcal{I}_{l}} G(y_{m,l})\varsigma^{2}           \leq \\
\int_{(-\delta_{l},\delta_{l})^{2}}G(y_{1},y_{2},0)dy_{1}dy_{2}=\int_{ \partial\Omega} \chi_{l}^{2}(x) |{\bf B}(x)|^{2}E\big(\theta(x), h^{-1}|{\bf B}(x)|^{-1} \Lambda_{1}(h,a,\varsigma,\varepsilon)\big)d\sigma(x).
\end{multline}
Plugging this into \eqref{eq-int-M-4}, and using that $\sum_{l\in J}\chi_{l}^{2}(x)=1$, we obtain
\begin{multline}\label{eq-int-M-3}  \Sum{j=1}{N}\big( \mathcal{Q}_{h}(f_{j})-\Lambda h\big) \\
  \geq -( 1-C_{6}\varepsilon)  (1+a)^{2}  \int_{\partial\Omega}  |{\bf B}(x)|^{2}E\big(\theta(x), h^{-1}|{\bf B}(x)|^{-1} \Lambda_{1}(h,a,\varsigma,\varepsilon)\big)d\sigma(x) +{O}(\varsigma).
\end{multline}
We make the following choice of $\varepsilon$, $a$ and $\varsigma$,
\begin{equation}\label{choice}
\varepsilon= h^{1/4},\qquad  a=h^{1/16} \qquad \varsigma=h^{3/8}.
\end{equation}
This choice yields that for some constant $C_{7}>0$, one has
\[
 h^{-1}\Lambda_{1}(h,a,\varsigma,\varepsilon) \sim \dfrac{\Lambda + C_{7}h^{1/8}}{1-\widetilde{C}h^{1/4}} \qquad {\rm as}\quad h\rightarrow 0.
\]
The function $[0,1)\times [0,\pi/2]\mapsto E(\theta,\lambda)$ is
locally Lipschitz according to Lemma~\ref{fct-Lip}. This gives
\begin{equation}
\Big|E\big(\theta(x),h^{-1}\Lambda_{1}(h,a,\varsigma,\varepsilon)|{\bf B}(x)|^{-1})-E(\theta(x),\Lambda |{\bf B}(x)|^{-1}\big)\Big|\leq C_{8}h^{1/8}b^{-1},
\end{equation}
for some constant $C_{8}>0$. The constant  $b$ is introduced in
\eqref{intensity+min}.
It follows that for some constant $C_{9}>0$, we have
  \begin{equation}
   \Sum{j=1}{N}\big( \mathcal{Q}_{h}(f_{j})-\Lambda h\big)\geq -(1+C_{9}h^{1/8}) \int_{\partial\Omega}  |{\bf B}(x)|^{2}E\big(\theta(x), \Lambda |{\bf B}(x)|^{-1} )\big)dx +O(h^{1/8}),
\end{equation}
uniformly with respect to $N$ and the orthonormal family
$\{f_{j}\}$. As a consequence, we may use Lemma~\ref{lem-VP-2} to
obtain
 the desired lower bound.
\subsection{Upper bound}
Let $\varsigma>0$ be as in \eqref{varsigma}  and $F_{\varsigma}^{m,l}$ be the set defined in \eqref{setsFml} with $l\in J$ and $m\in \mathcal{I}_{l}$ being the indices corresponding to the partitions $\{\chi_{l}\}_{l\in J}$ and $\{\widetilde \varphi_{m,l}\}_{m\in \mathcal I_{l}}$ introduced in \eqref{def-chi-l} and \eqref{Overlap-part} respectively.  Let $\{y_{m,l}\}$ be a finite family of points in $ F_{\varsigma}^{m,l}\times\{0\}$ to be specified later at the end of this section. To each point $y_{m,l}$ we associate $b_{m,l}=| \widetilde{\bf B}(y_{m,l})|$ and $\theta_{m,l}=\widetilde\theta(y_{m,l})$ defined in \eqref{mf-nc} and \eqref{theta-ang-0} respectively (with $y_{0}$ replaced by $y_{m,l}$). Let $y\in  Q_{(1+a)\varsigma}^{m,l}$ (see the definition of the set in \eqref{setsFml}) and $\lambda_{m,l,1},\lambda_{m,l,2}$ be the diagonal components of the matrix $g_{l}(y_{m,l})$ from \eqref{sml}.
We put $z=(\lambda_{m,l,1}^{1/2}y_{1},\lambda_{m,l,2}^{1/2}y_{2},y_{3})$. Let $(\xi,\tau)\in\R^{2}$. Recall the tilde-notation from \eqref{app-tild} and define the functions
\begin{align*}
\widetilde f_{j,l,m}(y,\xi;h)&:=h^{-3/4}b_{m,l}^{3/4}v_{j,\theta_{m,l}}(h^{-1/2}b_{m,l}^{1/2}z;\xi)\big({\widetilde\varphi}_{m,l}{\widetilde\chi_{l}}{
\widetilde\psi}_{1}\big)(y) & {\rm if}&\quad\theta_{m,l}\in(0,\pi/2]\\
\widetilde g_{l,m}(y;\xi,\tau,h)&:= (2\pi)^{-1/2}h^{-3/4}b_{m,l}^{3/4}v_{1}(h^{-1/2}b_{m,l}^{1/2}z;\xi,\tau)\big({\widetilde\varphi}_{m,l}{\widetilde\chi_{l}}{
\widetilde\psi}_{1}\big)(y)&{\rm if}&\quad\theta_{m,l}=0.
\end{align*}
where $v_{j,\theta}(\cdot;\xi)$, $v_{1}(\cdot;\xi,\tau)$ and $\widetilde\psi_{1}$ are the functions found in \eqref{vjth}, \eqref{def-vj} and \eqref{Psi-1} respectively.

Recall the coordinate transformation $\Phi_{l}$ valid near a neighborhood of the point $x_{l}$ (see Subsection~\ref{bndcoord}), and let $x=\Phi_{l}^{-1}(y)$. We define $f_{j,l,m}(x,\xi;h):={ \widetilde f_{j,l,m}}(y,\xi;h)$ and $g_{l,m}(x,\xi,\tau;h):= \widetilde g_{l,m}(y,\xi,\tau;h)$.
Let $\Lambda\in[0,b)$ and $L>0$ be a sufficiently large number to be selected later so that $\varsigma h^{-1/2}\gg L$. We put
\[
M_{j,m,l,\xi}=
{\bf 1}_{\big\{(j,m,l,\xi)\in\mathbb{N}\times I_{l}\times J\times\R \,:\, \zeta_{j}(\theta_{m,l})- b_{m,l}^{-1}\Lambda\leq 0, \quad |\frac{\xi}{\sin(\theta_{m,l})}|\leq L  \big\}}
\,,\]
and
\[
M^{\prime}_{m,l,\xi,\tau}= {\bf 1}_{\big\{(j,m,l,\xi,\tau)\in\mathbb{N}\times I_{l}\times J\times\R^{2}~:~\quad\mu_{1}(\xi)+\tau^{2}- b_{m,l}^{-1}\Lambda\leq 0\big\}}.
\]
Here we note that the condition $\mu_{1}(\xi)+\tau^{2}- b_{m,l}^{-1}\Lambda
\leq  0$ implies, in view of Proposition~\ref{properties-mu1} and the fact that $b_{m,l}^{-1}\Lambda\leq b^{-1}\Lambda<1$, that there exists a constant $K>0$ (independent of $m,l$) such that
\begin{equation}\label{int-xi-tau}
(\xi,\tau)\in I_{\xi,\tau}:=(0,K)\times(-1,1).
\end{equation}
Define, for $f\in L^{2}(\Omega)$,
\begin{align*}
(\gamma_{1} f)(x)&=\sum_{l\in J}\sum_{\substack{m\in\mathcal{I}_{l}\\ \theta_{m,l}\in(0,\pi/2]}}\sum_{j}\int_{\R}  M_{j,m,l,\xi}\big\langle f, {f_{j,l,m}}(\cdot,\xi;h)\big\rangle {f_{j,l,m}}(x,\xi;h) d\xi,
\end{align*}
and
\begin{align*}
(\gamma_{2} f)(x)&=\sum_{l\in J}\sum_{\substack{m\in\mathcal{I}_{l}\\
\theta_{m,l}=0}}\int_{\R^{2}}M^{\prime}_{m,l,\xi,\tau}\big\langle f, g_{l,m}(\cdot,\xi,\tau;h)\big\rangle g_{l,m}(x,\xi,\tau;h)d\xi d\tau.
\end{align*}
We have the following lemma.
\begin{lem}\label{Pro-dm}
Let $f\in L^{2}(\Omega)$ and define the operator $\gamma$ by
\[
\gamma f=\gamma_{1}f+\gamma_{2}f.
\]
There exists a constant $C_{10}>0$ such that the quadratic form associated to $\gamma$ satisfies
\begin{equation}\label{gam-dm}
0\leq \langle \gamma f, f\rangle_{L^{2}(\Omega)} \leq (1+C_{10}\varsigma)\norm{f}^{2}_{L^{2}(\Omega)}.
\end{equation}
\end{lem}
\begin{proof}
Consider $f\in L^{2}(\Omega)$. It is easy to see that $\langle \gamma f,f\rangle \geq 0$. Next, using that $M_{j,m,l,\xi}\leq 1$, we see that
\begin{equation}\label{gamma-1-dm}
\big\langle f,\gamma_{1} f\big\rangle_{L^{2}(\Omega)}
\leq \sum_{l\in J}\sum_{\substack{m\in \mathcal{I}_{l}\\ \theta_{m,l}\in (0,\pi/2]}}   \sum_{j}\int_{\R}\Big| \big\langle f, {f_{j,l,m}}(x,\xi;h) \big\rangle_{L^{2}(\Omega)} \Big|^{2} d\xi.
\end{equation}
By \eqref{new-norm} and \eqref{metr-approximated}, it follows that there exists a constant $C_{11}>0$ such that
\begin{equation}\label{Pr-des-mtrx}
\begin{aligned}
\Big| \big\langle f, {f_{j,l,m}}(x,\xi;h) \big\rangle_{L^{2}(\Omega)} \Big|^{2}
& \leq (1+C_{11}\varsigma)\lambda_{m,l,1}\lambda_{m,l,2}\Big| \int_{\R^{3}_{+}}{\widetilde f}\,\,\overline {\widetilde{ \chi}_{l}\widetilde{\psi}_{1}\widetilde \varphi_{m,l} \big(U_{h,b_{m,l}} v_{j,\theta_{m,l}}(z;\xi)\big) }  dy                                                                                 \Big|^{2}\\
&= (1+C_{11}\varsigma)\Big| \int_{\R^{3}_{+}}{ \breve{ f}}~\overline{\breve{{ \chi}_{l}}\breve{{\psi}_{1}}\breve{ \varphi}_{m,l} (U_{h,b_{m,l}}v_{j,\theta_{m,l}}(z;\xi)\big)  } dz \Big|^{2}.
 \end{aligned}
 \end{equation}
Here, the transformation $U_{h,b_{m,l}}$ is defined in
\eqref{unit-tran} and for a function $u$, $\breve u$ is associated
to $u$ using \eqref{not-dot} and \eqref{app-tild}.

Substituting \eqref{Pr-des-mtrx} into \eqref{gamma-1-dm}, we find
\begin{multline}\label{gamma-1-dm-2}
\big\langle f,\gamma_{1} f\big\rangle_{L^{2}(\Omega)}
\leq  (1+C_{11}\varsigma)\sum_{l\in J}\sum_{\substack{m\in\mathcal{I}_{l}\\ \theta_{m,l}\in(0,\pi/2]}}\sum_{j} \Big| \int_{\R^{3}_{+}} \breve{ f}~\overline{\breve{{ \chi}_{l}}\breve{{\psi}_{1}}\breve{\varphi}_{m,l}U_{h,b_{m,l}} v_{j,\theta_{m,l}}(z;\xi)  } dz \Big|^{2}d\xi.
\end{multline}
In a similar fashion, one can show that
 \begin{multline*}
\big\langle f,\gamma_{2} f\big\rangle_{L^{2}(\Omega)}\\
\leq (1+C_{11}\varsigma)\sum_{l\in J}\sum_{\substack{m\in\mathcal{I}_{l}\\ \theta_{m,l}=0}}\sum_{j}\int_{\R^2}\Big| \int_{\R^{3}_{+}}{ \breve{ f}}~\overline{\breve{{ \chi}_{l}}\breve{{\psi}_{1}}\breve{\varphi}_{m,l} \big(U_{h,b_{m,l}}(2\pi)^{-1/2}\big(v_{j}(h^{-1/2}b_{m,l}^{1/2}z;\xi,\tau)\big)\big) }  dz \Big|^{2}d\xi d\tau.
 \end{multline*}
Next, we recall the definition of $v_{j,\theta}(y;\xi)$ from \eqref{vjth} (resp. $v_{j}$ from \eqref{def-vj}) and use the fact that $\{v_{j,\theta_{m,l}}\}_{j}$ (resp. $u_{j}(\cdot,\xi)$ for all $\xi\in\R$) is an orthonormal set of eigenfunctions,
we thus find
 \begin{equation}\label{gam-dm-c2}
\big\langle f,\gamma f\big\rangle_{L^{2}(\Omega)}
\leq (1+C_{11}\varsigma)\sum_{l\in J}\sum_{m\in\mathcal{I}_{l}}\Int{\R^{3}_{+}}{}\big| \breve{{ \chi}}_{l}\breve{{\psi}}_{1}\breve{ \varphi}_{m,l}\breve{ f}(z)\big|^{2}dz\,.
\end{equation}
Similar reasoning to that in \eqref{Pr-des-mtrx} yields that for some constant $C_{12}>0$
 \begin{align*}
 \Int{\R^{3}_{+}}{}\big| \breve{{ \chi}}_{l}\breve{{\psi}}_{1}\breve{ \varphi}_{m,l}\breve{ f}(z)\big|^{2}dz
&= \lambda_{m,l,1}^{1/2}\lambda_{m,l,2}^{1/2}\Int{\R^{3}_{+}}{}|g_{l}|^{-1/2}(y)|g_{l}|^{1/2}(y)\big| {\widetilde{ \chi}}_{l}{\widetilde{\psi}}_{1}{\widetilde \varphi}_{m,l}\widetilde f(y)\big|^{2}dy\\
&\leq  (1+C_{12}\varsigma)\Int{\R^{3}_{+}}{}|g_{l}|^{1/2}\big| {\widetilde{ \chi}}_{l}{\widetilde{\psi}}_{1}{\widetilde \varphi}_{m,l} \widetilde f(y)\big|^{2}dy\\
&=(1+C_{12}\varsigma)\Int{\Omega}{}\big| \chi_{l}\varphi_{m,l}\psi _{1}f(x)\big|^{2}dx\,.
\end{align*}
Implementing this into \eqref{gam-dm-c2}, and using \eqref{def-chi-l} and \eqref{Overlap-part}, yields the claim of the lemma.
\end{proof}
By the variational principle in Lemma~\ref{lem-VP-3}, an upper bound of the sum of eigenvalues of $\mathcal{P}_{h}$ below $\Lambda h$ follows if we can prove an upper bound on
$$ (1+C_{12}\varsigma)^{-1}{\rm Tr}\big[(\mathcal{P}_{h}-\Lambda h){\gamma}\big]= (1+C_{12}\varsigma)^{-1}\Big({\rm Tr}\big[(\mathcal{P}_{h}-\Lambda h){\gamma_{1}}\big]+ {\rm Tr}\big[(\mathcal{P}_{h}-\Lambda h){\gamma_{2}}\big]\Big)$$
Recall the quadratic form $\mathcal{Q}_{h}$ defined in \eqref{main-qf}. We start by estimating
\begin{equation}\label{def-tr-gamma-1}
{\rm Tr}\big[(\mathcal{P}_{h}-\Lambda h){\gamma_{1}}\big]\\
 := \sum_{l\in J}\sum_{\substack{m\in \mathcal{I}_{l}\\ \theta_{m,l}\in(0,\pi/2]}} \sum_{j}\int_{\R}M_{j,m,l,\xi}\big(\mathcal{Q}_{h}(f_{j,l,m}(x,\xi;h))-\Lambda h\norm{f_{j,l,m}(x,\xi;h)}^{2}\big)d\xi.
\end{equation}
Recall the transformation $V_{z\to y}$ introduced in \eqref{def-Vzy}. Using \eqref{new-norm}, it follows from \eqref{metr-approximated} and \eqref{chi-l-app} that there exists a constant $C_{13}>0$ such that
\begin{multline}\label{norm-fjl-0}
\int_{\Omega}{}|f_{j,l,m}(x,\xi;h)|^{2}dx
= \Int{\R^{3}_{+}}{}|g_l|^{1/2}(y)|\widetilde f_{j,l,m}(y,\xi;h)|^{2}dy\\
\geq (\widetilde{\chi}_{l}^{2}(y_{m,l})-C_{13}\varsigma)
\lambda_{m,l,1}^{1/2}\lambda_{m,l,2}^{1/2}b_{m,l}^{3/2}h^{-3/2}\Int{\R^{3}_{+}}{}|V_{z\to
y}\big(v_{j,\theta_{m,l}}(h^{-1/2}b_{m,l}^{1/2}z;\xi)\big){\widetilde\varphi}_{m,l}{
\widetilde\psi}_{1} |^{2}  dy\,.
\end{multline}
Let us write the last integral as
\begin{multline}\label{norm-fjl-2}
\Int{Q_{(1+a)\varsigma}^{m,l}}{}\!\!\!\!\!\!|V_{z\to y}\big(v_{j,\theta_{m,l}}(h^{-1/2}b_{m,l}^{1/2}z;\xi)\big){\widetilde\varphi}_{m,l}{
\widetilde\psi}_{1} |^{2}  dy= \Int{Q_{(1+a)\varsigma}^{m,l}}{}|V_{z\to y}\big(v_{j,\theta_{m,l}}(h^{-1/2}b_{m,l}^{1/2}z;\xi)\big) |^{2}  dy \\
+ \Int{Q_{(1+a)\varsigma}^{m,l}}{}\!\!\!\!\!\! \big[-1+{\widetilde\varphi}^{2}_{m,l}{{\widetilde\psi}}_{1}^{2}            \big]|V_{z\to y}
\big(v_{j,\theta_{m,l}}(h^{-1/2}b_{m,l}^{1/2}z;\xi)\big) |^{2} dy.
\end{multline}
As we shall work on the support of $M_{j,m,l}$ in view of \eqref{def-tr-gamma-1}, we may restrict ourselves to the indices $(j,m,l,\xi)$ satisfying $\zeta_{j}(\theta_{m,l})\leq \Lambda b_{m,l}^{-1}<\Lambda b^{-1}$ and $|\xi/\sin(\theta_{m,l})|\leq L$. Recalling that $L$ is chosen so that $L\ll h^{-1/2}\varsigma$ and using Lemma~\ref{exp-dec}, it follows that for all $\alpha \in \sqrt{1-\Lambda b^{-1}}$, there exists a constant $C_{14}>0$ such that
\begin{multline}\label{norm-fjl-2:1}
\Int{Q_{(1+a)\varsigma}^{m,l}}{}\!\!\!\!\!\!|V_{z\to y}\big(v_{j,\theta_{m,l}}(h^{-1/2}b_{m,l}^{1/2}z;\xi)\big)\big({\widetilde\varphi}_{m,l}{
\widetilde\psi}_{1}\big)(y) |^{2}  dy\\
\geq (1-e^{-\alpha (C_{14}\varsigma h^{-1/2}-L)})\Int{Q_{(1+a)\varsigma}^{m,l}}{} |V_{z\to y}\big(v_{j,\theta_{m,l}}(h^{-1/2}b_{m,l}^{1/2}z;\xi) \big)|^{2}  dy,
\end{multline}
where we have used \eqref{supp-psi} and \eqref{Overlap-part}.

Implementing \eqref{norm-fjl-2:1} in \eqref{norm-fjl-0}, we obtain
\begin{multline}\label{norm-fjl-1}
\int_{\Omega}{}|f_{j,l,m}(x,\xi;h)|^{2}dx
 \geq (1-e^{-\alpha(C_{14}\varsigma h^{-1/2}-L)})(\widetilde{\chi}_{l}^{2}(y_{m,l})-C_{13}\varsigma)\lambda_{m,l,1}^{1/2}\lambda_{m,l,2}^{1/2} b_{m,l}^{3/2}h^{-3/2}\times\\
 \Int{Q_{(1+a)\varsigma}^{m,l}}{} |V_{z\to y}\big(v_{j,\theta_{m,l}}(h^{-1/2}b_{m,l}^{1/2}z;\xi) \big)|^{2}
 dy\,.
\end{multline}
Let us estimate $\mathcal{Q}_{h}( f_{j,l,m})$.
Applying Lemma~\ref{Lem-App} with $u= f_{j,l,m}$, we find, for all $\varepsilon\gg \varsigma$,
\begin{multline}\label{ub-Q-bnd}
 \mathcal{Q}_{h}(f_{j,l,m})\\
\leq (1+{C}\varepsilon)
\Int{ \widetilde{Q}_{(1+a)\varsigma}^{m,l}}{}
 \big|(-ih\nb_{z}+b_{m,l}{\bf F}_{\theta_{m,l}})e^{i\phi_{m,l}/h}\,\breve{f}_{j,l,m}\big|^{2}\,dz
     +{C}\varsigma^{4} \varepsilon^{-1}\Int{ \widetilde{Q}_{(1+a)\varsigma}^{m,l}}{}
|\breve{{f}}_{j,l,m}|^{2}\,dz\\
\leq
(1+{C}\varepsilon)h^{-3/2}b_{m,l}^{3/2} \Int{\widetilde{Q}_{(1+a)\varsigma}^{m,l}}{}
   ( \breve{{\varphi}}_{m,l}\breve{\psi}_{1}\breve{\chi}_{l} )^{2} \big|(-ih\nb_{z}+b_{m,l}{\bf F}_{\theta_{m,l}})       v_{j,\theta_{m,l}}(h^{-1/2}b_{m,l}^{1/2}z;\xi)\big|^{2}\,dz\\
 + h^{-3/2}b_{m,l}^{3/2}\int_{\widetilde{Q}_{(1+a)\varsigma}^{m,l}}\!\!\!\!\!\!\!\Big(|\nabla( \breve{\psi}_{1}\breve{\chi_{l}}\breve{\varphi}_{m,l})|^{2}+{C}\varsigma^{4}\varepsilon^{-1}
 \breve{\varphi}_{m,l}^{2}\breve{\psi}_{1}^{2}
 \breve{\chi}_{l}^{2}  \Big)\big|     v_{j,\theta_{m,l}}(h^{-1/2}b_{m,l}^{1/2}z;\xi)\big|^{2}dz,
 \end{multline}
where $C$ is the constant from Lemma \ref{Lem-App}.

By \eqref{bnd-gr-psi}, \eqref{bn-gr}, \eqref{bnd-grad-phi}, and approximating $\widetilde\chi_{l}^{2}$ using \eqref{chi-l-app}, it follows that for some constant $C_{15}>0$,
 \begin{multline}\label{Q-fjml}
 \mathcal{Q}_{h}(f_{j,l,m})\\
\leq h^{-3/2}b_{m,l}^{3/2}({\widetilde\chi}_{l}^{2}(y_{m,l})+C_{15}\varepsilon)\Int{\widetilde{Q}_{(1+a)\varsigma}^{m,l}}{}
  |(-ih\nb_{z}+b_{m,l}{\bf F}_{\theta_{m,l}})\, v_{j,\theta_{m,l}}(h^{-1/2}b_{m,l}^{1/2}z;\xi)|^{2}\,dz
     \\
     +C_{15}h^{-3/2}b_{m,l}^{3/2}(\varsigma^{4} \varepsilon^{-1}+h^{2}(a\varsigma)^{-2})\Int{\widetilde{Q}_{(1+a)\varsigma}^{m,l}}{}
| v_{j,\theta_{m,l}}(h^{-1/2}b_{m,l}^{1/2}z;\xi)|^{2}\,dz.
\\ \leq \lambda_{m,l,1}^{1/2}\lambda_{m,l,2}^{1/2}\big(h^{-1/2}b_{m,l}^{5/2}({\widetilde\chi}_{l}^{2}(y_{m,l})+C_{15}\varepsilon)\zeta_{j}(\theta_{m,l})
     +C_{15}h^{-3/2}b_{m,l}^{3/2}(\varsigma^{4} \varepsilon^{-1}+h^{2}(a\varsigma)^{-2}\big)\times\\ 
     \Int{{Q}_{(1+a)\varsigma}^{m,l}}{}
|V_{z\rightarrow y} v_{j,\theta_{m,l}}(h^{-1/2}b_{m,l}^{1/2}z;\xi)|^{2}\,dy.
\end{multline}
Combining \eqref{norm-fjl-1} and \eqref{ub-Q-bnd}, we obtain in view of \eqref{def-tr-gamma-1} that,
\begin{multline}
{\rm Tr}\big[(\mathcal{P}_{h}-\Lambda h){\gamma_{1}}\big] \leq 
\sum_{l\in J}\sum_{\substack{m\in \mathcal{I}_{l}\\ \theta_{m,l}\in(0,\pi/2]}}\sum_{j}\int_{\R}M_{j,m,l,\xi}\Big( \mathcal{Q}_{h}(f_{j,l,m})- \Lambda h\norm{f_{j,l,m}}^{2}\Big)d\xi \\ \leq 
\sum_{l\in J}\sum_{\substack{m\in \mathcal{I}_{l}\\ \theta_{m,l}\in(0,\pi/2]}}\sum_{j}\int_{\R}M_{j,m,l,\xi}\Big\{h^{-1/2}b_{m,l}^{5/2}\lambda_{m,l,1}^{1/2}\lambda_{m,l,2}^{1/2}\Big[ ({\widetilde\chi}_{l}^{2}(y_{m,l})+C_{15}\varepsilon)\zeta_{j}(\theta_{m,l})\\+ C_{15}(\varsigma^{4} \varepsilon^{-1}+h^{2}(a\varsigma)^{-2})-  \Lambda  b_{m,l}^{-1} (1-e^{-\alpha(C_{14}\varsigma h^{-1/2}-L)})(\widetilde{\chi}_{l}^{2}(y_{m,l})-C_{13}\varsigma) \Big]\times\\
     \int_{\R}\Int{{Q}_{(1+a)\varsigma}^{m,l}}{}
|V_{z\rightarrow y} v_{j,\theta_{m,l}}(h^{-1/2}b_{m,l}^{1/2}z;\xi)|^{2}\,dy \Big\}d\xi.
\end{multline}
Performing the translation $\nu= h^{-1/2}b_{m,l}^{1/2}s- \frac{\xi}{\sin{\theta_{m,l}}}$ and let $\nu_{\pm}=\pm h^{-1/2}b_{m,l}^{1/2}(1+a)\varsigma \pm L$, it follows that 
\begin{multline*}\label{QF-trial-fct-2}
  {\rm Tr}\big[(\mathcal{P}_{h}-\Lambda h){\gamma_{1}}\big]
  \leq -(1+a)^{2} \sum_{l\in J}\sum_{\substack{m\in \mathcal{I}_{l}\\ \theta_{m,l}\in(0,\pi/2]}}{\widetilde\chi}_{l}^{2}(y_{m,l})\lambda_{m,l,1}^{1/2}\lambda_{m,l,2}^{1/2}b_{m,l}^{2}
\varsigma^{2}\times\\
\Big\{\sum_{j}(2\pi)^{-1} \sin(\theta_{m,l})\big(\zeta_{j}(\theta_{m,l}) -\Lambda b_{m,l}^{-1}\big)_{-}\Big\} \int_{\nu_{-}}^{\nu_{+}}|u_{j,\theta_{m,l}}(\nu,t)|d\nu dt
+ I_{\rm rest}^{(1)},
\end{multline*}
where
\begin{multline}
I_{\rm rest}^{(1)}=(1+a)^{2}\sum_{l\in J}\sum_{\substack{m\in \mathcal{I}_{l}\\ \theta_{m,l}\in(0,\pi/2]}} \sum_{j} \varsigma^{2} \lambda_{m,l,1}^{1/2}\lambda_{m,l,2}^{1/2}b_{m,l}^{2}(2\pi)^{-1}\sin(\theta_{m,l})\int_{\nu_{-}}^{\nu_{+}}|u_{j,\theta_{m,l}}(\nu,t)|d\nu dt\times\\
\Big\{ C_{15}\big(\varepsilon\zeta_{j}(\theta_{m,l})
+h^{-1}b_{m,l}^{-1}(\varsigma^{4} \varepsilon^{-1}+h^{2}(a\varsigma)^{-2}\big)                                                                                      +\Lambda  b_{m,l}^{-1}\big( e^{-\alpha(C_{14}\varsigma h^{-1/2}-L)}(\widetilde{\chi}_{l}^{2}(y_{m,l})-C_{13}\varsigma)+ C_{13}\varsigma\big)  \Big\}.
\end{multline}
By virtue of Lemma~\ref{thm-exp-for-Eth}, \eqref{QF-trial-fct-2} reads
\begin{multline}\label{comp-tr-gam}
  {\rm Tr}\big[(\mathcal{P}_{h}-\Lambda h){\gamma_{1}}\big]
 \leq -(1+a)^{2}\sum_{l\in J}\sum_{\substack{m\in \mathcal{I}_{l}\\ \theta_{m,l}\in(0,\pi/2]}}{\widetilde\chi}_{l}^{2}(y_{m,l})\lambda_{m,l,1}^{1/2}\lambda_{m,l,2}^{1/2}b_{m,l}^{2}
\varsigma^{2}\times \\ E(\theta_{m,l},\Lambda b_{m,l}^{-1})\int_{\nu_{-}}^{\nu_{+}}|u_{j,\theta_{m,l}}(\nu,t)|d\nu dt
+I^{(1)}_{\rm rest}.
\end{multline}
It remains to estimate
\begin{multline}\label{def-gamma-tr-2}
{\rm Tr}\big[(\mathcal{P}_{h}-\Lambda h)\small{\gamma_{2}}\big]
:= \sum_{l\in J}\sum_{\substack{m\in\mathcal{I}_{l}\\ \theta_{m,l}=0}} M^{\prime}_{m,l,\xi,\tau}\int_{\R}\big(\mathcal{Q}_{h}(g_{l,m}(x,\xi,\tau;h))-\Lambda h\norm{g_{l,m}(x,\xi,\tau;h)}^{2}\big)d\xi.
\end{multline}
We start by estimating $\norm{g_{l,m}(x,\xi,\tau;h)}^{2}$. It follows from \eqref{new-norm} \eqref{metr-approximated} and \eqref{chi-l-app} that there exists a constant $C_{13}>0$ such that
\begin{multline}\label{norm-fjl-0:01}
\int_{\Omega}{}|g_{l,m}(x,\xi,\tau;h)|^{2}dx\\
\geq (2\pi)^{-1}(\widetilde{\chi}_{l}^{2}(y_{m,l})-C_{13}\varsigma) \lambda_{m,l,1}^{1/2}\lambda_{m,l,2}^{1/2}b_{m,l}^{3/2}h^{-3/2}\Int{\R^{3}_{+}}{}|V_{z\to y}\big(v_{1}(h^{-1/2}b_{m,l}^{1/2}z;\xi)\big){\widetilde\varphi}_{m,l}{
\widetilde\psi}_{1} (y)|^{2}  dy\\= (2\pi)^{-2}(\widetilde{\chi}_{l}^{2}(y_{m,l})-C_{13}\varsigma) \lambda_{m,l,1}^{1/2}\lambda_{m,l,2}^{1/2}b_{m,l}^{3/2}h^{-3/2}\Int{\R^{3}_{+}}{}|u_1(h^{-1/2}b_{m,l}^{1/2}y_{3};\xi){\widetilde\varphi}_{m,l}{
\widetilde\psi}_{1} (y)|^{2}  dy,
\end{multline}
where the function $u_1(\cdot,\xi)$ is introduced in \eqref{muj}.
Let us write the last integral as
\begin{multline}\label{norm-fjl-2:02}
\Int{Q_{(1+a)\varsigma}^{m,l}}{}\!\!\!\!\!\!|u_{1}(h^{-1/2}b_{m,l}^{1/2}y_3;\xi){\widetilde\varphi}_{m,l}{
\widetilde\psi}_{1} |^{2}  dy= \Int{Q_{(1+a)\varsigma}^{m,l}}{}|u_{1}(h^{-1/2}b_{m,l}^{1/2}y_3;\xi) |^{2}  dy \\
+ \Int{Q_{(1+a)\varsigma}^{m,l}}{}\!\!\!\!\!\! \big[-1+{\widetilde\varphi}^{2}_{m,l}{{\widetilde\psi}}_{1}^{2}            \big]|u_{1}(h^{-1/2}b_{m,l}^{1/2}y_3;\xi) |^{2} dy.
\end{multline}

Due to the support of $\widetilde\psi_1$, we note that the integral on the right hand side is restricted to the set where $y_3\geq \varsigma/2$. Recalling \eqref{int-xi-tau} and selecting $\varsigma$ as in \eqref{choice}, one has for $h$ sufficiently small,
\begin{equation}\label{Asp-5}
(b^{1/2}h^{-1/2}\varsigma-\xi)^{2}\geq (b^{1/2}h^{-1/2}\frac{\varsigma}{2}-\xi)^{2}\gg \frac{1}{16} bh^{-1}\varsigma^{2}\gg 1.
\end{equation}
Using this and Lemma~\ref{Agmon-lem}, we obtain for some constant $C_{16}>0$
\begin{multline}\label{norm-fjl-2:1:03}
\Int{Q_{(1+a)\varsigma}^{m,l}}{}\!\!\!\!\!\!|u_{1}(h^{-1/2}b_{m,l}^{1/2}y_3;\xi)\big({\widetilde\varphi}_{m,l}{
\widetilde\psi}_{1}\big)(y) |^{2}  dy
\geq (1-e^{-C_{16}\varsigma^2 h^{-1}})\Int{Q_{(1+a)\varsigma}^{m,l}}{} |u_{1}(h^{-1/2}b_{m,l}^{1/2}y_3;\xi) |^{2}  dy,
\end{multline}
where we have used \eqref{supp-psi} and \eqref{Overlap-part}.

Implementing \eqref{norm-fjl-2:1:03} in \eqref{norm-fjl-0:01}, we obtain
\begin{multline}\label{norm-fjl-1:04}
\int_{\Omega}{}|g_{l,m}(x,\xi,\tau;h)|^{2}dx
 \geq (2\pi)^{-2}(1-e^{-C_{16}\varsigma^{2} h^{-1}})(\widetilde{\chi}_{l}^{2}(y_{m,l})-C_{13}\varsigma)\lambda_{m,l,1}^{1/2}\lambda_{m,l,2}^{1/2} b_{m,l}^{3/2}h^{-3/2}
\times \\  \Int{Q_{(1+a)\varsigma}^{m,l}}{}\big|u_{1}(h^{-1/2}b_{m,l}^{1/2}y_3;\xi)  \big|^{2}  dy.
\end{multline}
Using the same arguments that have led to \eqref{Q-fjml}, one can show that
 \begin{multline}\label{Q-gjml}
 \mathcal{Q}_{h}(g_{l,m})
\leq h^{-3/2}b_{m,l}^{3/2}(2\pi)^{-1}
\big({\widetilde\chi}_{l}^{2}(y_{m,l})+C_{15}\varepsilon\big)\Int{Q_{(1+a)\varsigma}^{m,l}}{}
  |(-ih\nb_{z}+b_{m,l}{\bf F}_{\theta_{m,l}})\, v_{1}(h^{-1/2}b_{m,l}^{1/2}z;\xi)|^{2}\,dz\\
     +C_{15}h^{-3/2}b_{m,l}^{3/2}(2\pi)^{-1}\big(\varsigma^{4} \varepsilon^{-1}+h^{2}(a\varsigma)^{-2}\big)\Int{Q_{(1+a)\varsigma}^{m,l}}{}
| v_{1}(h^{-1/2}b_{m,l}^{1/2}z;\xi)|^{2}\,dz\\
\leq \sum_{l\in J}\sum_{\substack{m\in \mathcal{I}_{l}\\ \theta_{m,l}=0}} \lambda_{m,l,1}^{1/2}\lambda_{m,l,2}^{1/2} b_{m,l}^{2}(1+a)^{2}\varsigma^{2}(2\pi)^{-2}\times\\
\Big\{({\widetilde\chi}_{l}^{2}(y_{m,l})+C_{15}\varepsilon)(\mu_{1}(\xi)+\tau^{2})
     +C_{15}h^{-1}b_{m,l}^{-1}(\varsigma^{4} \varepsilon^{-1}+h^{2}(a\varsigma)^{-2})\Big\}.
\end{multline}
Integrating in $\xi$ and $\tau$ and taking into account \eqref{int-xi-tau}, it follows that (recall \eqref{def-gamma-tr-2})
\begin{multline}\label{QF-trial-fct-I}
  {\rm Tr}\big[(\mathcal{P}_{h}-\Lambda h)\small{\gamma_{2}}\big]
 \leq  \sum_{l\in J}\sum_{\substack{m\in \mathcal{I}_{l}\\ \theta_{m,l}=0}}M^{\prime}_{m,l,\xi,\tau}\lambda_{m,l,1}^{1/2}\lambda_{m,l,2}^{1/2}b_{m,l}^{2}(2\pi)^{-2}\varsigma^{2}(1+a)^{2}
 \times\\
\int_{\R^{2}}\Big(({\widetilde\chi}_{l}^{2}(y_{m,l})+C_{15}\varepsilon)(\mu_{1}(\xi)+\tau^{2})
     +C_{15}h^{-1}b_{m,l}^{-1}(\varsigma^{4} \varepsilon^{-1}+h^{2}a^{-2}\varsigma^{-2})\\
-\Lambda  b_{m,l}^{-1} (1-e^{-C_{16}\varsigma^{2} h^{-1}})(\widetilde{\chi}_{l}^{2}(y_{m,l})-C_{13}\varsigma)\Big)d\xi d\tau\\
=-(1+a)^{2}\sum_{l\in J}\sum_{\substack{m\in \mathcal{I}_{l}\\ \theta_{m,l}=0}}{\widetilde\chi}_{l}^{2}(y_{m,l})\lambda_{m,l,1}^{1/2}\lambda_{m,l,2}^{1/2}b_{m,l}^{2}
\varsigma^{2}(2\pi)^{-2}\int_{\R^{2}}\big(\mu_{1}(\xi) +\tau^{2}-\Lambda b_{m,l}^{-1}\big)_{-}d\xi d\tau\\
+I_{\rm rest}^{(2)},
\end{multline}
where
\begin{multline}
I_{\rm rest}^{(2)}=\sum_{l\in J}\sum_{\substack{m\in \mathcal{I}_{l}\\ \theta_{m,l}=0}}\lambda_{m,l,1}^{1/2}\lambda_{m,l,2}^{1/2}b_{m,l}^{2}(2\pi)^{-2}\varsigma^{2}(1+a)^{2}
\int_{\R^{2}}M^{\prime}_{m,l,\xi,\tau}\Big(C_{15}\varepsilon(\mu_{1}(\xi)+\tau^{2})\\
     +C_{15}h^{-1}b_{m,l}^{-1}(\varsigma^{4} \varepsilon^{-1}+h^{2}a^{-2}\varsigma^{-2})
-\Lambda  b_{m,l}^{-1}\big( e^{-C_{16}\varsigma^{2} h^{-1}}(\widetilde{\chi}_{l}^{2}(y_{m,l})-C_{13}\varsigma) +C_{13}\varsigma \big)\Big)d\xi d\tau.
\end{multline}
Taking into account the support of $M^{\prime}_{m,l,\xi,\tau}$, we deduce the following bound on $|I_{\rm rest}^{(2)}|$,
\begin{multline}
|I_{\rm rest}^{(2)}|\leq \sum_{l\in J}\sum_{\substack{m\in \mathcal{I}_{l}\\ \theta_{m,l}=0}}\lambda_{m,l,1}^{1/2}\lambda_{m,l,2}^{1/2}b_{m,l}^{2}(2\pi)^{-2}\varsigma^{2}(1+a)^{2}
4K\Big(C_{15}\varepsilon\Lambda b_{m,l}^{-1}\\
     +C_{15}h^{-1}b_{m,l}^{-1}(\varsigma^{4} \varepsilon^{-1}+h^{2}a^{-2}\varsigma^{-2})
-\Lambda  b_{m,l}^{-1}\big( e^{-C_{16}\varsigma^{2} h^{-1}}(\widetilde{\chi}_{l}^{2}(y_{m,l})-C_{13}\varsigma) +C_{13}\varsigma \big)\Big).
\end{multline}
In view of Lemma~\ref{value-th=0}, the estimate \eqref{QF-trial-fct-I} reads
\begin{equation}\label{comp-tr-gam-2}
  {\rm Tr}\big[(\mathcal{P}_{h}-\Lambda h)\small{\gamma_{2}}\big]
 \leq  -(1+a)^{2}\sum_{l\in J}\sum_{\substack{m\in \mathcal{I}_{l}\\ \theta_{m,l}=0}}{\widetilde\chi}_{l}^{2}(y_{m,l})\lambda_{m,l,1}^{1/2}\lambda_{m,l,2}^{1/2}b_{m,l}^{2}
\varsigma^{2}(2\pi)^{-2}E(0,\Lambda b_{m,l}^{-1})
+I^{(2)}_{\rm rest}.
\end{equation}
Combining \eqref{comp-tr-gam} and \eqref{comp-tr-gam-2}, and recalling \eqref{E-th-ld-th-n-z}, we obtain
  \begin{multline}\label{comp-tr-gam-1-2}
  {\rm Tr}\bigg[(\mathcal{P}_{h}-\Lambda h)\small\frac{\gamma}{1+C_{10}\varsigma}\bigg]
 \\ \leq  -(1+C_{10}\varsigma)^{-1}
 \Big\{(1+a)^{2} \sum_{l\in J}\sum_{\substack{m\in \mathcal{I}_{l}\\ \theta_{m,l}\in(0,\pi/2]}}{\widetilde\chi}_{l}^{2}(y_{m,l})\lambda_{m,l,1}^{1/2}\lambda_{m,l,2}^{1/2}b_{m,l}^{2}
\varsigma^{2}\times\\
E(\theta_{m,l},\Lambda b_{m,l}^{-1}) \int_{\nu_{-}}^{\nu_{+}}|u_{j,\theta_{m,l}}(\nu,t)|d\nu dt
- I_{\rm rest}^{(1)}\\+ (1+a)^{2}\sum_{l\in J}\sum_{\substack{m\in \mathcal{I}_{l}\\ \theta_{m,l}=0}} {\widetilde\chi}_{l}^{2}(y_{m,l})\lambda_{m,l,1}^{1/2}\lambda_{m,l,2}^{1/2}b_{m,l}^{2}\varsigma^2 E(0,\Lambda b_{m,l}^{-1})- I^{(2)}_{\rm rest}\Big\}.
\end{multline}
By Lemma \ref{lem-VP-3}, it is easy to see that 
  \begin{multline}\label{comp-tr-gam-1-2}
-\sum_{j}(e_{j}(h)-\Lambda h)_{-}
  \leq  -(1+C_{10}\varsigma)^{-1}
 \Big\{(1+a)^{2} \sum_{l\in J}\sum_{\substack{m\in \mathcal{I}_{l}\\ \theta_{m,l}\in(0,\pi/2]}}{\widetilde\chi}_{l}^{2}(y_{m,l})\lambda_{m,l,1}^{1/2}\lambda_{m,l,2}^{1/2}b_{m,l}^{2}
\varsigma^{2}\times\\
E(\theta_{m,l},\Lambda b_{m,l}^{-1}) \int_{\nu_{-}}^{\nu_{+}}|u_{j,\theta_{m,l}}(\nu,t)|d\nu dt
- I_{\rm rest}^{(1)}\\+ (1+a)^{2}\sum_{l\in J}\sum_{\substack{m\in \mathcal{I}_{l}\\ \theta_{m,l}=0}} {\widetilde\chi}_{l}^{2}(y_{m,l})\lambda_{m,l,1}^{1/2}\lambda_{m,l,2}^{1/2}b_{m,l}^{2}\varsigma^2 E(0,\Lambda b_{m,l}^{-1})- I^{(2)}_{\rm rest}\Big\}.
\end{multline}
Using the upper bound estimate in Lemma~\ref{Nb-Ev-Finite}, together with the fact that $|{\bf B}|$ is bounded on $\partial\Omega$ and that $\sum_{l\in J}\sum_{m\in \mathcal{I}_{l}}  \varsigma^{2}\lambda_{m,l,1}^{1/2}\lambda_{m,l,2}^{1/2}\sim|\partial\Omega|$, we find
\begin{equation}
|I_{\rm rest}^{(1)}|+ |I_{\rm rest}^{(2)}|=\mathcal{O}(\varepsilon+h^{-1}\varsigma^{4}\varepsilon^{-1}+h(a\varsigma)^{-2}+ e^{-\alpha (C_{14}\varsigma h^{-1/2}-L)}+e^{-C_{16}\varsigma^{2} h^{-1}} ).
\end{equation}
Recall the choice of $\varepsilon,\varsigma, a$ in \eqref{choice} and choose in addition $L=h^{-1/ 16}$, we thus obtain 
\[
|I_{\rm rest}^{(1)}|+ |I_{\rm rest}^{(2)}|= \mathcal{O}(h^{1/8}),
\]
and 
\[
\lim_{h\rightarrow 0}\int_{\nu_{-}}^{\nu_{+}}|u_{j,\theta_{m,l}}(\nu,t)|d\nu dt=\int_{-\infty}^{\infty}|u_{j,\theta_{m,l}}(\nu,t)|d\nu dt= 1.
\]
We thus get, when taking $\limsup_{h\rightarrow 0}$ on both sides of \eqref{comp-tr-gam-1-2} the following estimate 
  \begin{multline}
\limsup_{h\rightarrow 0}\big\{-\sum_{j}(e_{j}(h)-\Lambda h)_{-} \big\}
  \leq  -
 \sum_{l\in J}\sum_{m\in\mathcal{I}_{l}}{\widetilde\chi}_{l}^{2}(y_{m,l})\lambda_{m,l,1}^{1/2}\lambda_{m,l,2}^{1/2}b_{m,l}^{2}
\varsigma^{2}
E(\theta_{m,l},\Lambda b_{m,l}^{-1}) .
\end{multline}
By \eqref{sml}, we have $\lambda_{m,l,1}^{1/2}\lambda_{m,l,2}^{1/2}=|g_{l}(y_{m,l})|^{1/2}$. For $y=(y_{1},y_{2},0)\in F_{\varsigma}^{m,l}\times\{0\}$, we define the function
\begin{equation}
v(y):=
|g_{l}(y)|^{1/2}\widetilde\chi_{l}^{2}(y) |\widetilde{\bf B}(y)|^{2}
E\big(\widetilde\theta(y), | \widetilde {\bf B}(y)|^{-1} \Lambda\big)
\end{equation}
We choose the points $y_{m,l}\in F_{\varsigma}^{m,l}\times\{0\}$ so that
\[
\max_{y\in F^{m,l}_{\varsigma}\times\{0\}} v(y)= v(y_{m,l}).
\]
Then the right-hand side of \eqref{comp-tr-gam-1-2} is an upper Riemann sum. We thus get
\begin{multline}\label{RS}\sum_{m\in \mathcal{I}_{l}}|g_{l}(y_{m,l})|^{1/2} \widetilde\chi_{l}^{2}(y_{m,l}) b_{m,l}^{2}E\big(\theta_{m,l},b_{m,l}^{-1} \Lambda\big)\varsigma^{2} = \sum_{m\in \mathcal{I}_{l}} v(y_{m,l})\varsigma^{2}           \geq \\
\int_{(-\delta_{l},\delta_{l})^{2}}v(y_{1},y_{2},0)dy_{1}dy_{2}=\int_{\partial\Omega} \chi_{l}^{2}(x) |{\bf B}(x)|^{2}E\big(\theta(x), |{\bf B}(x)|^{-1} \Lambda\big)d\sigma(x).
\end{multline}

Inserting this and \eqref{RS} into \eqref{comp-tr-gam-1-2}, and using \eqref{def-chi-l}, we obtain
  \begin{equation*}
\limsup_{h\rightarrow 0}\big\{-\sum_{j}(e_{j}(h)-\Lambda h)_{-} \big\}
  \leq 
-\int_{\partial\Omega}|{\bf B}(x)|^{2}E\big(\theta(x), \Lambda |{\bf B}(x)|^{-1}\big)d\sigma(x),
\end{equation*}
which is the desired upper bound.
\subsection{Proof of Corollary \ref{Cor-nb}}

Let us start by computing the left- and right- derivatives of the
function $(0,1)\ni \lambda\to E(\theta,\lambda)$. Using the formula
of $E(\theta,\lambda)$ given in Theorem \ref{main-thm}, we find
\begin{equation}\label{right-der}
\dfrac{\partial E}{\partial\lambda_{+}}(\theta,\lambda)=\left\{\begin{array}{ll}
\dfrac{1}{2\pi^{2}}\displaystyle\int^{\infty}_{0}(\mu_{1}(\xi)-\lambda)_{-}^{1/2}d\xi &{\rm if}\quad\theta=0,\\
\dfrac{\sin(\theta)}{2\pi}  {\rm card}{\{j\,:\,\zeta_{j}(\theta)\leq \lambda\}} 
&{\rm if}\quad\theta\in(0,\pi/2],
\end{array}\right.
\end{equation}
and
\begin{equation}
\dfrac{\partial E}{\partial\lambda_{-}}(\theta,\lambda)=\left\{\begin{array}{ll}
\dfrac{1}{2\pi^{2}}\displaystyle\int^{\infty}_{0}(\mu_{1}(\xi)-\lambda)_{-}^{1/2}d\xi &{\rm if}\quad\theta=0,\\
\dfrac{\sin(\theta)}{2\pi}  {\rm card}{\{j\,:\,\zeta_{j}(\theta)< \lambda\}} &{\rm if}\quad\theta\in(0,\pi/2].
\end{array}\right.
\end{equation}
Notice that the condition in \eqref{asp-nb} ensures the equality of
$\dfrac{\partial E}{\partial\lambda_{+}}(\theta,\lambda)$ and
$\dfrac{\partial E}{\partial\lambda_{-}}(\theta,\lambda)$ when
$\lambda\in\{ \Lambda|\mathbf B(x)|^{-1}~:~x\in\partial\Omega\}$.

Let $\varepsilon>0$. 
Using Corollary~\ref{NB-vs-EN}, we obtain 
\begin{equation}{\label{First-eq}}
{\rm Tr}\,(\mathcal{P}_{h}-(\Lambda+\varepsilon)h)_{-} -{\rm Tr}\,(\mathcal{P}_{h}-\Lambda h)_{-} \geq \varepsilon h \mathcal{N}(\Lambda h;\mathcal{P}_{h},\Omega).
\end{equation}
On the other hand, by the formula in \eqref{main-for}, we have
\begin{multline*}
{\rm Tr}\,(\mathcal{P}_{h}-(\Lambda+\varepsilon)h)_{-} -{\rm Tr}\,(\mathcal{P}_{h}-\Lambda h)_{-}\\
 =\int_{\partial\Omega}|{\bf B}(x)|^{2}\Big( E(\theta(x),(\Lambda+\varepsilon) |{\bf B}(x)|^{-1})-E(\theta(x),\Lambda |{\bf B}(x)|^{-1}) \Big)d\sigma(x)+o(1),\quad {\rm as}~h\rightarrow 0.
\end{multline*}
Implementing this into \eqref{First-eq}, then taking $\limsup_{h\rightarrow 0_{+}}$, we see that
\[
 \limsup_{h\rightarrow 0_{+}}h  \mathcal{N}(\Lambda h;\mathcal{P}_{h},\Omega)\leq \int_{\partial\Omega}|{\bf B}(x)|\dfrac{ E(\theta(x),(\Lambda+\varepsilon) |{\bf B}(x)|^{-1})-E(\theta(x),\Lambda |{\bf B}(x)|^{-1})}{\varepsilon|{\bf B}(x)|^{-1}} d\sigma(x).
\]
We recall here that $|{\bf B}(x)|>0$ for all $x\in\partial\Omega$. Taking the limit $\varepsilon\rightarrow 0_{+}$, we deduce using \eqref{right-der}, and Lebesgue's dominated convergence Theorem, that
\begin{equation}\label{u-b-n}
\limsup_{h\rightarrow 0_{+}}  h\mathcal{N}(\Lambda h;\mathcal{P}_{h},\Omega)\leq \int_{\partial\Omega}|{\bf B}(x)|\dfrac{\partial E}{\partial\lambda_{+}} (\theta(x),\Lambda |{\bf B}(x)|^{-1})d\sigma(x).
\end{equation}
Replacing $\varepsilon$ by $-\varepsilon$ in \eqref{First-eq} and following the same arguments that led to \eqref{u-b-n}, we find
 \begin{equation}\label{l-b-n}
 \liminf_{h\rightarrow 0_{+}} h \mathcal{N}(\Lambda h;\mathcal{P}_{h},\Omega)\geq \int_{\partial\Omega}|{\bf B}(x)|\dfrac{\partial E}{\partial\lambda_{-}} (\theta(x),\Lambda |{\bf B}(x)|^{-1})d\sigma(x).
\end{equation}
It follows by the assumption \eqref{asp-nb} that
\begin{equation}
\int_{\partial\Omega}|{\bf B}(x)|\dfrac{\partial E}{\partial\lambda_{+}} (\theta(x),\Lambda |{\bf B}(x)|^{-1})d\sigma(x)= \int_{\partial\Omega}|{\bf B}(x)|\dfrac{\partial E}{\partial\lambda_{-}} (\theta(x),\Lambda |{\bf B}(x)|^{-1})d\sigma(x).
\end{equation}
Combining \eqref{u-b-n} and \eqref{l-b-n} we obtain
\begin{equation}\label{F-eq-nb}
\lim_{h\rightarrow 0_{+}}  h\mathcal{N}(\Lambda h;\mathcal{P}_{h},\Omega)= \int_{\partial\Omega}|{\bf B}(x)|\dfrac{\partial E}{\partial\lambda_{+}} (\theta(x),\Lambda |{\bf B}(x)|^{-1})d\sigma(x),
\end{equation}
which finishes the proof.

\section*{acknowledgements}
This paper is a major part of the author's Ph.D. dissertation. The author
wishes to thank her advisors S. Fournais and A. Kachmar.
Financial support through the Lebanese University
and CNRS as well as through the grant of S. Fournais from Lundbeck foundation.

\begin{appendix}
\section{Proof of Lemma~\ref{Lem-App}}
Using \eqref{new-QF}, \eqref{app-jac} and \eqref{app-mf}, we obtain that for some constant $c_{1}>0$ 
\begin{multline}\label{Eq-app-qff}
 (1-c_{1}(\ell+T)) \bigg\{\Int{Q_{0,\ell,T}}{}
     \Sum{p,q=1}{3}g^{pq}(y_{0})\big[(-ih\nb_{y_{p}}+\widetilde{\bf A}_{p})\,\widetilde{u}\big]\,\big[\overline{(-ih\nb_{y_{q}}+\widetilde{\bf A}_{q})\,\widetilde{u}}\big]\,|g(y_{0})|^{1/2}dy\bigg\}\\
\leq \mathcal{Q}_{h}(u)
     \leq  (1+c_{1}(\ell+T)) \bigg\{\Int{Q_{0,\ell,T}}{}
     \Sum{p,q=1}{3}g^{pq}(y_{0})\big[(-ih\nb_{y_{p}}+\widetilde{\bf A}_{p})\,\widetilde{u}\big]\,\big[\overline{(-ih\nb_{y_{q}}+\widetilde{\bf A}_{q})\,\widetilde{u}}\big]\,|g(y_{0})|^{1/2}dy\bigg\}.
\end{multline}
Similarly, using \eqref{app-mf} and \eqref{new-norm}, we have for
some constant $c_{2}>0$
\begin{multline}\label{Eq-app-norm}
 (1-c_{2}(\ell+T))\int_{Q_{0,\ell,T}}\!\!\!\!|g(y_{0})|^{1/2}|\widetilde u|^{2}dy
\leq \norm{u}_{L^{2}(\mathcal{V}_{x_{0}})}^{2}
\leq
(1+c_{2}(\ell+T))\int_{Q_{0,\ell,T}}\!\!\!\!|g(y_{0})|^{1/2}|\widetilde
u|^{2}dy.
\end{multline}
By the Cauchy-Schwarz inequality, we get using \eqref{app-MF} that
there exists a constant $c_{3}>0$ such that
\begin{multline}\label{Eq-app-qf-1}
    (1-\varepsilon) \Int{\R^{3}_{+}}{}
     \Sum{p,q=1}{3}g^{pq}(y_{0})\big[(-ih\nb_{y_{p}}+\widetilde{\bf A}^{\rm lin}_{p})\,\widetilde{u}\big]\,\big[\overline{(-ih\nb_{y_{q}}+\widetilde{\bf A}^{\rm lin}_{q})\,\widetilde{u}}\big]\,|g(y_{0})|^{1/2}dy \\
     -  c_{3}(\ell^{2}+T^{2})^{2} \varepsilon^{-1}\Int{\R^{3}_{+}}{}
|\widetilde{u}|^{2}\,|g(y_{0})|^{1/2}dy\\
\leq \Int{\R^{3}_{+}}{}
     \Sum{p,q=1}{3}g^{pq}(y_{0})\big[(-ih\nb_{y_{p}}+\widetilde{\bf A}_{p})\,\widetilde{u}\big]\,\big[\overline{(-ih\nb_{y_{q}}+\widetilde{\bf A}_{q})\,\widetilde{u}}\big]\,|g(y_{0})|^{1/2}dy\\
     \leq
     (1+\varepsilon) \Int{\R^{3}_{+}}{}
     \Sum{p,q=1}{3}g^{pq}(y_{0})\big[(-ih\nb_{y_{p}}+\widetilde{\bf A}^{\rm lin}_{p})\,\widetilde{u}\big]\,\big[\overline{(-ih\nb_{y_{q}}+\widetilde{\bf A}^{\rm lin}_{q})\,\widetilde{u}}\big]\,|g(y_{0})|^{1/2}dy\\
     + c_{3}(\ell^{2}+T^{2})^{2} \varepsilon^{-1}\Int{\R^{3}_{+}}{}
|\widetilde{u}|^{2}\,|g(y_{0})|^{1/2}dy.
\end{multline}
for any $\varepsilon>0$.
Next, we perform the change of variables \(z=(z_{1},z_{2},z_{3})=
\Big({{\lambda^{1/2}_{1}}}y_{1},{{\lambda^{1/2}_{2}}}y_{2},y_{3}\Big)\).
We thus infer using \eqref{sm} the following quadratic form in the
$(z_{1},z_{2},z_{3})$ variables
\begin{multline}\label{Eq-ncv-qf}
 \Int{\R^{3}_{+}}{}
     \Sum{p,q=1}{3}g^{pq}(y_{0})\big[(-ih\nb_{y_{p}}+\widetilde{\bf A}^{\rm lin}_{p})\,\widetilde{u}\big]\,\big[\overline{(-ih\nb_{y_{q}}+\widetilde{\bf A}^{\rm lin}_{q})\,\widetilde{u}}\big]\,|g(y_{0})|^{1/2}dy\\
     =\sum_{p=1}^{3}\Int{\widetilde{Q}_{0,\ell,T}}{}|(-ih\nb_{z_{p}}+{\bf F}_{p})\,\breve{{u}}|^{2}\,dz,
\end{multline}
where ${\bf F}=({\bf F}_{1}, {\bf F}_{2},{\bf F}_{3})$ is the
magnetic potential given by
\begin{equation*}
{\bf F}_{1}(z)=\lambda_{1}^{-1/2}\breve{ {\bf A}}^{\rm lin}_{1}(z),\quad
{\bf F}_{2}(z)=\lambda_{2}^{-1/2}\breve{ {\bf A}}^{\rm lin}_{2}(z)\quad
{\bf F}_{3}(z)=\breve{{\bf A}}^{\rm lin}_{3}(z).
\end{equation*}
Also, we have
\begin{equation}\label{Eq-ncv}
\Int{Q_{0,\ell,T}}{}
|\widetilde{u}|^{2}\,|g(y_{0})|^{1/2}dy= \Int{\widetilde{Q}_{0,\ell,T}}{}
|\breve{{u}}|^{2}\,dz
\end{equation}
Substituting this into \eqref{Eq-app-qf-1} yields
\begin{equation}\label{Eq-ncv1}
  (1-c_{2}(\ell+T)) \Int{\widetilde{Q}_{0,\ell,T}}{}
|\breve{{u}}|^{2}\,dz
\leq \norm{u}_{L^{2}(\mathcal{V}_{x_{0}})}^{2}
\leq (1+c_{2}(\ell+T))   \Int{\widetilde{Q}_{0,\ell,T}}{}
|\breve{{u}}|^{2}\,dz.
\end{equation}
Let $\beta=(\beta_{1},\beta_{2},\beta_{3})={\rm curl}_{z}({\bf
F}(z)) $ and note that the coefficients of $\beta$ and $\alpha$ (see
\eqref{def-alphas}) are related by
\[
\beta_{1}= {\lambda^{-1/2}_{2}}\alpha_{1},\qquad \beta_{2}= {\lambda^{-1/2}_{1}}\alpha_{2},\qquad \beta_{3}=(
\lambda_{1}\lambda_{2})^{-1/2}\alpha_{3}.
\]
The relation \eqref{relation-nc} gives that
\begin{equation}\label{mf-nc}
|\beta|=|{\rm curl}_{z}({\bf F}(z)) |=(\beta_{1}^{2}+\beta_{2}^{2}+\beta_{3}^{2})^{1/2} = |\widetilde {\bf B}|(y_{0})
\end{equation}
We thus perform a gauge transformation so that there exists a
function $\phi_{0}\in C^{\infty}(\widetilde Q_{0,\ell,T})$ such that
\begin{equation}\label{gT}
{\bf F}(z)=b_{0}{\bf F}_{\theta_{0}}(z)+\nabla \phi_{0},\qquad b_{0}=|\widetilde {\bf B}(y_{0})|,
\end{equation}
where, for $\theta\in[0,\pi/2]$, ${\bf F}_{\theta}$ is the magnetic field from \eqref{Ftheta} and 
\begin{equation}\label{theta-ang-0}
{\theta_{0}}:=\widetilde\theta(y_{0})=\arcsin\left(\dfrac{|\beta_{3}|}{|\beta|}\right).
\end{equation}
We emphasize here that \eqref{theta-ang-0} is compatible with the
definition of $\theta(x)$ given in \eqref{def-theta}, i.e.,
$\widetilde\theta(y_{0})=\theta(\Phi_{x_{0}}^{-1}(y_{0}))$.
Combining \eqref{Eq-app-qf-1}, \eqref{Eq-ncv-qf} and \eqref{gT}, we
obtain, using \eqref{Eq-ncv},
\begin{multline}\label{Eq-app-qf-2}
    (1-\varepsilon) \Int{\widetilde{Q}_{0,\ell,T}}{}|(-ih\nb_{z}+b_{0}{\bf F}_{\theta_{0}})e^{i\phi_{0}/h}\,\breve{{u}}|^{2}\,dz
     -  c_{3}(\ell^{2}+T^{2})^{2}\Int{\widetilde{Q}_{0,\ell,T}}{}|\breve{{u}}|^{2}dz\\
\leq \Int{\R^{3}_{+}}{}
     \Sum{p,q=1}{3}g^{pq}(y_{0})\big[(-ih\nb_{y_{p}}+\widetilde{\bf A}_{p})\,\widetilde{u}\big]\,\big[\overline{(-ih\nb_{y_{q}}+\widetilde{\bf A}_{q})\,\widetilde{u}}\big]\,|g(y_{0})|^{1/2}dy\\
     \leq
     (1+\varepsilon)\Int{\widetilde{Q}_{0,\ell,T}}{}|(-ih\nb_{z}+b_{0}{\bf F}_{\theta_{0}})e^{i\phi_{0}/h}\,\breve{{u}}|^{2}\,dz
     + c_{3}(\ell^{2}+T^{2})^{2} \varepsilon^{-1}\Int{\widetilde{Q}_{0,\ell,T}}{}
|\breve{{u}}|^{2}\,dz.
\end{multline}
for any $\varepsilon>0$.
Choose $\varepsilon\geq \ell+T$. Inserting \eqref{Eq-app-qf-2} into
\eqref{Eq-app-qff}, we obtain that for some constant $c_{4}>0$
\begin{multline}\label{Eq-app-qf}
   (1-c_{4}\varepsilon) \Int{\widetilde{Q}_{0,\ell,T}}{}|(-ih\nb_{z}+b_{0}{\bf F}_{\theta_{0}})e^{i\phi_{0}/h}\,\breve{{u}}|^{2}\,dz
     -  c_{4}(\ell^{2}+T^{2})^{2} \varepsilon^{-1}\Int{\widetilde{Q}_{0,\ell,T}}{}|\breve{{u}}|^{2}dz\\
\leq \mathcal{Q}_{h}(u)
     \leq (1+c_{4}\varepsilon)\Int{\widetilde{Q}_{0,\ell,T}}{}|(-ih\nb_{z}+b_{0}{\bf F}_{\theta_{0}})e^{i\phi_{0}/h}\,\breve{{u}}|^{2}\,dz
     + c_{4}(\ell^{2}+T^{2})^{2} \varepsilon^{-1}\Int{\widetilde{Q}_{0,\ell,T}}{}
|\breve{{u}}|^{2}\,dz.
\end{multline}
Recall \eqref{Eq-ncv1} and choose $C=\max\{c_{2},c_{4}\}$, thereby
establishing \eqref{Ap-QF} and \eqref{Ap-norm}.
\section{Proof of Lemma~\ref{lem:app'}}
According to Lemma~\ref{neg-bulk}, the lemma follows if we can prove
a lower bound on the right-hand side of \eqref{Res-Bnd}. We start by
estimating $\mathcal{Q}_{h}(\psi_{1}f_{j})$. Using the IMS
decomposition formula, it follows that
\begin{align}\label{ims-bnd}
\mathcal{Q}_{h}(\psi_{1}f_{j})= \Sum{l\in {J}}{}\left(\mathcal{Q}_{h}(\chi_{l}\psi_{1}f_{j})-h^{2}\norm{|\nb \chi_{l}|\psi_{1}f_{j}}_{L^{2}(\Omega)}^{2}\right).
\end{align}
Using \eqref{bn-gr}, and implementing \eqref{def-chi-l}, we get
\begin{multline}\label{eq-ims-chi-l}
\mathcal{Q}_{h}(\psi_{1}f_{j})-(\Lambda h+C_{1}h^{2}\varsigma^{-2})\norm{\psi_{1}f_{j}}_{L^{2}(\Omega)}^{2}\\
\geq \Sum{l\in
J}{}\left(\mathcal{Q}_{h}(\psi_{1}\chi_{l}f_{j})-\big(\Lambda h+
(C_{1}+C_{2}) h^{2}\varsigma^{-2}\big)
\norm{\psi_{1}\chi_{l}f_{j}}_{L^{2}(\Omega)}^{2}\right),
\end{multline}
where we used that $\varsigma^{-2}\gg 1$ (see \eqref{choice} below).

Applying the IMS formula once again, we then find, using that $a\ll
1$,
\begin{multline}\label{Qf-ujl}
\mathcal{Q}_{h}(\psi_{1}\chi_{l}f_{j})
    = \sum_{m\in\mathcal{I}_{l}} \Big\{  \mathcal{Q}_{h}(\varphi_{m,l}\psi_{1}\chi_{l}f_{j})                               -h^{2}\norm{|\nabla\varphi_{m,l}|\psi_{1}\chi_{l}f_{j}}_{L^{2}(\Omega)}^{2}\Big\}\\
    \geq \sum_{m\in\mathcal{I}_{l}} \Big\{  \mathcal{Q}_{h}(\varphi_{m,l}\psi_{1}\chi_{l}f_{j})
     -
     (C_{1}+C_{2}+C_{3}^{\prime})h^{2}(a\varsigma)^{-2}\norm{\varphi_{m,l}\psi_{1}\chi_{l}f_{j}}_{L^{2}(\Omega)}^{2}\Big\}\,.
\end{multline}
The last inequality follows from \eqref{bnd-grad-phi} and $C^{\prime}_{3}:=C_{3}\sup_{l\in J} \|D\Phi_{l}\|^{2}$. 
Inserting this into \eqref{eq-ims-chi-l}, it follows that
\begin{multline}\label{eq-ims-varphi-m-l}
\mathcal{Q}_{h}(\psi_{1}f_{j})-(\Lambda h+C_{1}h^{2}\varsigma^{-2})\norm{\psi_{1}f_{j}}_{L^{2}(\Omega)}^{2}\\
\geq \Sum{l\in
J}{}\sum_{m\in\mathcal{I}_{l}}\left(\mathcal{Q}_{h}(\varphi_{m,l}\psi_{1}\chi_{l}f_{j})-\big(\Lambda
h+ (C_{1}+C_{2}+C_{3}^{\prime}) h^{2}(a\varsigma)^{-2}\big)
\norm{\varphi_{m,l}\psi_{1}\chi_{l}f_{j}}_{L^{2}(\Omega)}^{2}\right)\,.
\end{multline}
Applying Lemma~\ref{Lem-App} with $y_{0}$ replaced by $y_{m,l}$,
$u=\varphi_{m,l}\psi_{1}\chi_{l}f_{j}$, $\ell=(1+a)\varsigma$,
$T=\varsigma$, we then deduce that there exists a function
$\phi_{m,l}:=\phi_{y_{m,l}}\in
C^{\infty}(\widetilde{Q}^{m,l}_{(1+a)\varsigma})$ such that, for all
$\varepsilon\in(0,1]$ satisfying $\varepsilon\gg\varsigma$, one has,
using $a\ll 1$,
\begin{multline}\label{eq-app-app}
\mathcal{Q}_{h}(\psi_{1}f_{j})-(\Lambda h+C_{1}h^{2}\varsigma^{-2})\norm{\psi_{1}f_{j}}_{L^{2}(\Omega)}^{2}\\
\geq (1-C\varepsilon) \sum_{l\in J}\sum_{m\in \mathcal{I}_{l}}
 \Int{\widetilde{Q}^{m,l}_{(1+a)\varsigma}}{}|(-ih\nb_{z}+b_{m,l}{\bf F}_{\theta_{m,l}})\,e^{i\phi_{m,l}/h}\breve{\varphi} _{m,l}\breve{\psi}_{1}\breve{\chi}_{l}\breve{f}_{j}|^{2}                 dz
   \\ -\big((\Lambda h+ (C_{1}+C_{2}+C_{3}^{\prime})h^{2}(a\varsigma)^{-2})(1+3C\varsigma)+ 25C\varsigma^{4}\varepsilon^{-1}\big)
    \sum_{l\in J}\sum_{m\in \mathcal{I}_{l}} \Int{{\widetilde{Q}^{m,l}_{(1+a)\varsigma}}}{}
|\breve{\varphi
}_{m,l}\breve{\psi}_{1}\breve{\chi}_{l}\breve{f}_{j}|^{2}\,dz ,
\end{multline}
where $C$ is the constant from Lemma~\ref{Lem-App}. Put $\widetilde
C=\max\big\{C_{1}+C_{2}+C^{\prime}_{3},25 C\big\}$. Inserting
\eqref{eq-app-app} into \eqref{Res-Bnd} yields the desired estimate
of the lemma.
\end{appendix}


    \newpage
\end{document}